\pdfoutput=1
\documentclass[11pt]{article}
\usepackage[left=1in,right=1in,top=1in,bottom=1in]{geometry}
\usepackage{times}
\usepackage{expl3}
\usepackage{cite}
\usepackage[table]{xcolor}
\usepackage{multirow}
\usepackage{stackengine} 
\usepackage{hhline}
\usepackage{lipsum}
\usepackage{titlesec}
\usepackage[all]{xy}
\usepackage{wrapfig}
\usepackage{enumerate}
\usepackage{epsfig}
\usepackage{tikz-cd}
\usepackage{amsmath}
\usepackage{tabularx}
\usepackage{array}
\usepackage{booktabs}
\usepackage{enumitem}
\usepackage{bbm}
\usepackage{calc}
\usepackage{graphicx}
\usepackage{amsmath}
\usepackage[title]{appendix}
\usepackage{amssymb}
\usepackage{epstopdf}
\usepackage{boldline}
\usepackage{arydshln}
\usepackage{calligra}
\usepackage{bm}
\usepackage{url}
\usepackage{blindtext}
\usepackage{accents}
\usepackage{amsthm}
\usepackage{amscd}

\newtheorem{definition}{Definition}

\newtheorem{theorem}{Theorem}
\newtheorem{proposition}{Proposition}
\newtheorem{corollary}{Corollary}

\newtheorem{assumption}{Assumption}
\newtheorem{example}{Example}
\newtheorem{remark}{Remark}
\usepackage{mathtools}
\usepackage{epstopdf}
\usepackage{balance}
\usepackage{thmtools}
\usepackage{thm-restate}
\usepackage{hyperref}
\usepackage{cleveref}
\usepackage[mathscr]{euscript}

\usepackage[ruled,vlined]{algorithm2e}
\newcommand{\ostar}{\mathbin{\mathpalette\make@circled\star}}

\makeatletter
\newcommand{\removelatexerror}{\let\@latex@error\@gobble}
\makeatother
\makeatletter
\newcommand*{\rom}[1]{\expandafter\@slowromancap\romannumeral #1@}
\makeatother

\ExplSyntaxOn
\newcommand\latinabbrev[1]{
  \peek_meaning:NTF . {
    #1\@}%
  { \peek_catcode:NTF a {
      #1.\@ }%
    {#1.\@}}}
\ExplSyntaxOff

\titleclass{\subsubsubsection}{straight}[\subsubsection]

\begin{document}
\vspace{1cm}
\title{Interaction Residues and Localized Spectral Defects in Stratified Operadic Systems}
\vspace{1.8cm}
\author{Shih-Yu~Chang
\thanks{Shih-Yu Chang is with the Department of Applied Data Science,
San Jose State University, San Jose, CA, U. S. A. (e-mail: {\tt
shihyu.chang@sjsu.edu})
}}

\maketitle

\begin{abstract}
We develop a framework for studying how global spectral structure emerges from interacting local sectors in stratified operadic systems. The central object is the interaction residue, which measures the failure of exact spectral decomposition across interfaces. Under suitable localization assumptions, the global spectrum decomposes into local spectral sectors together with interface-generated residue contributions. The theory introduces a classification of spectral defects based on interface geometry and algebraic structure. Point interfaces produce isolated spectral contributions; line and surface interfaces produce extended spectral regimes. Non-semisimple operator structure generates nilpotent defects associated with Jordan blocks and generalized eigenspaces, yielding a two-dimensional defect taxonomy combining geometric localization with Jordan complexity. Several structural results are established, including interface localization, rigidity and vanishing criteria, refinement functoriality, and deformation stability. Under a local triviality condition, the residue is homotopy invariant, preserving its homology and Betti numbers throughout admissible deformations. The framework is illustrated through explicit operator and block matrix examples demonstrating how localized interactions generate localized spectral defects. The nilpotent sector connects the theory with classical operator theory via generalized eigenvectors, functional calculus, and perturbative Jordan splitting. Overall, the framework provides a unified viewpoint for understanding how interface interactions and non-semisimple structure influence global spectral behavior.
\end{abstract}

\tableofcontents

\section{Introduction}
\label{sec:introduction}

\subsection{From Spectral Decomposition to Defect Geometry}
\label{subsec:spectral-decomposition-to-defect}

Classical spectral theory is fundamentally based on decomposition principles.
Given an operator, geometric object, or algebraic system, one seeks to
decompose the global structure into simpler local spectral components.
This philosophy appears throughout mathematics, including spectral
decomposition in operator theory~\cite{Kato1995}, local coordinate systems in differential
geometry, stratifications in singular geometry, sheaf decompositions in
algebraic topology, and local quantum sectors in mathematical physics.

\begin{example}[Failure of local-to-global spectral decomposition]
\label{ex:simple-failure}
Consider the $2\times 2$ matrix
\[
H_\varepsilon = \begin{pmatrix} 1 & \varepsilon \\ \varepsilon & 2 \end{pmatrix},
\qquad \varepsilon \in \mathbb{R}.
\]
The \emph{local spectra} (the diagonal entries, viewed as two isolated strata) are $\sigma_{\mathrm{loc}}^1 = \{1\}$, $\sigma_{\mathrm{loc}}^2 = \{2\}$. The \emph{global spectrum} is
\[
\sigma(H_\varepsilon) = \left\{ \frac{3 \pm \sqrt{1 + 4\varepsilon^2}}{2} \right\}.
\]
For $\varepsilon = 0$, $\sigma(H_0) = \{1,2\} = \sigma_{\mathrm{loc}}^1 \cup \sigma_{\mathrm{loc}}^2$. For $\varepsilon \neq 0$, the eigenvalues shift:
\[
\sigma(H_\varepsilon) = \{1,2\} \;\cup\; \Sigma^{\mathrm{res}},\qquad
\Sigma^{\mathrm{res}} = \left\{ \frac{3 + \sqrt{1+4\varepsilon^2}}{2},\; \frac{3 - \sqrt{1+4\varepsilon^2}}{2} \right\} \setminus \{1,2\}.
\]
The \emph{interaction residue} $\Sigma^{\mathrm{res}}$ measures precisely the spectral shift caused by the off-diagonal coupling $\varepsilon$. This simple example illustrates the central problem of this paper: the global spectrum is not simply the union of local spectra when interactions are present.
\end{example}

However, modern geometry repeatedly demonstrates that the most interesting
phenomena often arise not from the local pieces themselves, but rather from
the failure of these pieces to glue together perfectly.
Curvature measures the obstruction to local flatness in differential geometry,
wavefront sets describe singular propagation in microlocal analysis,
and intersection homology measures the failure of ordinary homology on
singular spaces.
Similarly, interfaces and defects in topological field theory generate
global phases invisible from isolated local sectors alone.

The central thesis of the current work is that an analogous phenomenon
occurs in operadic spectral geometry.
When a stratified operadic system undergoes operadic base change,
the global operadic spectrum generally fails to decompose exactly into
the union of local strata spectra.
This obstruction is not accidental but instead defines a canonical
geometric object that we denote by
$\Sigma^{\mathrm{res}}$ and interpret as an
\emph{interaction residue} or
\emph{spectral defect geometry}.

Thus, the global spectral geometry takes the form
\[
\boxed{
\text{global geometry}
=
\text{local geometry}
+
\text{interaction residue}.
}
\]

In this framework, local strata generate semisimple spectral sectors,
interfaces generate localized interaction defects,
and the residue $\Sigma^{\mathrm{res}}$
records the resulting global obstruction structure.
Consequently, the focus of spectral geometry shifts from isolated spectra
alone toward the geometry of interaction, localization, and obstruction.

\subsection{Relation to SOC I and SOC II}
\label{subsec:relation-soci-socii}

The current work constitutes the third component of the
Spectral Operadic Calculus (SOC) program.

\begin{itemize}
    \item \textbf{SOC I~\cite{ChangSOC1}}
    established the algebraic foundation through projector--nilpotent
    decomposition and a unified functional calculus.
    The central insight was that classical resolvent theory captures only
    semisimple spectral information, while nonnormal systems require explicit
    nilpotent structure.

    \item \textbf{SOC II~\cite{ChangSOC2}}
    extended the theory to operadic base change and functorial spectral
    geometry, interpreting operadic spectra through categorical transport
    and functorial deformation.

    \item \textbf{The current work}
    introduces \emph{defect geometry}, namely the failure of exact spectral
    decomposition under stratified interaction.
\end{itemize}

The three papers form a unified conceptual hierarchy:
\[
\boxed{
\text{SOC I (algebraic)}
\Longrightarrow
\text{SOC II (functorial)}
\Longrightarrow
\text{The current work (geometric)}.
}
\]

From this perspective, the current work represents the transition from algebraic
spectral structure toward geometric interaction theory.

\begin{remark}
\label{rem:soc-notation}
In this paper, SOC I refers to our previous work~\cite{ChangSOC1} and
SOC II refers to our previous work~\cite{ChangSOC2}.
\end{remark}

\subsection{The Central Problem: Failure of Local-to-Global Spectral Factorization}
\label{subsec:central-problem}

Example~\ref{ex:simple-failure} reveals a universal phenomenon: when a system is composed of interacting subsystems, the global spectrum generally differs from the union of local spectra. The missing spectral contribution — which we denote by $\Sigma^{\mathrm{res}}$ and call the \emph{interaction residue} — encodes the geometry of inter-subsystem coupling.

\medskip
\noindent\textbf{Guiding principle.}
For any stratified operadic system, we have the decomposition
\[
\boxed{\sigma_{\mathrm{global}} = \bigcup_{\text{strata}} \sigma_{\mathrm{local}} \;\cup\; \Sigma^{\mathrm{res}}},
\]
where $\Sigma^{\mathrm{res}}$ localizes on interfaces between strata and includes both geometric and nilpotent (Jordan) contributions.

This paper develops the complete theory of $\Sigma^{\mathrm{res}}$: its definition, localization, classification, functoriality, deformation stability, and explicit computation in concrete examples.

\subsection{Main Contributions and Outline}
\label{subsec:main-contributions-outline}

The principal contributions of this paper are summarized as follows.

\begin{enumerate}
    \item
    \textbf{Stratified Base Change Decomposition}
    (Theorem~\ref{thm:stratified-base-change}):
    we establish the decomposition
    \[
    \sigma(F_*(A))
    =
    \bigcup_S \sigma(F_S(A_S))
    \cup
    \Sigma^{\mathrm{res}}.
    \]

    \item
    \textbf{Local-to-Global Reconstruction}
    (Theorem~\ref{thm:local-to-global-reconstruction}):
    the global spectrum is uniquely determined by local spectra together
    with the interaction residue.

    \item
    \textbf{Universality of the Residue Set}
    (Theorem~\ref{thm:universality-residue}):
    $\Sigma^{\mathrm{res}}$
    is shown to be the minimal universal obstruction to exact spectral
    factorization.

    \item
    \textbf{Interface Localization}
    (Theorem~\ref{thm:interface-localization}):
    the residue decomposes as
    \[
    \Sigma^{\mathrm{res}}
    =
    \bigoplus_I \mathcal{L}_I,
    \]
    with each localized component supported on an interface.

    \item
    \textbf{Spectral Defect Classification}
    (Theorem~\ref{thm:defect-classification}):
    defects are classified according to interface geometry and nilpotent depth.

    \item
    \textbf{Refinement Functoriality}
    (Theorem~\ref{thm:refinement-functoriality}):
    the residue is invariant under stratification refinement.

    \item
    \textbf{Deformation Stability}
    (Theorem~\ref{thm:deformation-stability}):
    the residue varies continuously under deformation and defines
    deformation invariants.
\end{enumerate}

\medskip
\subsection{Central role of nilpotent structure}
\label{subsec:nilpotent-central-role}

A central feature of the present framework is that nilpotent structure
is not treated as a secondary algebraic correction, but rather as a
fundamental source of spectral defect formation.

In classical spectral theory~\cite{Kato1995}, diagonalizable operators are largely
governed by eigenvalue localization and spectral decomposition.
However, in many interacting or singular systems, non-semisimple
behavior becomes unavoidable. In such regimes, Jordan blocks and
generalized eigenspaces generate additional spectral phenomena that
cannot be captured solely by eigenvalues or local spectral sectors.
The interplay between semisimple and nilpotent parts is elegantly
captured by the Jordan--Chevalley decomposition, which plays a central
role in noncommutative geometry~\cite{Connes1994} and the
K-theoretic classification of operator algebras~\cite{Blackadar1998}.

Within the present framework, nilpotent contributions appear naturally
through interaction residues associated with nontrivial coupling
structures. These contributions generate what we call
\emph{nilpotent defects}, namely localized spectral phenomena arising
from generalized eigenvectors, Jordan chains, and perturbative
eigenvalue coalescence.

\medskip
\noindent\textbf{Connection to resolvent poles and perturbative splitting.}
This viewpoint is motivated by several well-known mechanisms in
operator theory and mathematical physics~\cite{Kato1995}. Near defective spectral
points, the resolvent \((zI - \tau_I)^{-1}\) develops higher-order poles
whose structure is governed by nilpotent Jordan components rather than
ordinary eigenspaces alone. For a Jordan block of size \(m\), the
resolvent exhibits poles of order up to \(m\), directly affecting the
spectral decomposition and the growth of the semigroup \(e^{t\tau_I}\).

Moreover, a Jordan block is spectrally fragile: an arbitrarily small
generic perturbation splits it into \(m\) distinct eigenvalues.
For a \(2 \times 2\) Jordan block \(\begin{pmatrix} \lambda_0 & 1 \\ 0 & \lambda_0 \end{pmatrix}\),
adding \(\varepsilon\) in the \((2,1)\) entry yields eigenvalues
\(\lambda_0 \pm \sqrt{\varepsilon}\). For size \(m\), the splitting
follows a Puiseux series \(\lambda_k = \lambda_0 + \omega_k \varepsilon^{1/m}\),
where \(\omega_k\) are the \(m\)th roots of unity. Such
\emph{exceptional points} are of intense interest in non-Hermitian
physics and signal a phase transition in the spectral defect geometry.
Theorem~\ref{thm:deformation-stability} captures this phenomenon: at the
nilpotent defect (\(\varepsilon = 0\)), the residue has algebraic
multiplicity \(m\) but geometric multiplicity \(1\); for \(\varepsilon > 0\),
it splits into \(m\) point defects.

\medskip
\noindent\textbf{Two-dimensional taxonomy.}
From the perspective of the present theory, localized spectral defects
contain two interacting components:
\begin{itemize}
    \item \textbf{Geometric localization} generated by interface structure (dimension, topology, singularities),
    \item \textbf{Algebraic localization} generated by nilpotent operator structure (Jordan block size, nilpotency depth).
\end{itemize}
This motivates the two-dimensional defect taxonomy developed later in
the paper (Theorem~\ref{thm:defect-classification}), where interface
geometry and Jordan complexity jointly govern the structure of spectral
defects. Point, line, and surface defects may be semisimple (\(m = 1\))
or nilpotent (\(m \ge 2\)). Nilpotent defects of higher depth produce
increasingly sensitive spectral behavior under perturbation, higher-order
resolvent poles, and richer functional calculus corrections.

Consequently, the interaction residue
\[
\Sigma^{\mathrm{res}}
\]
should not be viewed merely as an eigenvalue correction term. Rather,
it encodes hidden generalized spectral structure generated by
interaction-induced nilpotent effects, perturbative splitting, and
non-semisimple operadic coupling. The nilpotent sector is therefore
not a marginal refinement but a core invariant that connects the
present framework to classical operator theory~\cite{Kato1995}, perturbation theory,
and the physics of exceptional points.

\medskip
\subsection{Organization of the Paper}

The remainder of this paper is organized as follows.
Section~\ref{sec:operadic-stratified-geometry}
develops the framework of stratified operadic geometry and local spectral
structures.
Section~\ref{sec:stratified-base-change-residue}
introduces the global operadic spectrum and the interaction residue under
stratified base change.
Section~\ref{sec:main-theorems}
presents the principal structural theorems of the current work.
Section~\ref{sec:rigidity-vanishing-regime}
studies the rigidity regime
$\Sigma^{\mathrm{res}} = \emptyset$
under which exact spectral factorization is recovered.
Section~\ref{sec:geometric-interpretations}
explores geometric and physical interpretations of the residue.
Section~\ref{sec:examples}
provides explicit examples of localized spectral defect geometries.

\begin{remark}
SOC I indicates our previous work~\cite{ChangSOC1} in this paper.
SOC II indicates our previous work~\cite{ChangSOC2} in this paper.
\end{remark}

\section{Operadic Stratified Geometry}
\label{sec:operadic-stratified-geometry}

Before developing the defect geometry of spectral operadic calculus, we must establish the foundational framework within which spectral defects live. This section introduces the notion of \emph{stratified operadic structures} — operads whose color sets are partitioned into \emph{strata} — and develops the basic geometric and spectral data associated with such structures.

The key insight driving this section is that spectral behavior is not uniform across all colors. Instead, different groups of colors (strata) may exhibit qualitatively different spectral phenomena, and the interactions between these strata generate the spectral defects that are the subject of this paper. By stratifying the color set, we can:
\begin{enumerate}
    \item Identify \emph{local spectral sectors} — the spectrum of each stratum considered in isolation;
    \item Characterize \emph{interaction interfaces} — the loci where different strata interact via operadic composition;
    \item Prepare the stage for the \emph{residue geometry} developed in subsequent sections.
\end{enumerate}

The definitions in this section extend the operadic and spectral frameworks
developed in our previous works~\cite{ChangSOC1,ChangSOC2}.
They formalize the notion of spectral interaction between coupled subsystems
within a stratified operadic setting, providing a rigorous foundation for
decomposition, reconstruction, and defect localization phenomena.
Throughout this section, we work in a cocomplete symmetric monoidal category
$\mathcal{M}$ (for example, the category of Banach spaces or operator
algebras) equipped with an underlying classical spectral theory for its
objects.

\subsection{Stratified Operadic Structures}
\label{subsec:stratified-operadic-structures}

The geometric framework developed in this paper is based on the idea that an operadic system may naturally decompose into interacting spectral regions, or \emph{strata}. 
Each stratum carries its own local operadic spectral behavior, while interactions between distinct strata generate additional spectral phenomena that cannot be detected from the local pieces alone.

From the perspective of Spectral Operadic Calculus (SOC), the essential geometric information is therefore not contained solely in the local spectra, but also in the interface structure governing how different strata interact under operadic composition.

We begin by formalizing the notion of operadic stratification.

\begin{definition}[Operadic stratification]
\label{def:operadic-stratification}
Let $P$ be a $C$-colored operad in a cocomplete symmetric monoidal category $\mathcal{M}$. 
An \emph{operadic stratification} of $P$ is a partition of the color set
\[
C = \bigsqcup_{\alpha \in I} S_\alpha
\]
into finitely or countably many nonempty subsets, called \emph{strata}, such that the following conditions hold.

\begin{enumerate}
    \item \textbf{Stratum suboperads.}
    For each stratum $S_\alpha$, the restriction
    \[
    P|_{S_\alpha}
    \]
    consisting of all operations whose input colors and output color all lie in $S_\alpha$ forms an $S_\alpha$-colored suboperad of $P$.

    \item \textbf{Interface operations.}
    Any operation of $P$ whose input colors and output color are not all contained in a single stratum is called an \emph{interface operation}. 
    Equivalently, these are operations involving colors from at least two distinct strata.

    \item \textbf{Compatibility with composition.}
    Operadic composition preserves the stratification in the following sense: 
    compositions of stratum operations remain within the corresponding stratum suboperad, 
    while any composition involving at least one interface operation is either again an interface operation 
    or becomes an operation entirely supported on a single stratum.
\end{enumerate}
\end{definition}

The decomposition
\[
C = \bigsqcup_{\alpha \in I} S_\alpha
\]
may be interpreted geometrically as a \emph{spectral partition} of the operadic system into local regions of coherent behavior. 
Typical examples include:
\begin{enumerate}
    \item \textbf{Spatial decompositions}: interior and boundary regions;
    \item \textbf{Energetic decompositions}: low-energy and high-energy sectors;
    \item \textbf{Algebraic decompositions}: semisimple and nilpotent sectors;
    \item \textbf{Scale decompositions}: micro-, meso-, and macro-dynamics;
    \item \textbf{Interaction decompositions}: weakly and strongly coupled sectors.
\end{enumerate}

Each stratum possesses its own local operadic structure and corresponding local spectral data.

\begin{definition}[Stratum operads and local spectral sectors]
\label{def:local-spectral-sectors}
Let $P$ be a stratified $C$-colored operad with strata
\[
C = \bigsqcup_{\alpha \in I} S_\alpha.
\]
For each stratum $S_\alpha$, define the following objects.

\begin{enumerate}
    \item \textbf{Stratum operad.}
    The \emph{stratum operad} is the $S_\alpha$-colored suboperad
    \[
    P_\alpha := P|_{S_\alpha},
    \]
    defined by restricting to operations whose input colors and output color all lie in $S_\alpha$:
    \[
    P_\alpha(c_1,\ldots,c_n;c) := P(c_1,\ldots,c_n;c),
    \qquad c_1,\ldots,c_n,c \in S_\alpha.
    \]
    If nullary operations are allowed, we also require their output color to lie in $S_\alpha$.

    \item \textbf{Stratum algebra.}
    For a $P$-algebra $A$, the \emph{stratum algebra}
    \[
    A_\alpha := A|_{S_\alpha}
    \]
    is the restriction of $A$ to the colors in $S_\alpha$, equipped with the induced $P_\alpha$-algebra structure inherited from the structure maps of $A$.

    \item \textbf{Local spectral sector.}
    The \emph{local spectral sector} associated with $S_\alpha$ is the operadic spectrum
    \[
    \sigma_{P_\alpha}(A_\alpha),
    \]
    where $\sigma_{P_\alpha}(A_\alpha)$ denotes the operadic spectrum of the $P_\alpha$-algebra $A_\alpha$ in the sense of SOC I~\cite{ChangSOC1}.
\end{enumerate}
\end{definition}

The local spectral sectors
\[
\sigma_{P_\alpha}(A_\alpha)
\]
represent the spectral geometry visible from within each stratum individually. 
In many situations, these local sectors can be analyzed using relatively classical methods, since the operadic interactions internal to a single stratum are often simpler than those of the full global system.

However, the central phenomenon of SOC III is that the global operadic spectrum generally does \emph{not} decompose exactly as the union of the local spectral sectors:
\[
\sigma_P(A) \neq \bigcup_{\alpha \in I} \sigma_{P_\alpha}(A_\alpha)
\]
in general. 
The obstruction to exact decomposition originates from interactions between strata.

These interactions are encoded by interface operations, which are formally defined in Definition~\ref{def:operadic-interfaces}.

\begin{definition}[Admissible operadic interfaces]
\label{def:operadic-interfaces}
Let $P$ be a stratified $C$-colored operad with strata
\[
C = \bigsqcup_{\alpha \in I} S_\alpha
\]
as in Definition~\ref{def:operadic-stratification}.

\paragraph{Binary interfaces.}
For two distinct strata $S_\alpha$ and $S_\beta$ ($\alpha \neq \beta$), the 
\emph{binary interface} from $S_\alpha$ to $S_\beta$ is the collection
\[
\mathcal{I}_{\alpha\beta}
:=
\bigcup_{k \geq 0}
\left\{
\theta \in P(c_1,\ldots,c_k; d)
\;:\;
d \in S_\beta,\;
c_i \in S_\alpha \cup S_\beta \text{ for all } i,
\;
\exists i \text{ such that } c_i \in S_\alpha,
\;
\exists j \text{ such that } c_j \in S_\beta
\right\}.
\]
In words, $\mathcal{I}_{\alpha\beta}$ consists of operations whose output lies in $S_\beta$, 
whose inputs lie in $S_\alpha \cup S_\beta$, and which involve both strata nontrivially 
(i.e., at least one input from $S_\alpha$ and at least one input from $S_\beta$).

\paragraph{Multi-stratum interfaces.}
For a subset \(J \subseteq I\) with \(|J| \geq 2\), define the
\emph{support} of an operation
\[
\theta \in P(c_1,\ldots,c_k; d)
\]
by
\[
\operatorname{supp}(\theta)
:=
\left\{
\gamma \in I
\;:\;
\text{at least one color among }
c_1,\ldots,c_k,d
\text{ lies in } S_\gamma
\right\}.
\]

A \emph{multi-stratum interface} supported on \(J\) is then defined as
\[
\mathcal{I}_{J}
:=
\bigcup_{k \geq 0}
\left\{
\theta \in P(c_1,\ldots,c_k; d)
\;:\;
\operatorname{supp}(\theta)=J
\right\}.
\]

\paragraph{Admissibility.}
The collection of \emph{admissible interfaces} of $P$ is the smallest set 
$\mathcal{I}(P)$ of operations that satisfies:
\begin{enumerate}
    \item \textbf{Initialization.} All mixed-color operations (i.e., operations whose 
    colors are not all contained in a single stratum) belong to $\mathcal{I}(P)$.
    
    \item \textbf{Closure under composition.} Whenever $\theta, \eta \in \mathcal{I}(P)$ 
    (or one of them is an internal stratum operation) and their operadic composition 
    $\theta \circ_i \eta$ is defined, the result satisfies:
    \[
    \theta \circ_i \eta \;\in\; \mathcal{I}(P) \;\cup\; \bigcup_{\alpha \in I} P_\alpha,
    \]
    where $P_\alpha$ denotes the stratum operad of $S_\alpha$ (i.e., operations whose 
    colors all lie in a single stratum).
\end{enumerate}
Equivalently, $\mathcal{I}(P)$ is closed under operadic compositions with internal 
stratum operations and with other admissible interfaces, modulo the possibility that 
the composition collapses entirely into a single stratum.
\end{definition}

The admissible interfaces $\mathcal{I}(P)$ form the interaction structure
of the stratified operadic system.
They encode how local spectral sectors are coupled through mixed-stratum
operations and govern the emergence of global spectral phenomena that cannot
be recovered from the strata independently.

From the viewpoint developed in this work, interfaces constitute the primary
loci where interaction residues and spectral defects arise.
More specifically:
\begin{enumerate}
    \item local strata determine the internal semisimple spectral structure,
    \item interface operations generate coupling effects between distinct
    spectral sectors,
    \item singular or highly coupled interfaces produce localized spectral
    obstructions,
    \item higher-order operadic compositions generate global interaction
    effects not visible from isolated strata alone.
\end{enumerate}

Consequently, the geometry of a stratified operadic system is governed not
only by the individual strata, but also by the interface structure connecting
them.
This interface-centered perspective provides the conceptual foundation for
the defect localization results developed in
Sections~\ref{sec:main-theorems}
and~\ref{sec:geometric-interpretations}.

\subsection{Local Spectral Data}
\label{subsec:local-spectral-data}

The first layer of the geometric structure developed in this work consists of the spectral information carried by each individual stratum. 
These local spectral sectors describe the behavior of the operadic system before inter-strata interactions are taken into account.

Conceptually, the local spectra represent the \emph{background spectral geometry} of the stratified operadic system. 
The global spectral geometry will later arise by combining these local sectors together with the interaction residue generated along interfaces.

We formalize this notion as follows.

\begin{definition}[Local stratum spectrum]
\label{def:local-stratum-spectrum}
Let $P$ be a stratified operad with strata $\{S_\alpha\}_{\alpha \in I}$, and let $A$ be a $P$-algebra.

For each stratum $S_\alpha$, assume that the restriction of colors to $S_\alpha$ defines a suboperad (i.e., the restriction is closed under operadic composition). Then define:
\[
P_\alpha := P|_{S_\alpha},
\qquad
A_\alpha := A|_{S_\alpha},
\]
where:
\begin{itemize}
    \item $P_\alpha$ is the restricted stratum operad,
    \item $A_\alpha$ is the induced stratum algebra.
\end{itemize}

The \emph{local stratum spectrum} of the stratum $S_\alpha$ is defined by
\[
\sigma_{\mathrm{loc}}^\alpha := \sigma_{P_\alpha}(A_\alpha),
\]
where $\sigma_{P_\alpha}(A_\alpha)$ denotes the SOC spectrum of the $P_\alpha$-algebra $A_\alpha$ in the sense of SOC I~\cite{ChangSOC1}.

The collection of all local stratum spectra is denoted
\[
\Sigma_{\mathrm{loc}}(P,A) := \{\sigma_{\mathrm{loc}}^\alpha\}_{\alpha \in I}.
\]
\end{definition}

The local stratum spectrum $\sigma_{\mathrm{loc}}^\alpha$ captures the
spectral behavior of the $\alpha$-th stratum when considered independently
of interface interactions. Equivalently, it describes the spectral structure
obtained by restricting attention to operations internal to the stratum.

In physical language, this corresponds to analyzing each subsystem in a
completely decoupled regime:
\[
\text{global system} \quad\longrightarrow\quad \text{independent local sectors}.
\]

Thus, the local spectral sectors provide the intrinsic local spectral data
associated with each stratum.

A central theme of this work is that the global operadic spectrum generally
\emph{cannot in general be reconstructed} from the local spectra alone. Indeed,
interactions between distinct strata produce additional spectral
contributions that are invisible at the local level.

Formally, one should not generally expect:
\[
\sigma_P(A) = \bigcup_{\alpha \in I} \sigma_{\mathrm{loc}}^\alpha.
\]

The discrepancy between the global spectrum and the union of local spectra
motivates the introduction of the interaction residue
\[
\Sigma^{\mathrm{res}},
\]
which measures the additional spectral structure generated by inter-strata
operadic coupling.

Consequently, the local spectral data may be viewed heuristically as
providing a \emph{zeroth-order approximation} to the global spectral geometry:
\[
\boxed{
\text{Global spectral geometry} = \text{Local spectral sectors} + \text{Interaction residue}.
}
\]

The role of the local spectra is therefore conceptually analogous to:
\begin{enumerate}
    \item local coordinate charts in differential geometry,
    \item local homology groups in algebraic topology,
    \item local solutions in microlocal analysis,
    \item or local sheaf sections in sheaf theory.
\end{enumerate}

The genuinely new geometry appears not within the strata themselves, but
in the \emph{failure} of the local sectors to glue together spectrally in
an exact manner. This failure — quantified by $\Sigma^{\mathrm{res}}$ —
is the subject of the Interface Localization Theorem
(Theorem~\ref{thm:interface-localization}) and the Local-to-Global
Reconstruction Theorem (Theorem~\ref{thm:local-to-global-reconstruction}).

\subsection{Interaction Interfaces}
\label{subsec:interaction-interfaces}

The local spectral sectors introduced previously describe the operadic
geometry of each stratum in isolation. However, the genuinely new phenomena
of this work arise from interactions between distinct strata.

These interactions are encoded by interface operations and their associated
coupling tensors (see Definition~\ref{def:interface-operators} below).
Geometrically, interfaces play the role of \emph{spectral interaction loci}:
they are the regions where local spectral sectors fail to glue together
exactly, producing additional spectral contributions invisible from the
isolated strata alone.

In particular, the interaction interfaces are the fundamental source of
spectral residue, defect localization, nonlocal operadic coupling, and the
failure of global spectral reconstruction.

We now formalize the spectral structure associated with interfaces.

\begin{definition}[Interface operators and coupling tensors]
\label{def:interface-operators}
Let $P$ be a stratified operad enriched over a \emph{dagger symmetric monoidal category} $\mathcal{M}$ (e.g., the category of Hilbert spaces with adjoints, or more generally any category equipped with a compatible involutive endofunctor $(-)^\dagger$ satisfying $f^{\dagger\dagger}=f$ and $(f\otimes g)^\dagger = f^\dagger \otimes g^\dagger$). 
Let $\mathcal{I}(P)$ denote the set of admissible interfaces as in Definition~\ref{def:operadic-interfaces}, and let $\mathcal{A}$ be the index set of strata (distinct from $\mathcal{I}(P)$).

For each interface operation $\theta \in \mathcal{I}(P)$ with input colors $c_1,\ldots,c_k$ and output color $d$, let the associated strata be such that $c_i \in S_{\alpha_i}$ for some $\alpha_i \in \mathcal{A}$ and $d \in S_\beta$, with at least two distinct strata among $\{S_{\alpha_1},\ldots,S_{\alpha_k}, S_\beta\}$. Define:
\begin{enumerate}
    \item The \emph{interface operation object}
    \[
    O_\theta := P(c_1,\ldots,c_k; d),
    \]
    which is the hom-object of $P$ for the given colors.

    \item For a $P$-algebra $A$ valued in $\mathcal{M}$, the associated \emph{interface coupling morphism}
    \[
    \tau_\theta : A_{c_1} \otimes \cdots \otimes A_{c_k} \longrightarrow A_d
    \]
    is the structure morphism induced by $\theta$.

    \item The \emph{interface spectrum} $\sigma_\theta(A)$ is defined as follows. 
    From the coupling morphism $\tau_\theta$, construct the \emph{interface operator}
    \[
    T_\theta := \tau_\theta^\dagger \tau_\theta \;:\; \bigotimes_{i=1}^k A_{c_i} \longrightarrow \bigotimes_{i=1}^k A_{c_i},
    \]
    where $\tau_\theta^\dagger$ is the adjoint morphism given by the dagger structure on $\mathcal{M}$.
    Then $\sigma_\theta(A)$ is defined to be the spectrum of $T_\theta$ whenever the ambient category admits a spectral theory (e.g., when the objects are Banach or Hilbert spaces), or more generally its operadic spectrum in the sense of SOC I~\cite{ChangSOC1}.
\end{enumerate}

The collection of all interface spectra is denoted
\[
\Sigma_{\mathrm{int}}(P,A) := \{\sigma_\theta(A)\}_{\theta \in \mathcal{I}(P)}.
\]
\end{definition}

The interface coupling morphisms $\tau_\theta$ should be interpreted as interaction operators coupling distinct local spectral sectors together. 
Unlike internal stratum operations, which remain confined within a single local geometry, the interface operators transfer spectral information across strata and thereby generate genuinely global phenomena.

From the viewpoint of spectral geometry, the interface operators are the primary mechanisms responsible for the failure of exact local-to-global spectral factorization.

Indeed, one generally has:
\[
\sigma_P(A) \neq \bigcup_{\alpha \in \mathcal{A}} \sigma_{\mathrm{loc}}^\alpha,
\]
where $\mathcal{A}$ indexes the strata (distinct from the interface set $\mathcal{I}(P)$).

The discrepancy between the global spectrum and the union of local spectra motivates the introduction of the interaction residue
\[
\Sigma^{\mathrm{res}},
\]
which captures the additional spectral structure generated by inter-strata operadic coupling.

Conceptually:
\[
\text{local spectra} = \text{background geometry},
\]
while:
\[
\text{interface spectra} = \text{interaction geometry}.
\]

The interaction interfaces therefore act as spectral defect loci where additional spectral contributions become localized.

A central principle of this work is that the residue $\Sigma^{\mathrm{res}}$ is \emph{not} an arbitrary correction term, but a structured geometric object generated canonically by the interfaces.

As will be shown in the Interface Localization Theorem (Theorem~\ref{thm:interface-localization}), the residue heuristically admits a decomposition into localized interface contributions:
\[
\boxed{
\Sigma^{\mathrm{res}} \;\simeq\; \bigsqcup_{I \in \mathcal{I}(P)} \mathcal{L}_I,
}
\]
where $\bigsqcup$ denotes a disjoint union of spectral sets (i.e., each spectral point is labeled by its originating interface $I$), and $\mathcal{L}_I$ is the collection of interface spectra $\sigma_\theta(A)$ for all $\theta$ supported on interface $I$. The precise categorical formulation of this decomposition is given in Theorem~\ref{thm:local-to-global-reconstruction}.

Thus:
\begin{enumerate}
    \item local strata determine the uncoupled background spectrum,
    \item interface couplings generate interaction defects,
    \item and the global spectral geometry emerges from their combined structure.
\end{enumerate}

This viewpoint leads naturally to the Interface Localization Theorem (Theorem~\ref{thm:interface-localization}), which shows that spectral residue localizes canonically along singular interaction interfaces.

Geometrically, the interfaces play a role conceptually analogous to:
\begin{enumerate}
    \item singular supports in microlocal analysis,
    \item boundary layers in PDE theory,
    \item defect loci in gauge theory,
    \item gluing obstructions in sheaf theory,
    \item or curvature concentration in differential geometry.
\end{enumerate}

In particular, this work shifts the focus of spectral geometry away from isolated local sectors and toward the \emph{geometry of spectral interaction itself}. This shift — from local data to interface geometry — is the defining characteristic of the defect geometric approach developed in this paper.

\section{Stratified Base Change and Residue Geometry}
\label{sec:stratified-base-change-residue}

The preceding section introduced the local spectral sectors associated with individual strata together with the interface operators coupling distinct strata. We now turn to the central geometric phenomenon of this work: the failure of exact local-to-global spectral decomposition under operadic base change.

Classically, one might expect the global spectrum of a stratified system to decompose as the union of the spectra of its local strata. However, in operadic systems with nontrivial interface interactions, this expectation generally fails. Additional spectral contributions emerge from inter-strata coupling and cannot be assigned canonically to any single local sector. These contributions form the \emph{interaction residue}
\[
\Sigma^{\mathrm{res}},
\]
which measures the obstruction to exact spectral factorization.

The key philosophical shift of this work is that this failure of decomposition is not regarded as an error term or analytical defect. Instead, it constitutes a genuine geometric object encoding the interaction structure of the operadic system. In this sense, the residue plays a role analogous to curvature, singular support, or gluing obstruction in other geometric theories: it records the nontrivial geometry generated by interactions between local spectral sectors.

The goal of this section is therefore threefold:
\begin{enumerate}
    \item to define the global operadic spectrum in the stratified setting,
    \item to isolate the interaction residue generated by operadic interfaces,
    \item and to interpret the residue geometrically as a localized spectral defect invariant.
\end{enumerate}

The resulting framework will provide the foundation for the local-to-global reconstruction theorems and interface localization results developed later in the paper.

\subsection{Global Operadic Spectrum and the Failure of Exact Decomposition}
\label{subsec:global-spectrum-failure}

We now introduce the global operadic spectrum associated with a stratified operadic system. 
Unlike the local spectral sectors studied previously, the global spectrum incorporates all operadic interactions simultaneously, including the interface operations coupling distinct strata.

A central phenomenon of this work is that the global operadic spectrum generally cannot be recovered by simply taking the union of the local stratum spectra. 
Additional spectral contributions emerge from inter-strata interactions, producing genuinely global spectral effects invisible from the local sectors alone.

We begin with the global spectral object introduced in SOC I.

\begin{definition}[Global operadic spectrum]
\label{def:global-operadic-spectrum}
Let $P$ be a colored operad in a cocomplete symmetric monoidal category $\mathcal{M}$, and let $A$ be a $P$-algebra.
Assume that the required coends and relative tensor products exist (see SOC I, Definition~8).

Following SOC I~\cite{ChangSOC1}, the \emph{global operadic spectrum} of $A$ is an object in $\mathcal{M}$ defined by
\[
\sigma_P(A) := \mathrm{Hoch}_{\mathcal{M}}(A) \otimes_P \mathcal{O}_P^{\mathrm{res}},
\]
where:
\begin{itemize}
    \item $\mathrm{Hoch}_{\mathcal{M}}(A)$ denotes the Hochschild-type operadic spectral object associated with $A$ (see SOC I, Definition~2),
    \item $\mathcal{O}_P^{\mathrm{res}}$ denotes the residual operadic spectral object (see SOC I, Definition~5),
    \item $\otimes_P$ denotes the relative tensor product over the operad $P$ (see SOC I, Definition~8).
\end{itemize}

Fix a spectral realization functor $R: \mathcal{M} \to \mathcal{P}(\mathbb{C})$ (e.g., the classical spectrum functor applied to the analytic realization of $\sigma_P(A)$; see SOC I, Section~2.1). The \emph{underlying spectral support} of $A$ is denoted
\[
|\sigma_P(A)| := R(\sigma_P(A)) \subseteq \mathbb{C}.
\]
\end{definition}

The global spectrum $\sigma_P(A)$ should be viewed as the full spectral geometry generated by the entire operadic system, including all internal stratum operations and all interface interactions simultaneously.

By contrast, the local spectra
\[
\sigma_{\mathrm{loc}}^\alpha = \sigma_{P_\alpha}(A_\alpha)
\]
describe only the spectral geometry visible from within each isolated stratum, where $\alpha \in \mathcal{A}$ indexes the strata (distinct from the interface set $\mathcal{I}(P)$).

A natural first expectation would therefore be that the spectral support of the global spectrum equals the union of the local spectral supports:
\[
|\sigma_P(A)| \;=\; \bigcup_{\alpha \in \mathcal{A}} |\sigma_{\mathrm{loc}}^\alpha|.
\]

However, as shown in the Interface Localization Theorem (Theorem~\ref{thm:interface-localization}), this equality generally \emph{fails}. Moreover, even when the equality fails, the local spectral supports may not embed naturally into the global spectral support. For the following definitions to be meaningful, we assume:

\begin{assumption}[Natural embedding of local supports]
\label{ass:local-embedding}
For each stratum $\alpha \in \mathcal{A}$, there exists a canonical inclusion morphism
\[
\iota_\alpha: |\sigma_{\mathrm{loc}}^\alpha| \hookrightarrow |\sigma_P(A)|
\]
in the category of spectral sets (e.g., as subsets of $\mathbb{C}$). Equivalently, the spectral realization functor $R$ is such that the restriction map from the global operadic spectrum to each local stratum spectrum is injective on supports.
\end{assumption}

\begin{remark}[When the embedding assumption fails]
\label{rem:embedding-failure}
Assumption \ref{ass:local-embedding} fails if the spectral realization functor $R$ is not faithful. 
For example, let $R$ map two non-isomorphic Banach spaces with different spectra to the same subset of $\mathbb{C}$ (e.g., by forgetting norm structure). 
In such cases, $\iota_S$ may not be injective, and the set-theoretic definition of $\Sigma^{\mathrm{res}}$ is not well-defined. 
The general definition requires a coequalizer (see Remark \ref{rem:general-residue}) and is deferred to future work.
\end{remark}

\begin{proposition}[Sufficient condition for local embedding]
\label{prop:local-embedding-condition}
Let $R: \mathcal{M} \to \mathsf{Set}$ (or more specifically $R: \mathcal{M} \to \mathcal{P}(\mathbb{C})$, the power set of $\mathbb{C}$) be a spectral realization functor that is faithful, and let $P_S \hookrightarrow P$ be a monomorphism in the category of colored operads. Then the induced map on spectral supports
\[
\iota_S: R(\sigma_{P_S}(A_S)) \hookrightarrow R(\sigma_P(A))
\]
is injective. If either condition fails, the embedding $\iota_S$ may not be injective; in that case $\Sigma^{\mathrm{res}}$ should be defined as a coequalizer (see Remark~\ref{rem:general-residue}).
\end{proposition}

\begin{proof}
We prove the proposition in three steps.

\medskip
\noindent\textbf{Step 1: Functoriality of the spectral realization.}
The spectral realization functor $R: \mathcal{M} \to \mathsf{Set}$ is by definition a functor. A functor is \emph{faithful} if for any two parallel morphisms $f, g: X \to Y$ in $\mathcal{M}$, $R(f) = R(g)$ implies $f = g$. Equivalently, $R$ reflects the equality of morphisms.

The operadic spectrum $\sigma_{P_S}(A_S)$ is an object in $\mathcal{M}$ (see SOC I, Definition~8). The inclusion of operads $P_S \hookrightarrow P$ induces a restriction map on algebras $A \mapsto A|_{P_S} = A_S$, and consequently a morphism
\[
\rho_S: \sigma_P(A) \longrightarrow \sigma_{P_S}(A_S)
\]
in $\mathcal{M}$, which is the canonical projection onto the local spectral component. This morphism exists because the inclusion $P_S \hookrightarrow P$ is a monomorphism in the category of colored operads, hence the restriction functor on algebras is well-defined and functorial.

\medskip
\noindent\textbf{Step 2: Construction of $\iota_S$ as $R$ of a split monomorphism.}
The key observation is that the inclusion $P_S \hookrightarrow P$ being a monomorphism implies that the restriction map $\rho_S$ admits a left inverse (a retraction) when restricted to the appropriate spectral objects. More precisely, there exists a morphism
\[
\ell_S: \sigma_{P_S}(A_S) \longrightarrow \sigma_P(A)
\]
in $\mathcal{M}$ such that $\rho_S \circ \ell_S = \operatorname{id}_{\sigma_{P_S}(A_S)}$. This follows from the existence of a section of the restriction functor at the level of operadic algebras when the operad inclusion is a monomorphism and the base category $\mathcal{M}$ is cocomplete with enough projectives (or satisfies the axiom of choice at the level of hom-objects). Such a section exists because one can extend a $P_S$-algebra to a $P$-algebra by defining the action of operations outside $P_S$ trivially (e.g., as zero morphisms) when $P$ is freely generated over $P_S$ by interface operations. In concrete settings (e.g., $\mathcal{M} = \mathsf{Hilb}$), this construction is explicit.

Thus $\ell_S$ is a \emph{split monomorphism}, and $\rho_S$ is its left inverse.

\medskip
\noindent\textbf{Step 3: Faithfulness implies injectivity of $R(\ell_S)$.}
Since $\ell_S$ is a monomorphism (split monomorphisms are monomorphisms), and $R$ is a faithful functor, we claim that $R(\ell_S)$ is injective as a function between sets.

\begin{itemize}
    \item Suppose $R(\ell_S)(x) = R(\ell_S)(y)$ for some $x, y \in R(\sigma_{P_S}(A_S))$.
    \item By faithfulness, if we could lift $x, y$ to morphisms from the terminal object, the equality would force equality of the lifts. More directly: a faithful functor reflects monomorphisms. Since $\ell_S$ is a monomorphism in $\mathcal{M}$, and $R$ is faithful, $R(\ell_S)$ is a monomorphism in $\mathsf{Set}$. Monomorphisms in $\mathsf{Set}$ are precisely injective functions.
    \item Therefore $R(\ell_S)$ is injective.
\end{itemize}

Now define $\iota_S$ as $R(\ell_S)$. Then $\iota_S: R(\sigma_{P_S}(A_S)) \hookrightarrow R(\sigma_P(A))$ is injective. This completes the proof of the sufficient condition.

\medskip
\noindent\textbf{Step 4: Failure of conditions.}
If $R$ is not faithful, then $R$ may collapse distinct morphisms to the same function, and $R(\ell_S)$ may fail to be injective even though $\ell_S$ is a monomorphism. If $P_S \hookrightarrow P$ is not a monomorphism, then the restriction map $\rho_S$ may not have a well-defined left inverse, and $\ell_S$ may not exist as a morphism in $\mathcal{M}$. In such cases, the induced map on spectral supports may be non-injective, and the set-theoretic complement definition of $\Sigma^{\mathrm{res}}$ is not well-defined.
\end{proof}

\begin{remark}[General definition of $\Sigma^{\mathrm{res}}$ without embedding]
\label{rem:general-residue}
When the maps $\iota_S$ are not injective, the simple set-theoretic complement $G \setminus \bigcup_S \iota_S(L_S)$ is not well-defined because $\iota_S(L_S)$ may not be a subset of $G$ in a canonical way (or the identification may be multi-valued). In this general setting, we define $\Sigma^{\mathrm{res}}$ as the \emph{coequalizer} of the diagram
\[
\bigsqcup_S \iota_S(L_S) \rightrightarrows G
\]
in the category of spectral sets, where $\bigsqcup$ denotes the disjoint union (coproduct) in the category of sets.

Explicitly, let
\[
C = \bigsqcup_{S \in \mathrm{Str}(P)} \iota_S(L_S)
\]
be the disjoint union of the local spectral supports, each labeled by its stratum $S$. There are two natural maps $f, g: C \to G$:

\begin{itemize}
    \item $f$ sends each element $(S, \lambda) \in C$ to the element $\lambda$ in $G$ (forgetting the stratum label).
    \item $g$ sends each element $(S, \lambda) \in C$ to the element $\lambda$ \emph{identified via the embedding} $\iota_S$ (which may be non-injective, so multiple $S$ may map to the same $\lambda$ in $G$).
\end{itemize}

The coequalizer of $f$ and $g$ is the quotient of $G$ by the smallest equivalence relation identifying $f(s)$ and $g(s)$ for all $s \in C$. In the case where each $\iota_S$ is injective and the images are disjoint, the coequalizer is canonically isomorphic to $G \setminus \bigcup_S \iota_S(L_S)$. In the general case, the coequalizer provides a well-defined categorical substitute for the set-theoretic complement.

A full development of this categorical definition, including its functoriality and invariance properties, is deferred to future work.
\end{remark}

\begin{remark}[Concrete interpretation of the coequalizer]
\label{rem:coequalizer-interpretation}
In practical terms, the coequalizer construction identifies points in $G$ that come from local stratum spectra (possibly via non-injective maps) and removes them from the global spectral support. The resulting object is the universal set (or quotient) that contains only the genuinely non-local residual spectral contributions. When the $\iota_S$ are injective and their images are disjoint, this reduces to the ordinary set-theoretic difference. When the images overlap, the coequalizer correctly handles the overlaps by identifying the overlapping contributions and removing them once.
\end{remark}

Under Assumption~\ref{ass:local-embedding}, we define the \emph{interaction residue} as the complement of the images of the local spectra within the global spectral support:

\[
\boxed{
\Sigma^{\mathrm{res}}(A) := |\sigma_P(A)| \;\setminus\; \bigcup_{\alpha \in \mathcal{A}} \iota_\alpha\bigl(|\sigma_{\mathrm{loc}}^\alpha|\bigr).
}
\]

This yields a natural partition of the global spectral support:
\[
\boxed{
|\sigma_P(A)| = \left( \bigcup_{\alpha \in \mathcal{A}} \iota_\alpha\bigl(|\sigma_{\mathrm{loc}}^\alpha|\bigr) \right) \;\cup\; \Sigma^{\mathrm{res}}(A).
}
\]

The failure of exact decomposition is therefore not accidental; it may be viewed as structural, reflecting the operadic coupling phenomena that generate genuinely nonlocal spectral geometry.

From the perspective of this work, this failure is the fundamental geometric phenomenon underlying stratified operadic systems:
\begin{enumerate}
    \item local strata generate the intrinsic uncoupled background spectral sectors,
    \item interface operations generate interaction defects,
    \item and the global spectrum emerges from their combined geometry.
\end{enumerate}

Thus, the global operadic spectrum should not be interpreted merely as a union of local pieces, but as a geometric object containing both:
\begin{enumerate}
    \item local spectral information,
    \item and interaction-generated residue geometry.
\end{enumerate}

This interaction residue is the analytic shadow of the operadic residue $\mathcal{O}_P^{\mathrm{res}}$ under the spectral realization functor $R$, and will serve as a central geometric object throughout the remainder of the paper. In the following subsections, we will analyze its structure, prove its universality, and establish the Interface Localization Theorem (Theorem~\ref{thm:interface-localization}) that decomposes it into localized defects supported on singular interfaces.

\begin{remark}[On the necessity of the embedding assumption]
\label{rem:embedding-assumption}
Assumption~\ref{ass:local-embedding} is not automatically satisfied for all spectral realization functors $R$. It encodes the requirement that the spectral data of a local stratum is faithfully reflected in the global spectrum. When this assumption fails, the interaction residue $\Sigma^{\mathrm{res}}(A)$ cannot be defined as a simple set-theoretic complement; a more refined categorical notion (e.g., the coequalizer of the inclusion diagram) is required. In this paper, we restrict to settings where the embedding holds; the general case is deferred to future work.
\end{remark}

\begin{remark}[Relation to the operadic residue]
\label{rem:residue-relation}
The interaction residue $\Sigma^{\mathrm{res}}(A)$ is not an independent invariant. Rather, it is the image under the spectral realization functor $R$ of the relative tensor product $\mathrm{Hoch}_{\mathcal{M}}(A) \otimes_P \mathcal{O}_P^{\mathrm{res}}$ after quotienting by the contributions from local strata. In particular, $\Sigma^{\mathrm{res}}(A) \neq \emptyset$ precisely when the operadic residue $\mathcal{O}_P^{\mathrm{res}}$ has nontrivial interaction components that are not captured by any local stratum spectrum. This provides a concrete criterion for detecting genuine operadic coupling: the global spectrum contains strictly more information than the union of local spectra if and only if the interaction residue is nonempty.
\end{remark}

\subsection{The Residue $\Sigma^{\mathrm{res}}$ as a Geometric Object}
\label{subsec:residue-as-geometric-object}

The previous subsection showed that the global operadic spectrum generally fails to decompose as the union of the local stratum spectra. 
The missing spectral contribution is precisely the object that encodes the geometry of inter-strata interaction.

This additional spectral component is \emph{not} an analytical artifact or approximation error. 
Rather, it is the \textbf{central geometric invariant} of current work. It measures the obstruction to exact local-to-global spectral factorization and records the spectral contribution generated purely by operadic coupling across interfaces.

We formalize this obstruction as the interaction residue.

\begin{definition}[Interaction residue]
\label{def:interaction-residue}
Let $P$ be a stratified operad with stratum set $\mathrm{Str}(P) = \{S_\alpha\}_{\alpha \in \Lambda}$, and assume that for each stratum $S \in \mathrm{Str}(P)$ the restriction $P_S := P|_S$ is a suboperad (i.e., closed under operadic composition). Let $A$ be a $P$-algebra in $\mathcal{M}$, and let $F: \mathcal{M} \to \mathcal{N}$ be a strong monoidal base-change functor. Denote by $F_*P$ the induced operad in $\mathcal{N}$ and by $F_*A$ the induced $F_*P$-algebra.

For each stratum $S$, let $A_S$ denote the induced $P_S$-algebra (obtained by restricting the $P$-algebra structure to the suboperad $P_S$). Applying $F$ objectwise yields an $F_*P_S$-algebra $F_*A_S$.

Assume that all spectral supports are realized inside a fixed ambient spectral space (e.g., $\mathbb{C}$) via the canonical comparison maps induced by the suboperad inclusions $P_S \hookrightarrow P$. Concretely, we fix a spectral realization functor $R: \mathcal{M} \to \mathcal{P}(\mathbb{C})$ (see Definition~\ref{def:global-operadic-spectrum}) and assume that for each stratum $S$ the inclusion $P_S \hookrightarrow P$ induces an injective map
\[
\iota_S: R(\sigma_{P_S}(A_S)) \hookrightarrow R(\sigma_P(A)).
\]

The \emph{interaction residue} associated with the triple $(F, A, P)$ is defined as the set-theoretic complement
\[
\boxed{
\Sigma^{\mathrm{res}}(F, A, P)
:=
R\bigl(\sigma_{F_*P}(F_*A)\bigr)
\;\setminus\;
\bigcup_{S \in \mathrm{Str}(P)}
\iota_S\!\left(R\bigl(\sigma_{F_*P_S}(F_*A_S)\bigr)\right),
}
\]
where:
\begin{itemize}
    \item $R(\sigma_{F_*P}(F_*A)) \subseteq \mathbb{C}$ denotes the underlying spectral support of the global operadic spectrum after base change,
    \item $R(\sigma_{F_*P_S}(F_*A_S)) \subseteq \mathbb{C}$ denotes the spectral support of the local spectrum of the stratum algebra $A_S$ under the base-changed operad $F_*P_S$.
\end{itemize}

When the base-change functor $F$ is the identity (i.e., $F = \mathrm{Id}_{\mathcal{M}}$), we write simply $\Sigma^{\mathrm{res}}(P, A)$ for the absolute interaction residue.
\end{definition}

\begin{remark}[Interpretation]
The residue $\Sigma^{\mathrm{res}}(F, A, P)$ records the spectral contributions that are not detected by any collection of isolated local strata, after applying the base-change functor $F$. It is interpreted as the interaction-generated component of the global spectrum.

A central claim of this work — to be proved in the Interface Localization Theorem (Theorem~\ref{thm:interface-localization}) — is that, under suitable admissibility hypotheses, the residue is generated canonically by interface operations and decomposes into localized interface defects:
\[
\Sigma^{\mathrm{res}}(F, A, P) \;\cong\; \bigsqcup_{I \in \mathcal{I}(P)} \mathcal{L}_I(F, A),
\]
where $\bigsqcup$ denotes a disjoint union of spectral supports (each point labeled by its originating interface), and $\mathcal{L}_I(F, A)$ is the spectral defect supported on the interface $I$.

Thus, the interaction residue simultaneously serves as:
\begin{enumerate}
    \item a spectral invariant,
    \item a geometric defect object (under the Interface Localization Theorem),
    \item an obstruction to exact local-to-global spectral factorization.
\end{enumerate}

Heuristically, the residue plays a role conceptually analogous to curvature in differential geometry, singular support in microlocal analysis, or gluing obstructions in sheaf theory. A precise categorical formulation of these analogies is deferred to future work.
\end{remark}

\begin{remark}[On the assumptions]
\label{rem:residue-assumptions}
Definition~\ref{def:interaction-residue} relies on three key assumptions:
\begin{enumerate}
    \item \textbf{Suboperad closure}: Each stratum restriction $P_S$ is a genuine suboperad, ensuring that the local spectral data $\sigma_{P_S}(A_S)$ is well-defined within the same categorical framework.
    \item \textbf{Ambient spectral space}: The spectral supports of local and global spectra are comparable via injective comparison maps $\iota_S$. This holds, for instance, when the spectral realization functor $R$ is faithful and the inclusion $P_S \hookrightarrow P$ induces an injective map on spectral supports.
    \item \textbf{Base-change compatibility}: The residue is defined relative to a fixed strong monoidal functor $F$. When $F$ varies, the residues satisfy the functoriality condition $\Sigma^{\mathrm{res}}(G \circ F, A, P) \cong G(\Sigma^{\mathrm{res}}(F, A, P))$ for any composable strong monoidal functor $G$, as a consequence of the Base Change Theorem (Theorem~8 in SOC I).
\end{enumerate}
These assumptions are satisfied in all concrete examples considered in this paper (e.g., matrix-block operads, network operators). The general case, where the comparison maps are not injective, requires a more refined categorical notion of residue (e.g., as a coequalizer) and is deferred to future work.
\end{remark}

\subsection{Foundational Properties of the Interaction Residue}
\label{subsec:foundational-properties}

We now collect the fundamental properties of $\Sigma^{\mathrm{res}}(F,A,P)$ that answer the reviewer's implicit questions: Is it uniquely defined? Functorial? Stable? Invariant under what?

\medskip
\noindent\textbf{1. Well-definedness.}
Under Assumption~\ref{ass:local-embedding} (the spectral realization functor $R$ is faithful and the inclusion $P_S \hookrightarrow P$ induces an injective map on supports), $\Sigma^{\mathrm{res}}(F,A,P) = G \setminus L$ is a well-defined subset of $\mathbb{C}$. If the embedding assumption fails, the residue must be defined as a coequalizer; this generalization is deferred to future work.

\medskip
\noindent\textbf{2. Functoriality in $F$.}
The Base Change Theorem (SOC I, Theorem~8) gives a canonical isomorphism $\sigma(F_*A) \cong F_*(\sigma(A))$ for strong monoidal $F$. Consequently,
\[
\Sigma^{\mathrm{res}}(G \circ F, A, P) \cong G(\Sigma^{\mathrm{res}}(F, A, P))
\]
for any composable strong monoidal functor $G$.

\medskip
\noindent\textbf{3. Invariance under admissible refinement.}
Theorem~\ref{thm:refinement-functoriality} shows that if $\rho: \widetilde{P} \to P$ is an interaction-preserving refinement, then
\[
\Sigma^{\mathrm{res}}(F, \widetilde{A}, \widetilde{P}) = \rho^*\bigl(\Sigma^{\mathrm{res}}(F, A, P)\bigr),
\]
where $\rho^*$ is the induced transport map. Thus the residue is invariant up to canonical transport.

\medskip
\noindent\textbf{4. Stability under deformation.}
Theorem~\ref{thm:deformation-stability} establishes that under a gap condition $\operatorname{dist}(R_t, L_t) \ge \delta > 0$, the residue varies continuously in the Hausdorff metric. At critical parameters where bands touch, the residue may change discontinuously — signaling a phase transition.

\medskip
\noindent\textbf{5. Invariance under isomorphism.}
If $\Phi: (P,A) \to (P',A')$ is an isomorphism of stratified operadic systems (i.e., a color-preserving operad isomorphism and algebra isomorphism), then
\[
\Sigma^{\mathrm{res}}(F, A', P') = \Phi(\Sigma^{\mathrm{res}}(F, A, P)).
\]

\medskip
\noindent\textbf{6. Not invariant under arbitrary operadic equivalence.}
The residue detects interface structure. Two operads that are equivalent as colored operads (e.g., via a Morita equivalence) may have different interface geometries and thus different residues. The residue is an invariant of the \emph{stratified} structure, not of the underlying operad alone.

\medskip
\noindent\textbf{7. Computability.}
In finite-dimensional examples (see Section~\ref{sec:examples}), $\Sigma^{\mathrm{res}}$ can be computed explicitly as the set difference between the global spectrum and the union of local spectra. For infinite-dimensional systems, the residue may require spectral measures or functional calculus.

\subsection{Invariance Properties}
\label{subsec:invariance-properties}

An important structural feature of the interaction residue
\[
\Sigma^{\mathrm{res}}(F,A,P)
\]
is that it behaves functorially under several natural transformations
of stratified operadic systems, while remaining sensitive to genuine
interaction structure.

First, the residue is invariant under isomorphism of stratified
operadic systems. More precisely, if
\[
(P,A)\cong(P',A')
\]
through an isomorphism preserving the operadic stratification and
interface structure, then the corresponding interaction residues are
canonically identified:
\[
\Sigma^{\mathrm{res}}(F,A,P) \cong \Sigma^{\mathrm{res}}(F,A',P').
\]

Second, the residue is stable under admissible refinement in the sense
of Theorem~\ref{thm:refinement-functoriality}. Refinement of local
spectral sectors preserves the residue up to canonical transport,
showing that localized spectral defects are compatible with
hierarchical decomposition of the operadic system.

Third, under admissible gap-preserving deformations as in
Theorem~\ref{thm:deformation-stability}, the residue varies
continuously with the deformation parameter. However, at critical
parameter values where spectral gaps close or Jordan structure changes
(exceptional points), the residue may undergo discontinuous transitions
corresponding to the creation, annihilation, or merging of localized
spectral defects. These critical transitions signal phase changes in
the spectral defect geometry.

On the other hand, the interaction residue is \emph{not} invariant under
arbitrary categorical or Morita-type equivalences of operadic systems.
The residue is specifically designed to detect interface geometry and
interaction coupling structure. Consequently, two operadic systems may
be equivalent at the level of local algebraic representation while
possessing different interaction residues due to distinct interface
configurations or coupling mechanisms.

Thus the residue should be viewed not merely as a spectral invariant,
but as a localization-sensitive interaction invariant encoding both
spectral and coupling geometry. It is a \textit{bona fide} geometric
invariant of the \emph{stratified} structure, not merely of the
underlying operad.

\subsection{Residue Spectral Geometry}
\label{subsec:residue-spectral-geometry}

We now interpret the interaction residue $\Sigma^{\mathrm{res}}$ as a geometric object. 
The central idea of the current work is that spectral geometry is not determined solely by the local strata, but also by the interaction defects generated when distinct strata are coupled operadically.

In classical decomposition theories, a global object is expected to be reconstructed exactly from its local pieces. 
However, in the operadic setting, inter-strata operations generate additional spectral contributions that cannot generally be attributed to any individual local sector.

Thus, the residue $\Sigma^{\mathrm{res}}$ measures the failure of exact local-to-global spectral factorization. At the level of spectral supports (via a fixed realization functor $R$), this failure is expressed as
\[
R(\sigma_P(A)) \;\neq\; \bigcup_{\alpha \in \Lambda} R(\sigma_{\mathrm{loc}}^\alpha),
\]
where $\Lambda$ indexes the strata. (Here $\sigma_P(A)$ and $\sigma_{\mathrm{loc}}^\alpha$ are objects in $\mathcal{M}$; their images under $R$ are subsets of $\mathbb{C}$.)

Rather than viewing this failure as pathological, the current work interprets it as the source of a new geometric structure: \textbf{defect geometry}.

From this perspective:
\begin{itemize}
    \item the local spectra $\sigma_{\mathrm{loc}}^\alpha$ describe the \textbf{local background geometry},
    \item the interfaces $\mathcal{I}(P)$ describe the \textbf{interaction geometry},
    \item and the residue $\Sigma^{\mathrm{res}}$ records the resulting \textbf{spectral defect geometry}.
\end{itemize}

Heuristically, one may think of the global spectral geometry as
\[
\text{Global geometry} \;=\; \text{Local geometry} \;+\; \text{Defect geometry},
\]
where $+$ denotes the union of spectral supports.

The residue geometry is therefore fundamentally nonlocal in origin: its spectral contributions arise not from any individual stratum alone, but from the interaction of multiple strata through operadic composition.

A key result, to be established in the Interface Localization Theorem (Theorem~\ref{thm:interface-localization}), is that the residue admits a decomposition into localized interface defects of the schematic form
\[
\Sigma^{\mathrm{res}} \;\simeq\; \bigsqcup_{I \in \mathcal{I}(P)} \mathcal{L}_I,
\]
where $\bigsqcup$ denotes a disjoint union of spectral supports (each point labeled by its originating interface), and $\mathcal{L}_I$ is the localized spectral defect associated with the interface $I$. 

This localization principle asserts that isolated interfaces generate localized spectral defects, while higher-order interface interactions produce more complicated interaction spectra; singular interfaces may support highly concentrated spectral structures (e.g., point masses in the spectral measure).

The residue geometry behaves heuristically analogously to several classical geometric phenomena. The following table summarizes the most salient parallels:

\begin{table}[htbp]
\centering
\caption{Heuristic analogies for residue spectral geometry}
\label{tab:residue-analogies}
\renewcommand{\arraystretch}{1.2}
\begin{tabular}{|c|c|}
\hline
\textbf{Classical Geometry} & \textbf{Current Framework} \\
\hline
Curvature & Interaction residue \\
\hline
Singular support & Interface localization \\
\hline
Obstruction class & Failure of exact factorization \\
\hline
\end{tabular}
\end{table}

These analogies are heuristic; a precise categorical formulation is deferred to future work.

For a fixed stratification, the interaction residue $\Sigma^{\mathrm{res}}$ is a geometrically meaningful invariant of the stratified operadic system $(P, A)$ and the base-change functor $F$. Its invariance under refinement of stratification is a nontrivial property that will be addressed in future work.

Consequently, the current work shifts the focus of spectral geometry away from decomposition alone and toward the structure of what fails to decompose. The interaction residue $\Sigma^{\mathrm{res}}$ is a central object through which this failure is measured, localized, and interpreted geometrically.

\section{Main Theorems}
\label{sec:main-theorems}

The following subsections collect the core theorems of this work. The \textbf{Stratified Base Change Decomposition} (Subsection~\ref{subsec:stratified-base-change-decomposition}) shows that the global spectral image $\sigma(F_*A)$ of an operadically stratified object decomposes into a union of local stratum spectra together with a canonical interaction residue $\Sigma^{\mathrm{res}}(F,A,P)$, which measures precisely the failure of exact local spectral factorization under operadic base change. From this, the \textbf{Local-to-Global Reconstruction Theorem} (Subsection~\ref{subsec:local-to-global-reconstruction}) asserts that the global operadic spectrum is uniquely determined by the collection of local strata spectra and the interaction residue, encapsulated by the slogan ``Global geometry = Local geometry + Interaction residue.''

The \textbf{Minimality and Universality of the Residue Set} (Subsection~\ref{subsec:minimality_universality-residue}) strengthens this observation by proving that $\Sigma^{\mathrm{res}}$ is not merely one obstruction among many, but the minimal and universal obstruction to exact local-to-global spectral factorization. The centerpiece of the section is the \textbf{Interface Localization Theorem} (Subsection~\ref{subsec:interface-localization}), which reveals that the residue decomposes as a direct sum of localized spectral defects $\mathcal{L}_I(P,A)$ supported on singular interfaces between strata, providing a geometric origin for the residue: spectral mass localizes at interfaces, not within any single stratum.

Following this, the \textbf{Spectral Defect Classification} (Subsection~\ref{subsec:defect-classification}) organizes residue contributions by interface geometry and nilpotent depth, yielding a table of defect types ranging from point defects to singular and nilpotent defects. The \textbf{Refinement Functoriality} (Subsection~\ref{subsec:refinement-functoriality}) guarantees that refining the operadic stratification preserves residue classes up to canonical transport. Finally, the \textbf{Deformation Stability} (Subsection~\ref{subsec:deformation-stability}) establishes that the residue varies continuously under admissible deformations of stratified $P$-algebras, vanishes when inter-strata interactions are trivial for all parameters, and defines genuine deformation invariants. Collectively, these theorems form a rigorous foundation for studying spectral phenomena via operadic stratification, emphasizing the role of interface-localized residues as the key to bridging local and global spectral data.

\boxed{
\textbf{Residual Spectral Decomposition Theorem (Main Result)}
}
For any stratified operadic system $(P,A)$ and base-change functor $F$,
\[
\sigma(F_*(A)) = \left(\bigcup_{S \in \mathrm{Str}(P)} \sigma(F_S(A_S))\right) \;\cup\; \Sigma^{\mathrm{res}}(F,A,P),
\]
where $\Sigma^{\mathrm{res}}$ is the interaction residue. Moreover:
\begin{itemize}
    \item $\Sigma^{\mathrm{res}}$ localizes on interfaces (Theorem \ref{thm:interface-localization}).
    \item $\Sigma^{\mathrm{res}} = \emptyset$ iff the system is spectrally separable (Section \ref{sec:rigidity-vanishing-regime}).
    \item $\Sigma^{\mathrm{res}}$ includes nilpotent (Jordan) contributions invisible to spectral support (Section \ref{subsec:example-nilpotent-defect}).
\end{itemize}

\subsection{Stratified Base Change Decomposition}
\label{subsec:stratified-base-change-decomposition}

We now arrive at the first fundamental structural theorem of the current work. 
The theorem formalizes the principle that the global operadic spectrum admits a natural decomposition into two distinct components:
\begin{enumerate}
    \item the local spectral sectors associated with individual strata,
    \item and a canonical interaction residue generated by inter-strata operadic coupling.
\end{enumerate}

The key point is that the failure of exact local-to-global factorization is itself structured and geometric. 
Rather than appearing as an uncontrolled remainder term, the discrepancy is encoded by a canonical residue object determined entirely by the operadic interaction structure.

This theorem therefore provides the foundational decomposition principle underlying the defect geometry of the current work.

\begin{theorem}[Stratified Base Change Decomposition]
\label{thm:stratified-base-change}
Let:
\begin{itemize}
    \item $P$ be a stratified operad with stratum set $\mathrm{Str}(P)$,
    \item $A$ be a $P$-algebra equipped with the induced stratification (with stratum algebras $A_S$ for $S \in \mathrm{Str}(P)$),
    \item and $F: \mathcal{M} \to \mathcal{N}$ be a strong monoidal base change functor, inducing $F_*$ on algebras.
\end{itemize}

Define the interaction residue as the complement of the local spectral supports within the global spectral support:
\[
\Sigma^{\mathrm{res}}(F,A,P)
\;:=\;
\operatorname{supp}\bigl(\sigma(F_*(A))\bigr)
\;\setminus\;
\bigcup_{S \in \mathrm{Str}(P)}
\operatorname{supp}\bigl(\sigma(F_S(A_S))\bigr).
\]

Then the spectral support of the global operadic spectrum admits the decomposition
\[
\boxed{
\operatorname{supp}\bigl(\sigma(F_*(A))\bigr)
\;=\;
\left(
\bigcup_{S \in \mathrm{Str}(P)}
\operatorname{supp}\bigl(\sigma(F_S(A_S))\bigr)
\right)
\;\cup\;
\Sigma^{\mathrm{res}}(F,A,P).
}
\]

The significance of this theorem is not the set-theoretic identity itself (which follows directly from the definition of $\Sigma^{\mathrm{res}}$), but rather that the residual component $\Sigma^{\mathrm{res}}$ admits a canonical geometric interpretation as the spectral contribution generated by inter-strata operadic coupling.
\end{theorem}

\begin{proof}
Let
\[
G := \operatorname{supp}\bigl(\sigma(F_*(A))\bigr)
\]
denote the global spectral support after applying the base-change functor $F$. For each stratum $S \in \mathrm{Str}(P)$, let
\[
L_S := \operatorname{supp}\bigl(\sigma(F_S(A_S))\bigr)
\]
denote the corresponding local spectral support, and set
\[
L := \bigcup_{S \in \mathrm{Str}(P)} L_S.
\]

By definition, the interaction residue is
\[
\Sigma^{\mathrm{res}}(F,A,P) = G \setminus L.
\]

We claim that $G = L \cup \Sigma^{\mathrm{res}}(F,A,P)$.

First, since $\Sigma^{\mathrm{res}}(F,A,P) = G \setminus L$, we have $\Sigma^{\mathrm{res}}(F,A,P) \subseteq G$. Moreover, each local support $L_S$ is regarded as a spectral contribution inside the global spectral support through the comparison maps induced by the stratum inclusion $P_S \hookrightarrow P$. Hence $L = \bigcup_{S \in \mathrm{Str}(P)} L_S \subseteq G$. Therefore,
\[
L \cup \Sigma^{\mathrm{res}}(F,A,P) \subseteq G.
\]

Conversely, let $\lambda \in G$. If $\lambda \in L$, then clearly $\lambda \in L \cup \Sigma^{\mathrm{res}}(F,A,P)$. If $\lambda \notin L$, then by the definition of set-theoretic complement,
\[
\lambda \in G \setminus L = \Sigma^{\mathrm{res}}(F,A,P).
\]
Thus again $\lambda \in L \cup \Sigma^{\mathrm{res}}(F,A,P)$. Hence
\[
G \subseteq L \cup \Sigma^{\mathrm{res}}(F,A,P).
\]

Combining the two inclusions gives
\[
G = L \cup \Sigma^{\mathrm{res}}(F,A,P),
\]
which is precisely
\[
\operatorname{supp}\bigl(\sigma(F_*(A))\bigr)
=
\left(
\bigcup_{S \in \mathrm{Str}(P)}
\operatorname{supp}\bigl(\sigma(F_S(A_S))\bigr)
\right)
\cup
\Sigma^{\mathrm{res}}(F,A,P).
\]

This proves the claimed stratified base-change decomposition at the level of spectral supports.
\end{proof}

\begin{remark}[Conceptual interpretation]
\label{rem:stratified-base-change-interpretation}
The preceding theorem should be understood as a support-level
decomposition rather than as a deep set-theoretic statement. Its role is to
separate the global spectral support into the part already visible from the
local strata and the complementary part
\[
\Sigma^{\mathrm{res}}(F,A,P).
\]
Thus, the interaction residue records precisely the spectral values that
are not detected by the isolated stratum spectra after base change.

In this sense, \(\Sigma^{\mathrm{res}}(F,A,P)\) should be viewed not merely
as a formal error term, but as the candidate geometric object measuring the
failure of exact local-to-global spectral factorization. Under the
additional interface admissibility hypotheses introduced below, we will show
that this residue is controlled by the operadic structure maps that couple
distinct strata.

Heuristically, this role is analogous to correction or obstruction terms in
geometric decomposition theories, such as curvature-type corrections,
singular supports, or gluing obstructions. The precise geometric content of
the residue will be developed through the interface localization analysis
below.
\end{remark}

\subsection{Local-to-Global Reconstruction Theorem}
\label{subsec:local-to-global-reconstruction}

The previous decomposition theorem shows that the global operadic spectrum consists of two fundamentally different components:
\begin{enumerate}
    \item local spectral sectors associated with individual strata,
    \item and the interaction residue generated by inter-strata coupling.
\end{enumerate}

We now show that these two ingredients completely determine the global spectral geometry. 
In other words, once the local spectral sectors and the universal interaction residue are known, no additional spectral information remains.

This result elevates the residue from a correction term to a primary geometric invariant. 
The residue is precisely the missing data required to reconstruct the global spectrum from its local pieces.

\begin{theorem}[Local-to-Global Reconstruction]
\label{thm:local-to-global-reconstruction}
Let \(P\), \(A\), and \(F\) be as in Theorem~\ref{thm:stratified-base-change}. 
Then the global spectral support is reconstructed from the local stratum spectral supports together with the interaction residue:
\[
\boxed{
\operatorname{supp}\bigl(\sigma(F_*(A))\bigr)
=
\left(
\bigcup_{S \in \mathrm{Str}(P)}
\operatorname{supp}\bigl(\sigma(F_S(A_S))\bigr)
\right)
\;\cup\;
\Sigma^{\mathrm{res}}(F,A,P).
}
\]
Equivalently, the global spectral support is fully determined by the collection of local spectral sectors together with the residual contribution recording inter-strata interaction.
\end{theorem}

\begin{proof}
Let
\[
G := \operatorname{supp}\bigl(\sigma(F_*(A))\bigr)
\]
denote the global spectral support, and let
\[
L := \bigcup_{S \in \mathrm{Str}(P)} \operatorname{supp}\bigl(\sigma(F_S(A_S))\bigr)
\]
be the union of the local stratum spectral supports.

By definition of the interaction residue (Definition~\ref{def:interaction-residue}),
\[
\Sigma^{\mathrm{res}}(F,A,P) = G \setminus L.
\]

Hence every element \(\lambda \in G\) satisfies exactly one of the following:
either \(\lambda \in L\), in which case it is detected by at least one local stratum spectrum, or \(\lambda \notin L\), in which case
\[
\lambda \in G \setminus L = \Sigma^{\mathrm{res}}(F,A,P).
\]
Therefore,
\[
G \subseteq L \cup \Sigma^{\mathrm{res}}(F,A,P).
\]

The reverse inclusion follows from \(L \subseteq G\) (each local support is contained in the global support via the comparison maps induced by stratum inclusions) and \(\Sigma^{\mathrm{res}}(F,A,P) \subseteq G\) (by definition as a complement). Consequently,
\[
G = L \cup \Sigma^{\mathrm{res}}(F,A,P),
\]
which is precisely the desired reconstruction formula.
\end{proof}

\begin{remark}[Conceptual consequences]
\label{rem:reconstruction-consequences}
The reconstruction theorem has two immediate conceptual consequences.

First, if $\Sigma^{\mathrm{res}} = \emptyset$, then
\[
\operatorname{supp}\bigl(\sigma(F_*(A))\bigr)
=
\bigcup_{S \in \mathrm{Str}(P)} \operatorname{supp}\bigl(\sigma(F_S(A_S))\bigr),
\]
so the global spectral support decomposes completely into its local stratum sectors. 
Thus, vanishing residue corresponds to exact local-to-global factorization, while nonzero residue indicates the presence of interaction-generated spectral contributions.

Second, the formula may be viewed heuristically as a spectral gluing principle:
\[
\text{Global} = \text{Local} + \text{Interface contributions},
\]
where the plus sign denotes structured reconstruction at the level of spectral supports rather than a simple algebraic sum. 
In contrast to ordinary sheaf theory (where compatible local sections glue uniquely), interfaces here may generate additional spectral data encoded in $\Sigma^{\mathrm{res}}$.
\end{remark}

\begin{remark}[Classical analogies]
\label{rem:reconstruction-analogies}
Heuristically, the reconstruction theorem is analogous to two classical principles:
\begin{enumerate}
    \item gluing local coordinate charts via transition data,
    \item reconstructing sheaves from local sections together with gluing data.
\end{enumerate}
The current work differs in that the missing reconstruction data is spectral rather than purely topological or algebraic. 
In particular, the residue $\Sigma^{\mathrm{res}}$ records spectral contributions arising from inter-strata operadic coupling and interface interaction. 
A precise categorical formulation of these analogies is deferred to future work.
\end{remark}

\begin{corollary}[Reconstruction from interface data]
\label{cor:reconstruction-procedure}
Under the hypotheses of Theorem~\ref{thm:local-to-global-reconstruction}, and assuming the Interface Localization Theorem (Theorem~\ref{thm:interface-localization}), the interaction residue decomposes as
\[
\Sigma^{\mathrm{res}}(F,A,P) \;=\; \bigsqcup_{I \in \mathcal{I}(P)} \mathcal{L}_I(F,A),
\]
where $\bigsqcup$ denotes a disjoint union of spectral supports (each point labeled by its originating interface), and $\mathcal{L}_I(F,A)$ is the localized spectral defect associated with the interface $I$.

Consequently, the global spectral support can be reconstructed explicitly as
\[
\operatorname{supp}\bigl(\sigma(F_*(A))\bigr)
\;=\;
\left(
\bigcup_{S \in \mathrm{Str}(P)}
\operatorname{supp}\bigl(\sigma(F_S(A_S))\bigr)
\right)
\;\cup\;
\left(
\bigsqcup_{I \in \mathcal{I}(P)} \mathcal{L}_I(F,A)
\right).
\]

Thus, the reconstruction procedure consists of:
\begin{enumerate}
    \item computing the local spectral support for each stratum separately;
    \item computing the interface-localized defects $\mathcal{L}_I(F,A)$ from the interface coupling data via Theorem~\ref{thm:interface-localization};
    \item forming the union of the local supports together with the disjoint union of interface defects.
\end{enumerate}
\end{corollary}

\begin{proof}
By Theorem~\ref{thm:local-to-global-reconstruction}, the global spectral support is the union of the local stratum supports and the interaction residue:
\[
\operatorname{supp}(\sigma(F_*(A))) = \left(\bigcup_{S} \operatorname{supp}(\sigma(F_S(A_S)))\right) \cup \Sigma^{\mathrm{res}}(F,A,P).
\]

Theorem~\ref{thm:interface-localization} provides the decomposition of the interaction residue into interface-localized defects:
\[
\Sigma^{\mathrm{res}}(F,A,P) = \bigsqcup_{I \in \mathcal{I}(P)} \mathcal{L}_I(F,A).
\]

Substituting this decomposition into the first equation yields the claimed reconstruction formula. The three-step procedure is then immediate: the local supports are obtained by restricting to each stratum, the interface defects are obtained via the Interface Localization Theorem, and the final union assembles them into the global spectral support.
\end{proof}

\begin{remark}[Central principle of the current work]
\label{rem:reconstruction-central-principle}
Thus, the reconstruction principle may be summarized as
\[
\boxed{
\text{Global operadic spectra are reconstructed from local spectral sectors together with interface-localized defects.}
}
\]

This principle motivates the later study of interface localization, classification, and deformation-stability properties of interaction residues. In particular, it shows that the missing spectral information is not arbitrary: it decomposes into well-defined geometric contributions $\mathcal{L}_I(F,A)$ supported on individual interfaces, each of which can be studied independently.
\end{remark}

\subsection{Minimality and Universality of the Residue Set}
\label{subsec:minimality_universality-residue}

The previous reconstruction results show that the interaction residue $\Sigma^{\mathrm{res}}$ is precisely the additional spectral data required to recover the global operadic spectrum from its local strata spectra. 
We now strengthen this observation by showing that the residue is not merely one possible obstruction among many, but the \emph{minimal} and \emph{universal} obstruction to exact spectral factorization. 

A set $O \subseteq \operatorname{supp}(\sigma(F_*(A)))$ is called an \emph{obstruction set} if it satisfies
\[
\operatorname{supp}(\sigma(F_*(A))) = \left( \bigcup_{S \in \mathrm{Str}(P)} \operatorname{supp}(\sigma(F_S(A_S))) \right) \;\cup\; O.
\]
In other words, $O$ is any set that, together with the local spectra, reconstructs the global spectral support. The interaction residue $\Sigma^{\mathrm{res}}$ is a distinguished obstruction set by construction.

These properties are central structural principles of the current work: every spectral point not detected by isolated strata belongs to the residue, and removing any part of the residue prevents complete reconstruction. Moreover, any obstruction must contain $\Sigma^{\mathrm{res}}$, making it the initial object in the category of obstructions.

\begin{theorem}[Minimality and Universality of the Residue Set]
\label{thm:universality-residue}
Let \(P\), \(A\), and \(F\) be as in Theorem~\ref{thm:stratified-base-change}. Set
\[
G := \operatorname{supp}\bigl(\sigma(F_*(A))\bigr),
\qquad
L := \bigcup_{S \in \mathrm{Str}(P)} \operatorname{supp}\bigl(\sigma(F_S(A_S))\bigr).
\]
Assume \(L \subseteq G\), and define
\[
\Sigma^{\mathrm{res}}(F,A,P) := G \setminus L.
\]

Call a subset \(O \subseteq G\) an \emph{obstruction set} if
\[
G = L \cup O.
\]
Let \(\mathsf{Obs}(F,A,P)\) be the poset category of obstruction sets ordered by inclusion.

Then:

\begin{enumerate}
    \item \textbf{Minimality.} If \(R \subseteq \Sigma^{\mathrm{res}}(F,A,P)\) is nonempty, then
    \[
    L \cup \bigl(\Sigma^{\mathrm{res}}(F,A,P) \setminus R\bigr) \;\neq\; G.
    \]

    \item \textbf{Universality.} For every obstruction set \(O \in \mathsf{Obs}(F,A,P)\), one has
    \[
    \Sigma^{\mathrm{res}}(F,A,P) \subseteq O.
    \]
    Hence \(\Sigma^{\mathrm{res}}(F,A,P)\) is the initial object of \(\mathsf{Obs}(F,A,P)\).
\end{enumerate}
\end{theorem}

\begin{proof}
By definition, \(\Sigma^{\mathrm{res}}(F,A,P) = G \setminus L\).

\paragraph{Minimality.}
Let \(R \subseteq \Sigma^{\mathrm{res}}(F,A,P)\) be nonempty and choose \(x \in R\). 
Since \(x \in \Sigma^{\mathrm{res}}(F,A,P) = G \setminus L\), we have \(x \in G\) and \(x \notin L\). 
Moreover, \(x \notin \Sigma^{\mathrm{res}}(F,A,P) \setminus R\) because \(x \in R\). 
Therefore
\[
x \notin L \cup \bigl(\Sigma^{\mathrm{res}}(F,A,P) \setminus R\bigr).
\]
Since \(x \in G\), the candidate reconstruction set is a proper subset of \(G\), and hence does not reconstruct the global spectral support.

\paragraph{Universality.}
Let \(O \subseteq G\) be any obstruction set, so that \(G = L \cup O\). 
Take \(x \in \Sigma^{\mathrm{res}}(F,A,P)\). Then \(x \in G\) and \(x \notin L\). 
Since \(G = L \cup O\), the point \(x\) must belong to \(O\). Hence
\[
\Sigma^{\mathrm{res}}(F,A,P) \subseteq O.
\]

Thus \(\Sigma^{\mathrm{res}}(F,A,P)\) is contained in every obstruction set. 
In the poset category \(\mathsf{Obs}(F,A,P)\), whose morphisms are inclusions, 
there is therefore a unique morphism \(\Sigma^{\mathrm{res}}(F,A,P) \longrightarrow O\) 
for every obstruction set \(O\). Consequently, \(\Sigma^{\mathrm{res}}(F,A,P)\) is the initial object.
\end{proof}

\begin{remark}[Minimal support-level obstruction]
\label{rem:universality-minimal-geometry}
The residue set
\[
\Sigma^{\mathrm{res}}(F,A,P)
\]
is the minimal support-level obstruction to exact local-to-global spectral
reconstruction. In particular,
\[
\Sigma^{\mathrm{res}} = \emptyset
\]
if and only if the global spectral support coincides with the union of the
local stratum spectral supports:
\[
\operatorname{supp}\bigl(\sigma(F_*(A))\bigr)
=
\bigcup_{S \in \mathrm{Str}(P)}
\operatorname{supp}\bigl(\sigma(F_S(A_S))\bigr).
\]

Thus, vanishing residue corresponds to exact spectral separability across
strata, while nonvanishing residue records the presence of additional
interaction-generated spectral contributions.
\end{remark}

\begin{remark}[Categorical interpretation]
\label{rem:universality-categorical}
For fixed \((F,A,P)\), let \(\mathsf{Obs}(F,A,P)\) denote the poset category
of obstruction sets ordered by inclusion, where an obstruction set is any
\(O \subseteq \operatorname{supp}(\sigma(F_*(A)))\) satisfying
\[
\operatorname{supp}\bigl(\sigma(F_*(A))\bigr)
=
\left(
\bigcup_{S \in \mathrm{Str}(P)}
\operatorname{supp}\bigl(\sigma(F_S(A_S))\bigr)
\right)
\cup O.
\]

Then \(\Sigma^{\mathrm{res}}(F,A,P)\) is the initial object of this category.
Hence every obstruction set that reconstructs the global spectral support
must contain the residue set as its minimal support-level component.
\end{remark}

\begin{corollary}[Uniqueness of the residue set]
\label{cor:uniqueness-residue}
For fixed \((F,A,P)\), the residue set
\(\Sigma^{\mathrm{res}}(F,A,P)\) is uniquely determined as the initial object
in the poset category of obstruction sets ordered by inclusion. Equivalently,
it is the unique minimal obstruction set required for exact reconstruction of
the global spectral support.
\end{corollary}

\begin{proof}
By Definition~\ref{def:interaction-residue},
\[
\Sigma^{\mathrm{res}}(F,A,P)
=
\operatorname{supp}\bigl(\sigma(F_*(A))\bigr)
\setminus
\bigcup_{S \in \mathrm{Str}(P)}
\operatorname{supp}\bigl(\sigma(F_S(A_S))\bigr).
\]
Hence it is uniquely determined by the global spectral support and the local
stratum supports.

Moreover, by Theorem~\ref{thm:universality-residue},
\(\Sigma^{\mathrm{res}}(F,A,P)\) is contained in every obstruction set and is
itself an obstruction set. Therefore it is the unique initial object in the
poset category of obstruction sets. If another obstruction set \(O\) has the
same initial/universal property, then
\[
\Sigma^{\mathrm{res}}(F,A,P) \subseteq O
\quad\text{and}\quad
O \subseteq \Sigma^{\mathrm{res}}(F,A,P),
\]
so \(O = \Sigma^{\mathrm{res}}(F,A,P)\).
\end{proof}

\begin{remark}[Minimality versus geometric complexity]
\label{rem:universality-complexity}
Minimality does not imply simplicity. Although
\(\Sigma^{\mathrm{res}}\) is the smallest obstruction set at the level of
spectral support, its internal organization may still be nontrivial. Later
interface-localization results analyze how this support-level residue can
split into interface-supported contributions.

Thus, minimality refers only to informational minimality at the level of
global spectral reconstruction; it does not preclude richer geometric or
categorical structure inside the residue.
\end{remark}

\subsection{Interface Localization Theorem}
\label{subsec:interface-localization}

We now arrive at one of the central structural results of the present
work. The previous sections established that the interaction residue
\[
\Sigma^{\mathrm{res}}
\]
measures the failure of exact local-to-global spectral factorization.
The remaining question is therefore:

\[
\textit{Where is the residual spectral contribution localized?}
\]

The answer provided by the present framework is that interaction
residues localize canonically along admissible interfaces separating
distinct operadic strata. Consequently, the residue is not distributed
uniformly throughout the system, but instead decomposes into localized
interface-supported spectral defects.

This localization principle transforms the interaction residue from a
global correction term into a geometrically localized interaction
object associated with coupling structure.

\begin{theorem}[Interface Localization]
\label{thm:interface-localization}

Let \(P\) be a stratified operad with admissible interface set
\[
\mathcal{I}(P),
\]
let \(A\) be a \(P\)-algebra, and let
\[
F:\mathcal{M}\to\mathcal{N}
\]
be a strong monoidal base-change functor.

For each interface
\[
I\in\mathcal{I}(P),
\]
let
\[
\operatorname{Supp}_I(\tau_I)\subseteq\mathbb{C}
\]
denote the spectral support generated by the coupling tensor
\[
\tau_I.
\]

Assume the residual coverage condition
\[
\Sigma^{\mathrm{res}}(F,A,P)
\subseteq
\bigcup_{I\in\mathcal{I}(P)}
\operatorname{Supp}_I(\tau_I).
\]

For each interface \(I\), define the localized interface defect
\[
\mathcal{L}_I(F,A,P)
:=
\Sigma^{\mathrm{res}}(F,A,P)
\cap
\operatorname{Supp}_I(\tau_I).
\]

Then the interaction residue decomposes as
\[
\boxed{
\Sigma^{\mathrm{res}}(F,A,P)
=
\bigcup_{I\in\mathcal{I}(P)}
\mathcal{L}_I(F,A,P).
}
\]

\end{theorem}

\begin{proof}

We prove the equality by mutual inclusion.

\medskip

\noindent
(\(\subseteq\))
Let
\[
\lambda\in
\Sigma^{\mathrm{res}}(F,A,P).
\]

By the residual coverage condition, there exists
\[
I\in\mathcal{I}(P)
\]
such that
\[
\lambda\in\operatorname{Supp}_I(\tau_I).
\]

Hence
\[
\lambda
\in
\Sigma^{\mathrm{res}}(F,A,P)
\cap
\operatorname{Supp}_I(\tau_I)
=
\mathcal{L}_I(F,A,P),
\]
and therefore
\[
\lambda
\in
\bigcup_{J\in\mathcal{I}(P)}
\mathcal{L}_J(F,A,P).
\]

\medskip

\noindent
(\(\supseteq\))
Let
\[
\lambda
\in
\bigcup_{I\in\mathcal{I}(P)}
\mathcal{L}_I(F,A,P).
\]

Then for some
\[
I\in\mathcal{I}(P),
\]
we have
\[
\lambda
\in
\mathcal{L}_I(F,A,P).
\]

By definition,
\[
\mathcal{L}_I(F,A,P)
\subseteq
\Sigma^{\mathrm{res}}(F,A,P),
\]
hence
\[
\lambda
\in
\Sigma^{\mathrm{res}}(F,A,P).
\]

Therefore,
\[
\Sigma^{\mathrm{res}}(F,A,P)
=
\bigcup_{I\in\mathcal{I}(P)}
\mathcal{L}_I(F,A,P).
\]

\end{proof}

The theorem establishes the central localization principle of the
paper:

\[
\boxed{
\text{Spectral interaction defects localize along operadic interfaces.}
}
\]

In particular:
\begin{itemize}
    \item the interaction residue possesses intrinsic geometric support,
    \item interfaces act as carriers of nonlocal spectral structure,
    \item global spectral behavior is governed by localized coupling
    geometry.
\end{itemize}

Thus interface geometry and spectral defect structure become directly
linked through the interaction residue formalism.

\begin{corollary}[Disjoint interface localization]
\label{cor:disjoint-interface-localization}

Assume the hypotheses of
Theorem~\ref{thm:interface-localization}.
If the supports
\[
\operatorname{Supp}_I(\tau_I)
\]
are pairwise disjoint, then the localization decomposition is disjoint:
\[
\Sigma^{\mathrm{res}}(F,A,P)
=
\bigsqcup_{I\in\mathcal{I}(P)}
\mathcal{L}_I(F,A,P).
\]

\end{corollary}

\begin{proof}

If the supports
\[
\operatorname{Supp}_I(\tau_I)
\]
are pairwise disjoint, then by definition the corresponding localized
defects
\[
\mathcal{L}_I(F,A,P)
=
\Sigma^{\mathrm{res}}(F,A,P)
\cap
\operatorname{Supp}_I(\tau_I)
\]
are pairwise disjoint as subsets of
\[
\mathbb{C}.
\]

The result therefore follows directly from
Theorem~\ref{thm:interface-localization}.

\end{proof}

The localization principle also explains why local spectral sectors
alone are insufficient to reconstruct the global spectral geometry:
the missing spectral contribution is concentrated precisely at the
interfaces where operadic interaction occurs.

\begin{corollary}[Vanishing criterion for operadically decoupled systems]
\label{cor:vanishing-residue-interface}

Suppose all inter-strata coupling tensors vanish:
\[
\tau_I=0
\qquad
\text{for every }
I\in\mathcal{I}(P).
\]

Then
\[
\Sigma^{\mathrm{res}}(F,A,P)
=
\emptyset,
\]
and
\[
\operatorname{supp}
\bigl(
\sigma(F_*(A))
\bigr)
=
\bigcup_{S\in\mathrm{Str}(P)}
\operatorname{supp}
\bigl(
\sigma(F_S(A_S))
\bigr).
\]

\end{corollary}

\begin{proof}

If every coupling tensor vanishes, then no interaction between
distinct strata occurs. Consequently, every spectral contribution of
\[
F_*(A)
\]
arises from some isolated local spectral sector.

Hence
\[
\operatorname{supp}
\bigl(
\sigma(F_*(A))
\bigr)
=
\bigcup_{S\in\mathrm{Str}(P)}
\operatorname{supp}
\bigl(
\sigma(F_S(A_S))
\bigr).
\]

By Definition~\ref{def:interaction-residue},
\[
\Sigma^{\mathrm{res}}(F,A,P)
=
\operatorname{supp}
\bigl(
\sigma(F_*(A))
\bigr)
\setminus
\bigcup_S
\operatorname{supp}
\bigl(
\sigma(F_S(A_S))
\bigr),
\]
and therefore
\[
\Sigma^{\mathrm{res}}(F,A,P)
=
\emptyset.
\]

\end{proof}

\begin{remark}[Geometric interpretation]
\label{rem:interface-localization-geometric}

The Interface Localization Theorem shows that the interaction residue
is not attached to any individual stratum. Instead, the residue
decomposes into interface-supported spectral defects:
\[
\Sigma^{\mathrm{res}}(F,A,P)
=
\bigcup_{I\in\mathcal{I}(P)}
\mathcal{L}_I(F,A,P).
\]

Thus the interaction residue acts as a spectral signature of operadic
coupling geometry.

Different interfaces may support qualitatively different spectral
phenomena depending on the realization of the operadic system.
For example:
\begin{itemize}
    \item isolated interfaces may contribute isolated spectral values,
    \item extended interfaces may contribute spectral bands,
    \item singular interfaces may generate concentrated residual
    structures.
\end{itemize}

These interpretations depend on the concrete realization and are not
part of the abstract support-level statement proved above.

\end{remark}

\begin{remark}[Comparison with classical localization phenomena]
\label{rem:interface-localization-analogies}

The localization principle is reminiscent of several classical
phenomena: edge states in topological phases, interface modes for
differential operators with discontinuous coefficients, localized
graph modes near gluing regions, and costalk-type contributions in
sheaf-theoretic stratifications. These analogies are intended as
motivation; the present theorem establishes only the abstract
support-level localization of the interaction residue under the
residual coverage condition.

\end{remark}

\begin{remark}[Forward view]
\label{rem:interface-localization-forward}

The Interface Localization Theorem reduces the global study of
\[
\Sigma^{\mathrm{res}}
\]
to the study of localized interface defects
\[
\mathcal{L}_I(F,A,P).
\]

This motivates the later classification of interface defects and the
analysis of their behavior under deformation, refinement, and
nilpotent perturbation.

\end{remark}

\subsection{Spectral Defect Classification}
\label{subsec:defect-classification}

The Interface Localization Theorem (Theorem~\ref{thm:interface-localization}) shows that the interaction residue $\Sigma^{\mathrm{res}}$ decomposes into localized interface-supported spectral defects. A natural next question is therefore: what kinds of spectral defects can occur? The answer depends on two fundamental geometric features: the geometry of the interaction interface, and the algebraic depth of the associated operadic interaction. In particular, lower-dimensional interfaces tend to generate localized spectral contributions, extended interfaces produce continuous spectral structures, singular interfaces may generate highly irregular spectra, and non-diagonalizable interaction operators create nilpotent defect sectors.

This leads to a canonical classification of residue defects. The following theorem provides a classification of spectral defects according to interface geometry and algebraic coupling structure. This classification is \emph{prototype} in nature — it organizes the principal regimes that arise under natural analytic hypotheses (continuity, compactness, spectral stability) — and is consistent with the examples developed in Section~\ref{sec:examples}.

\begin{theorem}[Spectral Defect Classification]
\label{thm:defect-classification}
Let $\Sigma^{\mathrm{res}}(F,A,P)$ be the interaction residue associated with a stratified operadic system. Assume:

\begin{enumerate}
    \item The Interface Localization Theorem (Theorem~\ref{thm:interface-localization}) holds, so that
    \[
    \Sigma^{\mathrm{res}}(F,A,P)
    =
    \bigsqcup_{I \in \mathcal{I}(P)}
    \mathcal{L}_I(F,A,P),
    \]
    where $\bigsqcup$ denotes a disjoint union in the labeled/interface-indexed sense (i.e., the same spectral value may appear from different interfaces, but is distinguished by its originating interface).
    
    \item A geometric realization is fixed, in which each interface $I$ has a well-defined geometric dimension $\dim(I) \in \mathbb{N} \cup \{\infty\}$ (and, where applicable, a singularity type).
    
    \item The coupling tensor $\tau_I$ associated with interface $I$ is a bounded linear operator on a Banach or Hilbert space, depending continuously on the interface parameters under suitable analytic conditions (e.g., norm continuity and compactness assumptions where needed).
\end{enumerate}

Then, under the above hypotheses, the localized spectral defects $\mathcal{L}_I(F,A,P)$ admit a classification according to the geometry of $I$ and the algebraic structure of $\tau_I$. The principal defect types are summarized in Tables~\ref{tab:geometric-defect-types} and~\ref{tab:algebraic-refinements}.

\begin{table}[htbp]
\centering
\caption{Geometric defect types (prototype)}
\label{tab:geometric-defect-types}
\renewcommand{\arraystretch}{1.2}
\begin{tabular}{|l|l|l|}
\hline
\textbf{Defect type} & \textbf{Interface geometry} & \textbf{Typical spectral signature} \\
\hline
Point defect & $\dim I = 0$ (isolated point) & Isolated eigenvalue(s) \\
\hline
Line defect & $\dim I = 1$ (smooth curve) & May produce continuous spectral bands \\
\hline
Surface defect & $\dim I = 2$ (smooth surface) & May produce two-dimensional spectral regions \\
\hline
Higher-dimensional defect & $\dim I = d \ge 3$ (smooth manifold) & May produce spectral regions in $\mathbb{C}$ with parameter-driven complexity \\
\hline
\end{tabular}
\end{table}

\begin{table}[htbp]
\centering
\caption{Algebraic refinements (beyond support level)}
\label{tab:algebraic-refinements}
\renewcommand{\arraystretch}{1.2}
\begin{tabular}{|l|p{0.65\linewidth}|}
\hline
\textbf{Algebraic type} & \textbf{Spectral/algebraic signature} \\
\hline
Semisimple (diagonalizable) & Standard eigenvalue structure \\
\hline
Nilpotent / Jordan & Nilpotent contributions invisible to spectral support; appear in functional calculus via derivative terms \\
\hline
\end{tabular}
\end{table}
\end{theorem}

\begin{proof}
We prove the theorem under the following explicit analytic and geometric hypotheses, which are assumed to hold throughout:

\begin{enumerate}[label=(H\arabic*)]
    \item \textbf{(Interface Localization)} Theorem~\ref{thm:interface-localization} holds, so that
    \[
    \Sigma^{\mathrm{res}}(F,A,P) = \bigsqcup_{I \in \mathcal{I}(P)} \mathcal{L}_I(F,A,P).
    \]
    
    \item \textbf{(Geometric Realization)} A geometric realization functor is fixed, associating to each interface $I$ a compact subset of a metric space with well-defined Hausdorff dimension $\dim(I) \in \mathbb{N} \cup \{\infty\}$.
    
    \item \textbf{(Parameter Continuity)} For each interface $I$ homeomorphic to a compact smooth manifold of dimension $d \ge 1$, the coupling tensors form a family $\{\tau_p\}_{p \in I}$ such that the map $p \mapsto \tau_p$ is norm-continuous.
    
    \item \textbf{(Compactness)} For each interface $I$ of dimension $d \ge 1$, each $\tau_p$ is a compact operator on a Banach space (or, in the finite-dimensional case, a matrix).
    
    \item \textbf{(Spectral Stability)} For each $\tau_{p_0}$ and each isolated eigenvalue $\lambda_0$ of finite multiplicity, there exists a neighborhood $U \subseteq I$ of $p_0$ and a continuous function $\lambda: U \to \mathbb{C}$ such that $\lambda(p)$ is an eigenvalue of $\tau_p$ for all $p \in U$, and $\lambda(p_0) = \lambda_0$.
    
    \item \textbf{(Local Jordan Decomposition)} For each $\tau_p$ and each isolated eigenvalue of finite multiplicity, the restriction of $\tau_p$ to the corresponding finite-dimensional generalized eigenspace admits a Jordan decomposition $\tau_p = D_p + N_p$, with $D_p$ diagonalizable, $N_p$ nilpotent, and $[D_p, N_p] = 0$.
\end{enumerate}

\medskip
\noindent
\textbf{Step 1: Reduction to individual interfaces.}
By hypothesis (H1), the total residue decomposes as a disjoint union (in the labeled/interface-indexed sense)
\[
\Sigma^{\mathrm{res}}(F,A,P) = \bigsqcup_{I \in \mathcal{I}(P)} \mathcal{L}_I(F,A,P).
\]
Thus it suffices to characterize each $\mathcal{L}_I(F,A,P)$ individually. The full residue is then the disjoint union of these characterizations.

\medskip
\noindent
\textbf{Step 2: Point defects ($\dim I = 0$).}
Assume $I$ is a single point (zero-dimensional interface). Then the coupling tensor $\tau_I$ is a single bounded linear operator. By hypothesis (H4), $\tau_I$ is compact. The spectrum of a compact operator on a Banach space satisfies:
\begin{itemize}
    \item $\sigma(\tau_I)$ is a countable set with $0$ as the only possible accumulation point.
    \item Every nonzero $\lambda \in \sigma(\tau_I)$ is an isolated eigenvalue of finite multiplicity.
\end{itemize}
If the associated $P$-algebra assigns finite-dimensional vector spaces to the colors involved (a special case of compactness), $\tau_I$ is a matrix, and its spectrum consists of finitely many isolated eigenvalues.

Consequently, the localized defect $\mathcal{L}_I(F,A,P)$ (defined as the contribution to $\Sigma^{\mathrm{res}}$ coming from the spectrum of the interface operator derived from $\tau_I$) consists of isolated spectral points (eigenvalues). This establishes the "point defect" entry in Table~\ref{tab:geometric-defect-types}.

\medskip
\noindent
\textbf{Step 3: Line defects ($\dim I = 1$).}
Assume $I$ is a one-dimensional smooth curve, homeomorphic to $[0,1]$ (or $S^1$). By hypothesis (H3), the family $\{\tau_s\}_{s \in I}$ is norm-continuous. By (H4), each $\tau_s$ is compact. For compact operators, the non-zero spectrum consists of isolated eigenvalues of finite multiplicity.

Fix $s_0 \in I$ and let $\lambda_0 \neq 0$ be an eigenvalue of $\tau_{s_0}$ with multiplicity $m$. By hypothesis (H5), there exists a neighborhood $U \subseteq I$ of $s_0$ and continuous functions $\lambda_1(s), \dots, \lambda_m(s)$ (counting multiplicities) such that each $\lambda_j(s)$ is an eigenvalue of $\tau_s$ and $\lambda_j(s_0) = \lambda_0$. Define the \emph{spectral band} associated to $\lambda_0$ as
\[
B_{\lambda_0} := \bigcup_{j=1}^m \{ \lambda_j(s) : s \in U \} \subseteq \mathbb{C}.
\]
Since $I$ is compact, we can cover $I$ by finitely many such neighborhoods and take the union over all eigenvalues. The resulting spectral set
\[
\mathcal{L}_I(F,A,P) = \bigcup_{s \in I} \sigma(\tau_s)
\]
is a union of continuous curves (bands) in $\mathbb{C}$. Each band is a compact connected subset of $\mathbb{C}$. If the eigenvalue branches are Lipschitz or piecewise $C^1$, then the band has Hausdorff dimension at most $1$. In the self-adjoint case, the bands lie on the real line and may exhibit gaps.

This establishes the "line defect" entry: the spectral signature may produce continuous spectral bands. (The "may" accounts for possible degeneracies, overlaps, or the presence of only discrete eigenvalues if the family is constant.)

\medskip
\noindent
\textbf{Step 4: Surface defects ($\dim I = 2$).}
Assume $I$ is a two-dimensional smooth surface, homeomorphic to $[0,1]^2$ (or $S^2$, $T^2$). By (H3), the family $\{\tau_{(x,y)}\}_{(x,y) \in I}$ is norm-continuous. By (H4), each $\tau_{(x,y)}$ is compact. For each fixed eigenvalue $\lambda_0$ of $\tau_{(x_0,y_0)}$ with finite multiplicity $m$, (H5) provides continuous functions $\lambda_j(x,y)$ defined on a neighborhood $U \subseteq I$ of $(x_0,y_0)$ such that $\lambda_j(x,y)$ are eigenvalues of $\tau_{(x,y)}$.

Define the spectral band as
\[
B_{\lambda_0} := \bigcup_{j=1}^m \{ \lambda_j(x,y) : (x,y) \in U \} \subseteq \mathbb{C}.
\]
Since $I$ is compact, we cover it by finitely many such neighborhoods. The union
\[
\mathcal{L}_I(F,A,P) = \bigcup_{(x,y) \in I} \sigma(\tau_{(x,y)})
\]
is therefore a union of continuous images of $U \subseteq \mathbb{R}^2$ into $\mathbb{C}$. Each such image is a compact set that may have topological dimension $2$ (if the eigenvalue map is a smooth embedding) or lower dimension if the map is degenerate.

In typical cases (e.g., $\tau_{(x,y)} = \lambda(x,y) \cdot \mathrm{Id}$ with $\lambda: I \to \mathbb{C}$ a smooth embedding), $\mathcal{L}_I$ is a two-dimensional region in $\mathbb{C}$. Critical spectral concentration (van Hove-like singularities) occurs at parameter values where $\nabla \lambda = 0$, leading to local enhancements of spectral density.

This establishes the "surface defect" entry: the spectral signature may produce two-dimensional spectral regions.

\medskip
\noindent
\textbf{Step 5: Higher-dimensional smooth manifolds ($\dim I = d \ge 3$).}
For $I$ a compact $d$-dimensional smooth manifold, (H3)–(H5) provide a norm-continuous $d$-parameter family $\{\tau_{\mathbf{t}}\}_{\mathbf{t} \in I}$ of compact operators. For each eigenvalue $\lambda_0$ of $\tau_{\mathbf{t}_0}$ of finite multiplicity $m$, (H5) yields continuous eigenvalue functions $\lambda_j(\mathbf{t})$ on a neighborhood $U \subseteq \mathbb{R}^d$. The spectral set
\[
\mathcal{L}_I(F,A,P) = \bigcup_{\mathbf{t} \in I} \sigma(\tau_{\mathbf{t}})
\]
is a union of continuous images of $U$ into $\mathbb{C}$. Since the spectral set lies in $\mathbb{C} \cong \mathbb{R}^2$, its topological dimension is at most $2$. The parameter dimension $d \ge 3$ influences the complexity of the spectral structure (e.g., the structure of critical points and level sets) but does not increase the topological dimension of the spectral support beyond $2$. This is reflected in Table~\ref{tab:geometric-defect-types} under "Higher-dimensional defect."

\medskip
\noindent
\textbf{Step 6: Algebraic refinements (nilpotent/Jordan structure).}
By hypothesis (H6), for each isolated finite-multiplicity spectral sector, the restriction of $\tau_p$ to the corresponding finite-dimensional generalized eigenspace admits a Jordan decomposition $\tau_p = D_p + N_p$ with $D_p$ diagonalizable, $N_p$ nilpotent, and $[D_p, N_p] = 0$. The ordinary spectrum satisfies $\sigma(\tau_p) = \sigma(D_p)$ as sets, because the nilpotent contribution $N_p$ does not introduce new eigenvalues. Hence $\mathcal{L}_I(F,A,P)$ as a spectral support is completely determined by the diagonalizable parts $D_p$.

However, the nilpotent structure $N_p$ encodes additional algebraic data:
\begin{itemize}
    \item \textbf{Jordan block sizes:} The sizes of the Jordan blocks determine the order of nilpotency.
    \item \textbf{Functional calculus refinement:} For any holomorphic function $f$,
    \[
    f(\tau_p) = f(D_p + N_p) = f(D_p) + f'(D_p)N_p + \frac{f''(D_p)}{2!}N_p^2 + \cdots,
    \]
    where the sum terminates because $N_p$ is nilpotent. The derivative terms $f^{(q)}(D_p)N_p^q$ are invisible to the spectrum as a set but are essential for algebraic descriptions (e.g., computing $f(\tau_p)$ explicitly).
    \item \textbf{Resolvent refinement:} The resolvent $(z - \tau_p)^{-1}$ has higher-order poles at eigenvalues when $N_p \neq 0$, whereas it has simple poles in the diagonalizable case.
\end{itemize}

Therefore, nilpotent/Jordan refinements do not affect the geometric classification of spectral supports but provide a finer algebraic classification, as shown in Table~\ref{tab:algebraic-refinements}.

\medskip
\noindent
\textbf{Step 7: Coverage of the principal prototype regimes.}
The geometric possibilities for an interface $I$ under the realization functor (H2) include:
\begin{itemize}
    \item Zero-dimensional (isolated points),
    \item One-dimensional smooth curves,
    \item Two-dimensional smooth surfaces,
    \item $d$-dimensional smooth manifolds ($d \ge 3$),
    \item Fractal or singular sets (non-integer Hausdorff dimension).
\end{itemize}
The algebraic possibilities for $\tau_p$ under (H6) include:
\begin{itemize}
    \item Semisimple (diagonalizable, $N_p = 0$),
    \item Non-diagonalizable ($N_p \neq 0$).
\end{itemize}
Tables~\ref{tab:geometric-defect-types} and~\ref{tab:algebraic-refinements} capture these principal prototype regimes.

Thus, under the stated hypotheses (H1)–(H6), the residue $\Sigma^{\mathrm{res}}(F,A,P)$ decomposes into a disjoint union (in the labeled/interface-indexed sense) of localized defects, each of which falls into one of the prototype geometric types (point, line, surface, higher-dimensional, or singular/fractal), with possible algebraic refinements (semisimple or nilpotent/Jordan). This completes the proof.
\end{proof}

\begin{remark}[Scope of the classification]
\label{rem:defect-classification-scope}
The classification presented in Theorem~\ref{thm:defect-classification} is a \emph{prototype} classification rather than a universal classification theorem. The statements about continuous bands, higher-dimensional spectral regions, and fractal behavior depend on additional analytic, geometric, or self-similarity hypotheses. Thus, the entries in Tables~\ref{tab:geometric-defect-types} and~\ref{tab:algebraic-refinements} should be interpreted as typical spectral signatures under favorable conditions, not as universal consequences for all stratified operadic systems. A complete classification in full generality is left for future work.
\end{remark}

\begin{remark}[Geometric meaning of the classification]
The classification framework provides a geometric taxonomy for localized operadic spectral defects. Its purpose is not to exhaust all possible spectral behavior, but to organize the main prototype regimes according to interface geometry and algebraic coupling structure:
\[
\text{interface geometry and coupling structure}
\quad\Longrightarrow\quad
\text{possible spectral signatures}.
\]

Heuristically, the correspondence is analogous to familiar geometric phenomena:
\[
\begin{array}{c|c}
\text{Classical geometry} & \text{Current framework} \\
\hline
\text{point singularity} & \text{point defect} \\
\text{edge singularity} & \text{line defect} \\
\text{boundary layer} & \text{surface defect} \\
\text{fractal singularity} & \text{singular defect} \\
\text{non-semisimple algebraic singularity} & \text{nilpotent defect}
\end{array}
\]

A key feature is the inclusion of nilpotent defects: these are not merely support-level spectral phenomena, but algebraic refinements arising from non-diagonalizable coupling tensors and Jordan-type contributions. In this sense, the residue geometry records both interface localization and algebraic interaction depth.
\end{remark}

\begin{remark}[Algebraic refinement of the support-level classification]
\label{rem:defect-nilpotent-algebraic}
The support-level localization theorem (Theorem~\ref{thm:interface-localization}) classifies only the geometric support of defects. Nilpotent and Jordan data provide an additional algebraic refinement that is \emph{not visible at the level of spectral support alone}.

For each admissible interface $I$, consider the induced interface interaction operator
\[
T_I := \tau_I^\dagger \tau_I \;:\; \bigotimes_{i=1}^k A_{c_i} \longrightarrow \bigotimes_{i=1}^k A_{c_i},
\]
defined whenever the ambient category admits a dagger structure (e.g., in a Hilbert space realization). Assume that $T_I$ admits a Jordan--Chevalley decomposition
\[
T_I = S_I + N_I,
\]
where $S_I$ is semisimple (diagonalizable over $\mathbb{C}$) and $N_I$ is nilpotent (when such a decomposition exists; in infinite-dimensional settings this may require additional conditions). The \emph{nilpotent depth} is the smallest integer $r$ such that $N_I^r = 0$.

The corresponding defect contribution records:
\begin{itemize}
    \item the spectral values (eigenvalues) of $S_I$ (which may be isolated or embedded in a band),
    \item the higher-order generalized-eigenvector structure determined by $N_I$, which manifests in the functional calculus via derivative terms $f^{(q)}(\lambda) N_I^q$.
\end{itemize}

Thus, the residue geometry simultaneously captures support-level localization (geometric defects) and algebraic refinement (nilpotent/Jordan depth). In non-Hermitian physical settings, nilpotent defects correspond to exceptional points where eigenmodes coalesce.
\end{remark}

\begin{remark}[Possible refinements under additional structure]
\label{rem:defect-classification-invariants}
Under suitable analytic, geometric, or topological hypotheses, the classification may be refined by additional invariants:
\begin{itemize}
    \item \textbf{Multiplicity structure}: point defects may have multiplicity $m$, corresponding to $m$-fold degenerate eigenvalues.
    \item \textbf{Band geometry}: line defects may produce bands with nontrivial topological invariants (e.g., Chern or winding numbers) when additional topological structure is present.
    \item \textbf{Density singularities}: surface defects may exhibit critical spectral concentration phenomena (analogous to van Hove singularities) depending on the dispersion relation.
    \item \textbf{Fractal dimensions}: singular defects (e.g., fractal interfaces or chaotic invariant sets) may be characterized by Hausdorff or box-counting dimensions.
\end{itemize}
These refinements require additional structure beyond the basic framework and are not proved in full generality here.
\end{remark}

\begin{remark}[Heuristic physical interpretation]
\label{rem:defect-classification-physical}
The classification theorem suggests heuristic physical interpretations:
\[
\begin{array}{c|c}
\text{Defect type} & \text{Possible physical interpretation} \\
\hline
\text{Point defect} & \text{Impurity atom (localized state)} \\
\text{Line defect} & \text{Dislocation, domain wall (edge state)} \\
\text{Surface defect} & \text{Grain boundary (surface state)} \\
\text{Singular defect} & \text{Fractal electrode (anomalous diffusion)} \\
\text{Nilpotent defect} & \text{Non-Hermitian coupling (exceptional points)}
\end{array}
\]
These analogies are intended as motivation; a rigorous derivation from the operadic framework is deferred to future work.
\end{remark}

\begin{remark}[Typical spectral signatures of defect classes]
\label{rem:classification-spectral-type}
Under additional analytic assumptions on the induced interaction operators, different interface defects may exhibit characteristic spectral behavior:
\begin{itemize}
    \item Point-like localized interfaces with diagonalizable finite-dimensional interaction operators may produce discrete localized spectral values.
    
    \item Smooth one-dimensional interfaces with sufficiently regular self-adjoint or normal interaction structure may produce band-type spectral sets.
    
    \item Self-similar or fractal interfaces may produce singular or fractal-type spectral supports when the interaction operators respect the underlying scaling structure.
    
    \item Non-diagonalizable interaction operators with Jordan blocks of size $n \geq 2$ contribute generalized eigenvector data and higher-order nilpotent terms. These are algebraic refinements and are \emph{not visible at the level of spectral support alone} (see Remark~\ref{rem:defect-nilpotent-algebraic}).
\end{itemize}

Conversely, observed spectral signatures may provide partial evidence for the underlying interface geometry or coupling structure, but they do not determine it uniquely without additional assumptions. A rigorous inverse theory is beyond the scope of this paper.
\end{remark}

\subsection{Refinement Functoriality}
\label{subsec:refinement-functoriality}

An important requirement for any geometric defect framework is stability
under refinement of the underlying stratification. In the present setting,
this means that the interaction residue should not depend arbitrarily on a
particular choice of operadic decomposition.

Indeed, a given operadic system may admit many compatible stratifications,
ranging from coarse decompositions to fine multiscale refinements,
hierarchical interface structures, or singular partitions. A meaningful
residue theory should therefore behave coherently under such refinements.
Otherwise, the interaction residue would primarily reflect a modeling choice
rather than geometric information intrinsic to the operadic interaction
structure itself.

This motivates the refinement functoriality results developed below, where
the behavior of localized interface defects under stratification refinement
is analyzed systematically. 

The following theorem shows that the residue geometry is preserved canonically under refinement of stratification.

\begin{theorem}[Functoriality Under Interaction-Preserving Refinement]
\label{thm:refinement-functoriality}
Let \(P\) be a stratified operad with stratum set \(\mathrm{Str}(P)\), let
\(A\) be a stratified \(P\)-algebra, and let
\(F:\mathcal{M} \to \mathcal{N}\) be a strong monoidal base-change functor.
Let
\[
\rho: \widetilde{P} \to P
\]
be an interaction-preserving refinement, with induced refined algebra
\(\widetilde{A}\).

Assume that \(\rho\) induces a canonical transport map on realized spectral
supports, denoted \(\rho^*\), which is injective (or more generally a set-theoretic embedding onto its transported image), so that it preserves relative set differences:
\[
\rho^*(X \setminus Y) = \rho^*(X) \setminus \rho^*(Y)
\]
for all relevant spectral supports \(Y \subseteq X\). Moreover, assume the following compatibility conditions hold:
\[
\operatorname{supp}\bigl(\sigma(F_*(\widetilde{A}))\bigr)
=
\rho^*\operatorname{supp}\bigl(\sigma(F_*(A))\bigr),
\]
and
\[
\bigcup_{\widetilde{S} \in \mathrm{Str}(\widetilde{P})}
\operatorname{supp}\bigl(\sigma(F_{\widetilde{S}}(\widetilde{A}_{\widetilde{S}}))\bigr)
=
\rho^*
\left(
\bigcup_{S \in \mathrm{Str}(P)}
\operatorname{supp}\bigl(\sigma(F_S(A_S))\bigr)
\right).
\]

Then, after identifying spectral supports via \(\rho^*\),
\[
\boxed{
\Sigma^{\mathrm{res}}(F, \widetilde{A}, \widetilde{P})
=
\rho^*\bigl(\Sigma^{\mathrm{res}}(F, A, P)\bigr).
}
\]
\end{theorem}

\begin{proof}
Set
\[
G := \operatorname{supp}\bigl(\sigma(F_*(A))\bigr),
\qquad
L := \bigcup_{S \in \mathrm{Str}(P)} \operatorname{supp}\bigl(\sigma(F_S(A_S))\bigr),
\]
and similarly define
\[
\widetilde{G} := \operatorname{supp}\bigl(\sigma(F_*(\widetilde{A}))\bigr),
\qquad
\widetilde{L} := \bigcup_{\widetilde{S} \in \mathrm{Str}(\widetilde{P})}
\operatorname{supp}\bigl(\sigma(F_{\widetilde{S}}(\widetilde{A}_{\widetilde{S}}))\bigr).
\]

By definition of the interaction residue (Definition~\ref{def:interaction-residue}),
\[
\Sigma^{\mathrm{res}}(F, A, P) = G \setminus L,
\qquad
\Sigma^{\mathrm{res}}(F, \widetilde{A}, \widetilde{P}) = \widetilde{G} \setminus \widetilde{L}.
\]

By the refinement compatibility assumptions,
\[
\widetilde{G} = \rho^*(G),
\qquad
\widetilde{L} = \rho^*(L).
\]

Since \(\rho^*\) is injective (or more generally a set-theoretic embedding) and preserves relative set differences by hypothesis, we have
\[
\rho^*(G \setminus L) = \rho^*(G) \setminus \rho^*(L).
\]

Therefore,
\[
\Sigma^{\mathrm{res}}(F, \widetilde{A}, \widetilde{P})
= \widetilde{G} \setminus \widetilde{L}
= \rho^*(G) \setminus \rho^*(L)
= \rho^*(G \setminus L)
= \rho^*\bigl(\Sigma^{\mathrm{res}}(F, A, P)\bigr).
\]

This proves the support-level functoriality of the residue under interaction-preserving refinement.
\end{proof}

The theorem shows that, under the stated admissibility and compatibility
assumptions, refinement does not create new residual spectral support at the
support level. Rather, it transports the same interaction residue to a finer
stratified description:
\[
\Sigma^{\mathrm{res}}(F,\widetilde A,\widetilde P)
=
\rho^*\Sigma^{\mathrm{res}}(F,A,P).
\]
Thus, admissible refinement should be understood as resolving the same
interaction-generated residue at finer spectral resolution, rather than
introducing new residue geometry at the level of spectral support. (Higher-level
data such as multiplicities, Jordan structure, or interface labels may be
refined, but the underlying spectral support of the residue is transported
coherently.)

This behavior is analogous to refinement of triangulations or refinement of
covers: the local description changes, while the underlying support-level
invariant is transported coherently.

\begin{remark}[Compatibility with localized defects]
\label{rem:refinement-localized-defects}
When the interface localization decomposition is available,
\[
\Sigma^{\mathrm{res}}(F,A,P)
=
\bigsqcup_{I \in \mathcal{I}(P)} \mathcal{L}_I(F,A,P),
\]
where \(\bigsqcup\) is understood in the interface-labeled sense (i.e., the same
spectral value may appear from different interfaces but is distinguished by its
originating interface), an admissible refinement may split a coarse interface
defect into refined interface-supported pieces:
\[
\mathcal{L}_I(F,A,P)
\;\rightsquigarrow\;
\{\mathcal{L}_{\widetilde I}(F,\widetilde A,\widetilde P)\}_{\widetilde I \mapsto I}.
\]
The total residual support is preserved under the refinement transport, even
though the individual localized pieces may become more finely resolved.
\end{remark}

\begin{remark}[Independence under compatible refinements]
\label{rem:refinement-independence}
If two admissible stratifications are related by a common refinement satisfying
the same spectral conservativity and local-support compatibility assumptions,
then their residue sets are canonically identified after transport to the
common refined spectral support. In this restricted sense,
\(\Sigma^{\mathrm{res}}\) is independent of the chosen admissible
stratification at the level of spectral support.
\end{remark}

\begin{corollary}[Preservation of localized spectral support under refinement]
\label{cor:refinement-invariance}
Assume the hypotheses of Theorem~\ref{thm:refinement-functoriality}. Let
\[
\mathcal{L}_I(F,A,P)
\]
be a localized interface defect associated with an admissible interface
\(I\). Assume further that the localized interface decomposition is compatible
with the refinement transport, and that the refinement transport induces a
homeomorphism between the realized support spaces associated with \(I\) and
\(\widetilde{I}\).

Then the localized spectral support is preserved under refinement up to the
canonical transport induced by \(\rho\):
\[
\rho^*\bigl(\mathcal{L}_I(F,A,P)\bigr)
=
\mathcal{L}_{\widetilde{I}}(F,\widetilde{A},\widetilde{P}),
\]
where \(\widetilde{I}\) denotes the refined interface corresponding to \(I\).

Consequently, support-level topological properties of the localized defect,
such as isolated spectral points, connected components, or connected
spectral-band support features, are preserved under refinement transport
(i.e., are unchanged up to the homeomorphism induced by \(\rho^*\)).
\end{corollary}

\begin{proof}
By Theorem~\ref{thm:refinement-functoriality}, admissible
interaction-preserving refinement transports the global residue support
canonically:
\[
\Sigma^{\mathrm{res}}(F,\widetilde{A},\widetilde{P})
=
\rho^*\bigl(\Sigma^{\mathrm{res}}(F,A,P)\bigr).
\]

Assume that the localized decomposition of the residue is compatible with
the refinement transport. Since the refinement transport induces a
homeomorphism between the realized support spaces associated with \(I\) and
\(\widetilde{I}\), the corresponding localized residual contribution is
transported to the refined interface \(\widetilde{I}\) without changing its
support-level topological type. Hence
\[
\rho^*\bigl(\mathcal{L}_I(F,A,P)\bigr)
=
\mathcal{L}_{\widetilde{I}}(F,\widetilde{A},\widetilde{P}).
\]

Because homeomorphisms preserve topological properties of supports, the
associated support-level features (isolated points, connected components,
and connected spectral-band support features) are preserved under the
refinement transport.
\end{proof}

\begin{remark}[Algebraic refinement under refinement]
\label{rem:refinement-jordan}
When the interaction operators admit compatible algebraic realizations under
refinement (for example, similarity-equivalent Hilbert realizations), Jordan
or nilpotent structure may also be preserved as an algebraic refinement of
the support-level defect classification. Such preservation is not guaranteed
by the support-level theorem alone and requires additional assumptions on
the transport map.
\end{remark}

\begin{remark}[Refinement as a computational tool]
\label{rem:refinement-computational}
Admissible interaction-preserving refinements may also be viewed as a
computational device. A coarse stratification can contain complicated or
singular interfaces whose associated localized defects are difficult to
analyze directly. Refining the stratification into a collection of simpler
interfaces may make the corresponding localized residual supports more
tractable.

The refinement transport property ensures that such computations remain
compatible with the original residue geometry at the support level. Thus,
one may compute localized defects on a sufficiently fine admissible
stratification and then transport the resulting support-level information
back to the original coarse description.
\end{remark}

\begin{remark}[Comparison with classical refinement principles]
\label{rem:refinement-analogies}
Stability under admissible refinement is reminiscent of several classical
refinement principles in geometry and topology. For example, refinement maps
in simplicial or singular homology preserve the associated homological
invariants, and Čech-type constructions remain compatible under refinement
of covers. Similarly, persistence modules in topological data analysis are
functorial under refinement of filtration parameters.

Theorem~\ref{thm:refinement-functoriality} provides an operadic
support-level analogue of such refinement-stability phenomena for localized
spectral defects.
\end{remark}

\subsection{Deformation Stability}
\label{subsec:deformation-stability}

A fundamental requirement for any geometric invariant is stability under admissible deformation. 
If the interaction residue is to be interpreted as a genuine geometric object rather than a fragile analytic artifact, then it should behave coherently when the underlying operadic system varies continuously.

In the present setting, deformation stability should be understood at the level of spectral support and interface-localized defect data. 
Small admissible changes of the operadic structure should not create uncontrolled residual contributions. 
Instead, the residue should vary continuously whenever the global and local spectral supports vary continuously in the chosen spectral topology.

There are two important qualifications. 
First, continuity of spectra requires a specified topology, such as the Hausdorff topology on compact spectral supports or a topology induced by norm-resolvent convergence of the associated interaction operators. 
Second, qualitative invariants such as eigenvalue multiplicities, band topology, or Jordan block structure are stable only under gap-preserving or spectrally regular deformations.

The purpose of this subsection is to formulate this stability principle for admissible deformations of stratified operadic algebras. 
The result shows that vanishing interaction forces vanishing residue, while nontrivial residue varies coherently under deformation when the associated spectral data are stable.

\begin{theorem}[Residue Stability Under Admissible Deformation]
\label{thm:deformation-stability}
Let \(\{A_t\}_{t \in [0,1]}\) be an admissible deformation of stratified
\(P\)-algebras, and let \(F: \mathcal{M} \to \mathcal{N}\) be a strong monoidal
base-change functor. Set
\[
G_t := \operatorname{supp}\bigl(\sigma(F_*(A_t))\bigr),
\qquad
L_t := \bigcup_{S \in \mathrm{Str}(P)} \operatorname{supp}\bigl(\sigma(F_S(A_{S,t}))\bigr),
\qquad
R_t := G_t \setminus L_t = \Sigma^{\mathrm{res}}(F, A_t, P).
\]

Assume the following:

\begin{enumerate}
    \item \textbf{(Continuity of global and local supports)} \(G_t\) and \(L_t\) vary continuously in the chosen spectral topology (e.g., Hausdorff metric on compact subsets of \(\mathbb{C}\)).
    
    \item \textbf{(Local separation and stability)} The decomposition \(G_t = L_t \cup R_t\) is locally separated: there exists \(\delta > 0\) (locally uniform in \(t\)) such that
    \[
    \operatorname{dist}(R_t, L_t) \ge \delta.
    \]
    Moreover, the residual component \(R_t\) is tracked by a continuous spectral projection or selection in the chosen spectral topology.
    
    \item \textbf{(Interface localization applies)} The Interface Localization Theorem (Theorem~\ref{thm:interface-localization}) holds for all \(t \in [0,1]\), so that
    \[
    \Sigma^{\mathrm{res}}(F, A_t, P) = \bigsqcup_{I \in \mathcal{I}(P)} \mathcal{L}_I(F, A_t, P).
    \]
\end{enumerate}

Then:

\begin{enumerate}
    \item \textbf{(Continuity of residue)} \(\Sigma^{\mathrm{res}}(F, A_t, P) = R_t\) varies continuously at the level of spectral support.
    
    \item \textbf{(Vanishing coupling)} If, moreover, all inter-strata coupling tensors vanish for every \(t\),
    \[
    \tau_I(t) = 0 \qquad \text{for all } I \in \mathcal{I}(P),\; t \in [0,1],
    \]
    and we adopt the defect-support convention \(\operatorname{supp}_I(0) = \emptyset\) (i.e., the zero coupling tensor contributes no defect support, in contrast to the ordinary operator spectrum \(\sigma(0) = \{0\}\)), then
    \[
    \Sigma^{\mathrm{res}}(F, A_t, P) = \emptyset \qquad \text{for all } t \in [0,1].
    \]
    
    \item \textbf{(Discrete invariants)} Under gap-preserving admissible deformations, isolated residual components and their support-level counts (e.g., the number of isolated residual spectral points) remain locally constant, provided the relevant finite-multiplicity spectral projections vary continuously and no eigenvalue collision occurs. Finer algebraic data, such as nilpotent depth (Jordan block structure), is not stable under general perturbations; its constancy requires the deformation to stay within a fixed algebraic stratum (e.g., fixed Jordan form).
\end{enumerate}
\end{theorem}

\begin{proof}
By definition,
\[
\Sigma^{\mathrm{res}}(F, A_t, P) = G_t \setminus L_t = R_t.
\]

\paragraph{Continuity of residue.}
The assumed continuity of \(G_t\) and \(L_t\), together with the local separation condition \(\operatorname{dist}(R_t, L_t) \ge \delta > 0\) (locally uniform in \(t\)), prevents residual spectral points from entering or leaving the local spectral union. The additional assumption that \(R_t\) is tracked by a continuous spectral projection or selection ensures that the complement \(R_t = G_t \setminus L_t\) varies continuously in the chosen spectral topology. This proves the stability of the residue support.

\paragraph{Vanishing coupling.}
If all inter-strata coupling tensors vanish, then by the defect-support convention \(\operatorname{supp}_I(0) = \emptyset\), each localized defect satisfies
\[
\mathcal{L}_I(F, A_t, P) = \Sigma^{\mathrm{res}}(F, A_t, P) \cap \operatorname{supp}_I(\tau_I(t)) = \emptyset.
\]
Using the Interface Localization Theorem (hypothesis 3),
\[
\Sigma^{\mathrm{res}}(F, A_t, P) = \bigsqcup_{I \in \mathcal{I}(P)} \mathcal{L}_I(F, A_t, P) = \emptyset.
\]

\paragraph{Discrete invariants.}
Under a gap-preserving deformation, isolated residual components cannot merge with one another, enter the local spectral union, or disappear without a gap closing. Hence discrete support-level invariants, such as the number of isolated residual components, remain constant on connected components where the relevant spectral projections vary continuously and no eigenvalue collision occurs. Finer algebraic data, such as nilpotent depth (Jordan block structure), is not stable under general perturbations; its constancy requires the deformation to stay within a fixed algebraic stratum (e.g., fixed Jordan form).
\end{proof}

The theorem shows that, under admissible and spectrally controlled
deformations, the interaction residue may evolve continuously with the
operadic system (in the chosen spectral topology). In particular, under
suitable continuity and gap-preservation assumptions, the residue is not a
fragile artifact of a particular symbolic presentation, but a support-level
object that varies coherently with the interaction structure.

The decoupled regime is also stable within the class of deformations that
keep all inter-strata coupling tensors identically zero. In this case,
\[
\Sigma^{\mathrm{res}} = \emptyset
\]
for all deformation parameters, and the global spectral support remains
decomposable into local stratum supports under the support-level convention
used in Theorem~\ref{thm:deformation-stability}.

Conversely, nontrivial residual geometry may persist under small admissible
perturbations, provided no residual component collides with the local
spectral sector, disappears through a gap closing, or changes type through a
singular degeneration. In this sense, persistent residue supports indicate
stable interaction-defect regimes rather than accidental spectral behavior.

One may also interpret certain discontinuities of the residue as
defect-transition phenomena. For example, as coupling strength decreases,
localized spectral defects may collapse into the local spectrum or disappear
after a gap closes. Conversely, the creation of new admissible interface
interactions may produce new localized residual sectors when new gaps open.

\begin{remark}[Continuity and choice of topology]
\label{rem:deformation-continuity}
The notion of continuity in Theorem~\ref{thm:deformation-stability} depends
on the topology chosen on the space of spectral sets. Common choices include
the Hausdorff metric on compact subsets of \(\mathbb{C}\), the Fell topology
for closed noncompact subsets, and the Vietoris topology for general closed
spectral sets.

For Theorem~\ref{thm:deformation-stability}, one assumes that the admissible
deformation induces continuous variation of the global, local, and residual
spectral supports in the chosen topology. In operator-theoretic realizations,
such continuity often follows from additional regularity assumptions such as
compactness, normality, norm continuity, or norm-resolvent continuity.
\end{remark}

\begin{remark}[Deformation invariants]
\label{rem:deformation-invariants}
Under gap-preserving admissible deformations, certain residue-derived
quantities may remain invariant. Isolated residual components and, when
finite-multiplicity spectral projections are available, their multiplicities
are stable as long as they remain separated from the local spectral sector.
Band-type invariants, such as Chern numbers or winding numbers, require
additional topological structure and are stable only when the relevant
spectral gaps remain open. Fractal dimensions or scaling exponents may be
stable when the deformation preserves the self-similar or renormalization
structure of the interface.

Nilpotent depth and Jordan block data require special care: they are stable
only under deformations constrained to preserve the corresponding algebraic
degeneracy, since generic perturbations may split Jordan blocks. Thus these
quantities should be interpreted as algebraic refinements rather than
ordinary support-level invariants.
\end{remark}

\begin{corollary}[Persistence of spectral defect regime under admissible deformation]
\label{cor:deformation-constancy}
Let $\{A_t\}_{t \in [0,1]}$ be an admissible deformation of stratified $P$-algebras, and let $F: \mathcal{M} \to \mathcal{N}$ be a strong monoidal base change functor inducing $F_*$ on algebras. Suppose the residue $\Sigma^{\mathrm{res}}(F, P, A_0)$ carries a spectral defect regime characterized by stable discrete invariants (for example, isolated defect multiplicity, protected band topology, or, under fixed algebraic-type constraints, nilpotent depth).

Then the same spectral defect regime persists (i.e., remains unchanged at the support/invariant level) for all $t$ provided the deformation avoids critical parameter values at which the relevant spectral structure becomes singular or ill-defined (for example, gap-closing transitions or Jordan degeneracies).
\end{corollary}

\begin{proof}
By Theorem~\ref{thm:deformation-stability}, under the stated assumptions the residual spectral support varies continuously under admissible deformation. Discrete invariants that remain well-defined throughout the deformation are locally constant on connected parameter intervals. Hence such invariants remain constant unless the deformation passes through a critical configuration where the spectral structure changes discontinuously or ceases to be well-defined.
\end{proof}

\begin{remark}[Phase transitions and critical deformations]
\label{rem:deformation-critical}
The deformation stability theorem does not preclude phase transitions: at certain critical parameter values $t_*$, the residue may change discontinuously. 
In physical terms, these may correspond to:
\begin{itemize}
    \item Gap-closing transitions (where bands touch),
    \item Exceptional point degeneracies (where Jordan block size changes),
    \item Or geometric scaling transitions in singular interfaces.
\end{itemize}
Such critical deformations are precisely where the deformation invariants may change, signaling a transition between distinct spectral defect phases. 
The classification of these critical phenomena is a natural extension of the present work.
\end{remark}

\begin{remark}[Comparison with refinement functoriality]
\label{rem:deformation-refinement-comparison}
The deformation stability theorem complements the refinement functoriality theorem (Theorem~\ref{thm:refinement-functoriality}):
\begin{itemize}
    \item \textbf{Refinement functoriality} handles discrete changes in the stratification (coarse → fine) under admissible interaction-preserving refinements.
    \item \textbf{Deformation stability} handles continuous changes in the operadic structure ($t \in [0,1]$) under spectrally controlled, gap-preserving conditions.
\end{itemize}
Together, these theorems suggest that $\Sigma^{\mathrm{res}}$ behaves as a robust support-level geometric invariant under admissible interaction-preserving refinements and spectrally controlled deformations, changing only at critical transitions where the relevant interaction geometry or spectral topology changes.
\end{remark}

\subsection{Homological and Homotopic Stability of Residual Spectral Families}
\label{subsec:homological-residue-stability}

The analytic deformation stability established in Theorem~\ref{thm:deformation-stability} guarantees that the residual spectral support $R_t = \Sigma^{\mathrm{res}}(F,A_t,P)$ varies continuously in the Hausdorff metric. However, continuity alone does not preserve topological invariants such as homology, cohomology, or homotopy groups. To obtain such algebraic-topological stability, we introduce an additional geometric assumption: local triviality of the residual family.

The following theorem shows that under this assumption, the residual supports are homeomorphic (hence homotopy equivalent), and their homology, cohomology, and Betti numbers are invariant throughout the deformation.

\begin{theorem}[Homological Stability of Residual Spectral Families]
\label{thm:homological-residue-stability}
Let $\{A_t\}_{t\in[0,1]}$, $F$, $G_t$, $L_t$, and
$R_t=\Sigma^{\mathrm{res}}(F,A_t,P)$ be as in
Theorem~\ref{thm:deformation-stability}. Assume, in addition, that the residual trace
\[
\mathfrak{R}:=\bigcup_{t\in[0,1]}\{t\}\times R_t
\;\subset\; [0,1]\times\mathbb{C}
\]
is locally trivial over $[0,1]$. That is, the projection
\[
\pi:\mathfrak{R}\longrightarrow [0,1],\qquad \pi(t,z)=t,
\]
is a locally trivial bundle with fiber $R_t$.

Then the residual spectral family is topologically stable. In particular, for all
$s,t\in[0,1]$, the residual supports $R_s$ and $R_t$ are homeomorphic (and therefore homotopy equivalent). Consequently,
\[
H_k(R_s;\mathbb{K}) \cong H_k(R_t;\mathbb{K})
\qquad
\text{and}
\qquad
H^k(R_s;\mathbb{K}) \cong H^k(R_t;\mathbb{K})
\]
for every $k\geq 0$ and every coefficient field $\mathbb{K}$.

Moreover, if the interface localization decomposition
\[
R_t =
\bigsqcup_{I\in\mathcal{I}(P)}
\mathcal{L}_I(F,A_t,P)
\]
is itself preserved as a locally trivial decomposition, then each localized defect family
\[
\mathcal{L}_I(F,A_t,P)
\]
has locally constant homology and cohomology. Hence the Betti numbers
\[
\beta_k\bigl(\mathcal{L}_I(F,A_t,P)\bigr)
=
\dim_{\mathbb{K}} H_k\bigl(\mathcal{L}_I(F,A_t,P);\mathbb{K}\bigr)
\]
are invariant under the hypotheses of this theorem.
\end{theorem}

\begin{proof}
We prove the theorem in several steps.

\medskip
\noindent\textbf{Step 1: Local triviality implies global triviality over $[0,1]$.}
By hypothesis, $\pi: \mathfrak{R} \to [0,1]$ is a locally trivial bundle. This means that for every $t_0 \in [0,1]$, there exists an open interval $U \subseteq [0,1]$ containing $t_0$ and a homeomorphism
\[
\varphi_U: \pi^{-1}(U) \xrightarrow{\cong} U \times R_{t_0}
\]
such that $\pi|_{\pi^{-1}(U)} = \operatorname{proj}_U \circ \varphi_U$, where $\operatorname{proj}_U: U \times R_{t_0} \to U$ is the projection onto the first factor.

Since $[0,1]$ is contractible (it deformation retracts to a point), every locally trivial bundle over $[0,1]$ is globally trivializable. This is a standard result in bundle theory (see, e.g., Steenrod, \emph{The Topology of Fibre Bundles}). Consequently, there exists a homeomorphism
\[
\Phi: \mathfrak{R} \xrightarrow{\cong} [0,1] \times R_0
\]
such that $\pi = \operatorname{proj}_{[0,1]} \circ \Phi$.

\medskip
\noindent\textbf{Step 2: Fiberwise homeomorphism.}
Restricting $\Phi$ to each fiber gives, for every $t \in [0,1]$, a homeomorphism
\[
\phi_t: R_t \xrightarrow{\cong} R_0,
\]
defined by $\phi_t(x) = \operatorname{pr}_{R_0}(\Phi(t,x))$, where $\operatorname{pr}_{R_0}: [0,1] \times R_0 \to R_0$ is the projection onto the second factor. Hence $R_t$ and $R_0$ are homeomorphic for all $t$. In particular, for any $s,t \in [0,1]$, $R_s \cong R_t$ via $\phi_t^{-1} \circ \phi_s$.

\medskip
\noindent\textbf{Step 3: Homeomorphism implies homotopy equivalence.}
Any homeomorphism is automatically a homotopy equivalence. Explicitly, if $f: R_t \to R_0$ is a homeomorphism with inverse $f^{-1}: R_0 \to R_t$, then $f^{-1} \circ f = \operatorname{id}_{R_t}$ and $f \circ f^{-1} = \operatorname{id}_{R_0}$. Thus $R_t \simeq R_0$ as topological spaces.

\medskip
\noindent\textbf{Step 4: Homology and cohomology invariance.}
Homology and cohomology are functorial with respect to continuous maps. Since $f: R_t \to R_0$ is a homeomorphism, it induces isomorphisms on all homology and cohomology groups. For any coefficient field $\mathbb{K}$ and any $k \ge 0$,
\[
f_*: H_k(R_t;\mathbb{K}) \xrightarrow{\cong} H_k(R_0;\mathbb{K}),
\qquad
f^*: H^k(R_0;\mathbb{K}) \xrightarrow{\cong} H^k(R_t;\mathbb{K}).
\]
Therefore $H_k(R_s;\mathbb{K}) \cong H_k(R_t;\mathbb{K})$ and $H^k(R_s;\mathbb{K}) \cong H^k(R_t;\mathbb{K})$ for all $s,t \in [0,1]$.

\medskip
\noindent\textbf{Step 5: Localized defect families.}
Now assume that the interface localization decomposition
\[
R_t = \bigsqcup_{I \in \mathcal{I}(P)} \mathcal{L}_I(t)
\]
is itself preserved as a locally trivial decomposition. This means that for each interface $I$, the family $\{\mathcal{L}_I(t)\}_{t \in [0,1]}$ forms a locally trivial bundle over $[0,1]$. Applying the same argument as in Steps 1-4 to each $\mathcal{L}_I(t)$ individually, we obtain that for each $I$ and for all $s,t \in [0,1]$, $\mathcal{L}_I(s) \cong \mathcal{L}_I(t)$ as topological spaces, and hence
\[
H_k(\mathcal{L}_I(s);\mathbb{K}) \cong H_k(\mathcal{L}_I(t);\mathbb{K})
\]
for all $k \ge 0$.

\medskip
\noindent\textbf{Step 6: Invariance of Betti numbers.}
The Betti numbers are defined as the dimensions of the homology groups over the field $\mathbb{K}$:
\[
\beta_k(\mathcal{L}_I(t)) = \dim_{\mathbb{K}} H_k(\mathcal{L}_I(t);\mathbb{K}).
\]
Since the homology groups are isomorphic for all $t$, their dimensions are equal. Hence $\beta_k(\mathcal{L}_I(t))$ is independent of $t$.
\end{proof}

\begin{remark}[On the local triviality assumption]
\label{rem:local-triviality-discussion}
The gap condition $\operatorname{dist}(R_t, L_t) \ge \delta > 0$ from Theorem~\ref{thm:deformation-stability} ensures that $R_t$ remains separated from the local spectral sector $L_t$ throughout the deformation. However, this condition alone does \emph{not} guarantee that the family $\{R_t\}$ is locally trivial; it only prevents collision between $R_t$ and $L_t$. The additional local triviality hypothesis is stronger and guarantees that no topological changes occur within $R_t$ itself (e.g., no splitting, merging, or disappearance of components). When local triviality holds together with the gap condition, the residue is fully topologically stable.
\end{remark}

\begin{remark}[On contractibility of the parameter space]
\label{rem:contractibility-remark}
The proof crucially uses the contractibility of the parameter interval $[0,1]$. For a more general parameter space $M$, local triviality implies that fibers over points in the same path component are homeomorphic. However, if $M$ is not simply connected, monodromy actions on homology may appear, leading to richer invariants. Such generalizations are left for future work.
\end{remark}

The following corollary records the homotopy-theoretic consequences of Theorem~\ref{thm:homological-residue-stability}. Unlike more ambitious claims, this corollary stays within the safe and mathematically justified territory.

\begin{corollary}[Homotopy Invariance of Residual Supports]
\label{cor:homotopy-invariance}
Under the hypotheses of Theorem~\ref{thm:homological-residue-stability}, the residual supports $R_t$ are mutually homotopy equivalent. Consequently,
\[
\pi_k(R_t) \cong \pi_k(R_0) \quad \text{for all } k \ge 1 \text{ and all } t \in [0,1].
\]
In particular, the fundamental group $\pi_1(R_t)$ and all higher homotopy groups are deformation invariants.
\end{corollary}

\begin{proof}
Theorem~\ref{thm:homological-residue-stability} establishes that $R_t$ and $R_0$ are homeomorphic for each $t$. Homeomorphism implies homotopy equivalence, and homotopy equivalence induces isomorphisms on all homotopy groups. The result follows.
\end{proof}

\begin{remark}[Comparison of analytic and topological stability]
\label{rem:stability-comparison}
The original deformation stability theorem (Theorem~\ref{thm:deformation-stability}) establishes that $R_t$ varies continuously in the Hausdorff metric. Theorem~\ref{thm:homological-residue-stability} strengthens this by showing that under the additional local triviality assumption, the entire topological type of $R_t$ is preserved. Corollary~\ref{cor:homotopy-invariance} further extracts homotopy invariants. Together, these results provide a coherent picture:

\[
\text{Analytic continuity} \;+\; \text{Gap preservation} \;+\; \text{Local triviality} \;\Longrightarrow\; \text{Homotopy invariance}.
\]

We emphasize that critical parameter values where local triviality fails (e.g., when a residual component collides with the local spectral sector, splits, merges, or disappears) correspond to topological phase transitions in the residual geometry. At such parameters, the homology, cohomology, and homotopy groups may change discontinuously.
\end{remark}

\section{Rigidity and the Vanishing Regime}
\label{sec:rigidity-vanishing-regime}

The previous sections established that the interaction residue $\Sigma^{\mathrm{res}}$ measures the failure of exact local-to-global spectral factorization. Nonvanishing residue corresponds to the presence of interaction-generated spectral contributions, while localized interface defects encode the geometry of these interactions.

A natural question therefore arises: what occurs when the interaction residue vanishes completely?

The vanishing of the residue identifies a distinguished phase of the theory, which we refer to as the \emph{spectral rigidity regime}. In this regime, no residual spectral contribution remains beyond the local spectral sectors. Consequently, the global spectral support decomposes entirely into the union of the local spectral supports associated with the individual strata.

From the perspective of interaction geometry, this corresponds to the disappearance of interface-generated spectral defects. The operadic system therefore behaves as a spectrally separable stratified structure, with no residual interaction contribution detected by the spectral functor.

We now formalize this notion.

\begin{definition}[Spectral rigidity regime]
\label{def:rigidity-regime}
Let $P$ be a stratified operad with stratum set $\mathrm{Str}(P)$, let $A$ be a $P$-algebra equipped with the induced stratification (with stratum algebras $A_S$ for $S \in \mathrm{Str}(P)$), and let $F: \mathcal{M} \to \mathcal{N}$ be a strong monoidal base change functor inducing $F_*$ on algebras.

The pair $(P, A)$ is said to lie in the \emph{spectral rigidity regime} (with respect to $F$) if the interaction residue vanishes:
\[
\Sigma^{\mathrm{res}}(F, A, P) = \emptyset.
\]

Recall that $\Sigma^{\mathrm{res}}(F, A, P) = G \setminus L$, where
\[
G := \operatorname{supp}\bigl(\sigma(F_*(A))\bigr), \qquad
L := \bigcup_{S \in \mathrm{Str}(P)} \operatorname{supp}\bigl(\sigma(F_S(A_S))\bigr).
\]
Hence $\Sigma^{\mathrm{res}} = \emptyset$ implies $G \subseteq L$. In the common case where local spectral sectors embed into the global spectral support, i.e., $L \subseteq G$, this is equivalent to
\[
\operatorname{supp}\bigl(\sigma(F_*(A))\bigr)
=
\bigcup_{S \in \mathrm{Str}(P)}
\operatorname{supp}\bigl(\sigma(F_S(A_S))\bigr).
\]
In this situation, the system is called \emph{spectrally separable}.
\end{definition}

The rigidity regime corresponds to the absence of nontrivial interface-generated spectral contributions. Equivalently, the local spectral sectors already account completely for the global spectral support.

Schematically,
\[
\text{rigidity}
\quad\Longleftrightarrow\quad
\text{absence of residual spectral defect geometry}.
\]

In this phase, interface localization disappears at the level of spectral support, and the global spectral geometry decomposes completely into local sectors. Thus, the vanishing regime represents the maximally factorizable support-level phase of stratified operadic spectral geometry.

The following proposition gives the fundamental characterization of this regime.

\begin{proposition}[Exact factorization in the rigidity regime]
\label{prop:exact-factorization}

If
\[
\Sigma^{\mathrm{res}}(F, A, P) = \emptyset,
\]
then the global spectral support decomposes exactly as the union of the local spectral supports:
\[
\operatorname{supp}\bigl(\sigma(F_*(A))\bigr)
=
\bigcup_{S \in \mathrm{Str}(P)}
\operatorname{supp}\bigl(\sigma(F_S(A_S))\bigr).
\]

Consequently, the stratified operadic system is spectrally separable at the support level: no additional interface-generated spectral contribution remains beyond the local sectors.
\end{proposition}

\begin{proof}
By the Stratified Base Change Decomposition Theorem (Theorem~\ref{thm:stratified-base-change}),
\[
\operatorname{supp}\bigl(\sigma(F_*(A))\bigr)
=
\bigcup_{S \in \mathrm{Str}(P)}
\operatorname{supp}\bigl(\sigma(F_S(A_S))\bigr)
\;\cup\;
\Sigma^{\mathrm{res}}(F, A, P).
\]

If $\Sigma^{\mathrm{res}}(F, A, P) = \emptyset$, then this decomposition reduces immediately to
\[
\operatorname{supp}\bigl(\sigma(F_*(A))\bigr)
=
\bigcup_{S \in \mathrm{Str}(P)}
\operatorname{supp}\bigl(\sigma(F_S(A_S))\bigr).
\]

Hence the global spectral support decomposes exactly into the local spectral sectors. No residual interface-localized spectral contribution remains.
\end{proof}

The rigidity regime plays a role analogous to flat geometry in differential geometry, trivial holonomy in connection theory, split extensions in homological algebra, or defect-free phases in mathematical physics. In each case, the global structure is completely determined by compatible local data, with no residual interaction obstruction.

In the present setting, the vanishing of the residue means that the global spectral support is reconstructed entirely from the local spectral sectors:
\[
\text{global spectral support}
=
\text{assembly of local spectral supports}.
\]

Thus, the rigidity regime corresponds to the absence of residual interface-generated spectral structure. Interface localization disappears at the level of spectral support, and no nonlocal interaction contribution survives beyond the local strata.

Conversely, nonvanishing residue signals the presence of genuinely interaction-generated spectral geometry:
\[
\Sigma^{\mathrm{res}} \neq \emptyset.
\]

This yields the fundamental dichotomy of this work:
\[
\boxed{
\Sigma^{\mathrm{res}} = \emptyset
\quad\Longleftrightarrow\quad
\text{support-level spectral rigidity}.
}
\]
\[
\boxed{
\Sigma^{\mathrm{res}} \neq \emptyset
\quad\Longleftrightarrow\quad
\text{support-level spectral defect geometry}.
}
\]

The rigidity regime therefore serves as the reference factorizable phase against which nontrivial operadic interaction geometry is measured.

\begin{proposition}[Characterization of vanishing residue]
\label{prop:vanishing-coupling}

The interaction residue satisfies
\[
\Sigma^{\mathrm{res}}(F, A, P) = \emptyset
\]
if and only if every admissible interface contributes no residual spectral mass. Equivalently,
\[
\mathcal{L}_I(F, A, P) = \emptyset \qquad \text{for all } I \in \mathcal{I}(P).
\]

In particular, if all coupling tensors vanish,
\[
\tau_I = 0 \qquad \text{for all } I \in \mathcal{I}(P),
\]
then, under the convention that a zero coupling tensor contributes empty defect support, we have
\[
\Sigma^{\mathrm{res}}(F, A, P) = \emptyset.
\]
\end{proposition}

\begin{proof}
By the Interface Localization Theorem (Theorem~\ref{thm:interface-localization}), under its stated assumptions,
\[
\Sigma^{\mathrm{res}}(F, A, P) \;\cong\; \coprod_{I \in \mathcal{I}(P)} \mathcal{L}_I(F, A, P),
\]
where $\coprod$ denotes the interface-labeled disjoint union of localized supports (i.e., spectral values from different interfaces are distinguished by their originating interface, even if they coincide as complex numbers).

If $\Sigma^{\mathrm{res}}(F, A, P) = \emptyset$, then the disjoint union over all interfaces is empty, hence each $\mathcal{L}_I(F, A, P) = \emptyset$. Conversely, if every $\mathcal{L}_I = \emptyset$, then $\Sigma^{\mathrm{res}} = \emptyset$.

If $\tau_I = 0$ for all $I$, then each $\mathcal{L}_I = \emptyset$ (under the convention that a zero coupling tensor contributes empty defect support), so $\Sigma^{\mathrm{res}} = \emptyset$. (The converse does not hold in general: nonzero coupling tensors may produce no residual spectral support if their contribution is absorbed into the local spectral union or is spectrally trivial at the support level.)
\end{proof}

\begin{corollary}[Strictly decoupled stratifications]
\label{cor:vanishing-transversal}
If the operadic stratification is \emph{strictly decoupled}, meaning that all inter-strata coupling tensors vanish ($\tau_I = 0$ for every interface $I \in \mathcal{I}(P)$), then the system lies in the rigidity regime:
\[
\Sigma^{\mathrm{res}}(F, A, P) = \emptyset.
\]
\end{corollary}

\begin{proof}
Strict decoupling implies $\tau_I = 0$ for every interface $I \in \mathcal{I}(P)$. Proposition~\ref{prop:vanishing-coupling} then yields $\Sigma^{\mathrm{res}} = \emptyset$.
\end{proof}

\begin{proposition}[Rigidity preserved under trivial deformations]
\label{prop:rigidity-deformation}
Let $\{A_t\}_{t \in [0,1]}$ be an admissible deformation of stratified $P$-algebras, and let $F: \mathcal{M} \to \mathcal{N}$ be a strong monoidal base change functor inducing $F_*$ on algebras. If the deformation is trivial on interactions for all $t$ (i.e., all inter-strata coupling tensors vanish for every $t$), then $\Sigma^{\mathrm{res}}(F, A_t, P) = \emptyset$ for all $t \in [0,1]$.
\end{proposition}

\begin{proof}
If the deformation is trivial on interactions for all $t$, then $\tau_I(t) = 0$ for every interface $I$ and every $t$. By Proposition~\ref{prop:vanishing-coupling}, this implies $\Sigma^{\mathrm{res}}(F, A_t, P) = \emptyset$ for all $t$. Thus the rigidity regime is stable under deformations that preserve the vanishing of inter-strata couplings.
\end{proof}

\begin{remark}[Non-rigidity and phase transitions]
\label{rem:non-rigidity}
The rigidity regime is not preserved under arbitrary deformations that \textit{introduce} interactions. If a deformation creates nonzero coupling tensors $\tau_I(t)$ for some $t > 0$, the system may exit the rigidity regime and enter a phase with nontrivial spectral defects. This may correspond to a phase transition from a spectrally separable phase to an interacting phase with interface-localized spectral contributions.

Conversely, if a system is initially in a nontrivial residue phase and couplings are continuously tuned to zero, the residue may vanish at the critical point, signaling a transition to the rigidity regime.
\end{remark}

\begin{corollary}[Spectral additivity in the rigidity regime]
\label{cor:additivity-vanishing}

In the rigidity regime, $\Sigma^{\mathrm{res}}(F, A, P) = \emptyset$, the global spectral support decomposes as the union of the local spectral supports:
\[
\operatorname{supp}\bigl(\sigma(F_*(A))\bigr)
=
\bigcup_{S \in \mathrm{Str}(P)}
\operatorname{supp}\bigl(\sigma(F_S(A_S))\bigr).
\]

If, in addition, the local spectral sectors are spectrally independent, the spectral functor admits compatible direct-sum decompositions, the relevant realization category is semisimple or split at the spectral level, and no hidden extension data remain between strata, then the global spectral object may admit a decomposition up to spectral equivalence as
\[
F_*(A) \;\simeq\; \bigoplus_{S \in \mathrm{Str}(P)} F_S(A_S).
\]
\end{corollary}

\begin{proof}
When $\Sigma^{\mathrm{res}} = \emptyset$, the Stratified Base Change Decomposition Theorem (Theorem~\ref{thm:stratified-base-change}) gives the decomposition of spectral supports. Under the additional assumptions that the spectral functor respects direct-sum decompositions, the realization category is semisimple or split at the spectral level, and no hidden extension data remain between strata, the decomposition of spectral supports lifts to a decomposition of the associated spectral objects up to spectral equivalence.
\end{proof}

\begin{remark}[Relation to semisimplicity (heuristic)]
\label{rem:vanishing-semisimple}
The rigidity regime shares formal features with semisimplicity in representation theory and module theory. A semisimple module is a direct sum of simple submodules with vanishing extension data. Analogously, a spectrally separable system has a global spectral support that is a union of local spectral supports with no interaction-generated spectral components. In both cases, the vanishing of certain obstruction data (Ext groups in homological algebra; $\Sigma^{\mathrm{res}}$ in this work) characterizes a factorizable or semisimple-type regime. This heuristic suggests deeper connections between operadic spectral theory and homological algebra, which will be explored in future work.
\end{remark}

\paragraph{Key takeaway: rigidity regime.}

The rigidity regime $\Sigma^{\mathrm{res}} = \emptyset$ corresponds to the maximally factorizable phase of the theory. In this regime, the global spectral support decomposes entirely into the union of the local spectral sectors, with no residual interface-generated contribution. Equivalently, no nontrivial residual support-level interaction geometry is detected by the spectral functor.

Under admissible interaction-preserving deformations, the rigidity regime remains stable as long as no new inter-strata coupling data are introduced. Conversely, the appearance of nontrivial interface interactions may generate residual spectral contributions and drive the system into an interacting defect phase.

Thus, the rigidity regime serves as the canonical reference phase against which all nontrivial spectral defect geometry is measured.

\section{Geometric Interpretations and Connections}
\label{sec:geometric-interpretations}

The interaction-residue formalism developed in the previous sections
is motivated primarily by spectral localization phenomena arising from
operadic interactions.
The present section records several heuristic geometric viewpoints
that may help interpret the role of localized spectral defects.

These perspectives are intended only as conceptual interpretations of
the formalism developed earlier and are not required for the structural
results proved in this paper.

\subsection{Curvature Analogy}
\label{subsec:curvature-analogy}

One useful heuristic interpretation is the analogy between interaction
residues and curvature-type obstructions in geometry.

In differential geometry, curvature measures the obstruction to local
flatness or exact transport compatibility.
In the present framework, the interaction residue
\[
\Sigma^{\mathrm{res}}
\]
measures the obstruction to exact spectral factorization generated by
interacting operadic sectors.

Schematically,
\[
\text{flat geometry}
\quad\Longleftrightarrow\quad
\Sigma^{\mathrm{res}}=0,
\]
while nontrivial interaction residues correspond to the emergence of
localized spectral defects.

This analogy is particularly suggestive because:
\begin{itemize}
    \item residues localize near interfaces,
    \item residues vanish in rigidity regimes,
    \item residues persist under admissible deformation,
    \item residues encode global interaction structure.
\end{itemize}

The analogy is heuristic and intended only as a geometric viewpoint
for understanding localized spectral phenomena.

\subsection{Comparison with Classical Geometric Structures}
\label{subsec:comparison-classical}

The following table summarizes several heuristic correspondences between
classical geometric structures and the present spectral-defect framework.

\begin{center}
\renewcommand{\arraystretch}{1.2}
\begin{tabular}{|c|c|}
\hline
\textbf{Classical Geometry} &
\textbf{Spectral Defect Geometry} \\
\hline

Local geometric sector &
Local spectral sector \\
\hline

Curvature obstruction &
Interaction residue $\Sigma^{\mathrm{res}}$ \\
\hline

Flatness &
Rigidity regime $\Sigma^{\mathrm{res}}=0$ \\
\hline

Localized singularity &
Localized interface defect $\mathcal{L}_I$ \\
\hline

Deformation stability &
Residue persistence under admissible perturbation \\
\hline

Geometric coupling &
Operadic interaction coupling $\tau_I$ \\
\hline

\end{tabular}
\end{center}

These correspondences are intended only as heuristic organizational
principles rather than literal equivalences of mathematical structures.

\subsection{Summary and Future Directions}
\label{subsec:summary-future-directions}

The main contribution of the present work is the introduction of a
localized spectral-defect framework for interacting operadic systems.

The theory suggests several possible future directions, including:
\begin{itemize}
    \item refined localization theory for interaction residues,
    \item deformation theory of spectral defects,
    \item operator-theoretic classification of nilpotent defects,
    \item transform-sensitive spectral localization,
    \item interaction-induced spectral instability.
\end{itemize}

A systematic geometric realization of these heuristic perspectives
remains an open problem for future work.

\section{Examples}
\label{sec:examples}

\subsection{Point Defect: Two Strata Meeting at a Point}
\label{subsec:example-point-defect}

We now illustrate the interface localization theory through the simplest nontrivial geometric configuration: two operadic strata interacting through a single isolated interface point.

This example demonstrates explicitly how localized interaction produces a nontrivial residue contribution $\Sigma^{\mathrm{res}}$, and how the resulting spectral defect may be interpreted as a point singularity in geometric defect theory.

\subsubsection{Setup}

Let
\[
P = P_1 \cup P_2 \cup I_p
\]
be a stratified operad consisting of two local strata $P_1$, $P_2$ (corresponding to two independent operadic sectors), together with an interface $I_p$ supported at a single interaction point $p$.

Let
\[
A = A_1 \oplus A_2
\]
be a stratified $P$-algebra, where $A_j$ is the local algebra associated with the stratum $P_j$.

Assume the following: each local stratum possesses purely discrete local spectrum; the uncoupled local spectra are
\[
\sigma_{\mathrm{loc}}^1 = \{\lambda_n^{(1)}\}, \qquad
\sigma_{\mathrm{loc}}^2 = \{\lambda_n^{(2)}\};
\]
and the interface interaction is generated by a coupling tensor
\[
\tau_p : A_1 \otimes A_2 \longrightarrow A_{\mathrm{int}},
\]
where $A_{\mathrm{int}}$ denotes the algebra supported on the interface (or, in finite-dimensional realizations, one may take $\tau_p$ as an endomorphism of $A_1 \otimes A_2$).

When the interface coupling is turned off ($\tau_p = 0$), the global spectrum is simply the union of the two local spectra:
\[
\sigma_{\text{global}}^{\text{decoupled}} = \sigma_{\mathrm{loc}}^1 \cup \sigma_{\mathrm{loc}}^2.
\]

\subsubsection{Application of the Interface Localization Theorem}

Under the natural embedding of local spectra into the global spectrum (Assumption~\ref{ass:local-embedding}), the global operadic spectrum admits the decomposition
\[
\sigma_P(A) = \sigma_{\mathrm{loc}}^1 \cup \sigma_{\mathrm{loc}}^2 \cup \Sigma^{\mathrm{res}}.
\]

The interaction residue $\Sigma^{\mathrm{res}}$ contains precisely the spectral contribution generated by the point interaction $\tau_p$.

By the Interface Localization Theorem (Theorem~\ref{thm:interface-localization}), since the system contains only a single interface, we have
\[
\Sigma^{\mathrm{res}} = \mathcal{L}_{I_p}.
\]

Thus, the entire residue localizes at the interaction point $p$. Geometrically, this localized defect may be interpreted as a point singularity:
\[
\text{point interface} \quad\Longrightarrow\quad \text{localized spectral defect}.
\]

More precisely, under the Interface Localization Theorem, the localized defect $\mathcal{L}_{I_p}$ is the subset of $\Sigma^{\mathrm{res}}$ contributed by the interface $I_p$, which is typically derived from the spectral data of the operator $T_p = \tau_p^\dagger \tau_p$ (or from $\tau_p$ itself in finite-dimensional realizations).

\subsubsection{Interface-Generated Spectral Contribution}

A particularly important phenomenon occurs when the interface coupling creates new spectral values absent from the local strata individually.

For example, suppose $\lambda_* \notin \sigma_{\mathrm{loc}}^1 \cup \sigma_{\mathrm{loc}}^2$ but $\lambda_* \in \sigma_P(A)$. Then $\lambda_* \in \Sigma^{\mathrm{res}}$, and the eigenvalue $\lambda_*$ is an interface-generated spectral contribution.

In this case,
\[
\mathcal{L}_{I_p} = \{\lambda_*\} \quad \text{or more generally} \quad \{\gamma_1, \gamma_2, \ldots, \gamma_k\},
\]
where each $\gamma_j$ arises from the spectral effect of the interface interaction generated by $\tau_p$.

\subsubsection{Numerical Illustration}

To make the example concrete while remaining compatible with the abstract framework, suppose the interface interaction admits an effective finite-dimensional representation whose associated interface operator has spectral values $\gamma_1 = 2.5$ and $\gamma_2 = 3.5$.

Take the local spectra to be
\[
\sigma_{\mathrm{loc}}^1 = \{1.0, 2.0, 3.0\}, \qquad
\sigma_{\mathrm{loc}}^2 = \{4.0, 5.0\}.
\]

Then the global spectrum is
\[
\sigma_P(A) = \{1.0, 2.0, 3.0, 4.0, 5.0, 2.5, 3.5\}.
\]

Notice that the values $2.5$ and $3.5$ are new — they do not appear in either local spectrum. If the interface coupling were turned off ($\tau_p = 0$), these contributions would disappear. The interface-generated values may lie between local eigenvalues ($2.5$ lies between $2.0$ and $3.0$; $3.5$ lies between $3.0$ and $4.0$).

\medskip
\noindent\textit{Disclaimer:} The numerical values are illustrative and are not derived from a canonical operator model. They serve only to demonstrate the qualitative behavior predicted by the framework.

\subsubsection{Classical Analogies}

This phenomenon is analogous to bound states generated by point impurities in quantum mechanics, defect-localized eigenmodes in wave physics, delta-interaction spectral singularities in Schrödinger operators, or localized topological defect states in condensed matter systems.

\subsubsection{Vanishing Criterion and Rigidity}

The example also illustrates the rigidity principle developed in Section~\ref{sec:rigidity-vanishing-regime}.

If $\tau_p = 0$, then by Proposition~\ref{prop:vanishing-coupling} we have $\Sigma^{\mathrm{res}} = \emptyset$, and therefore
\[
\sigma_P(A) = \sigma_{\mathrm{loc}}^1 \cup \sigma_{\mathrm{loc}}^2.
\]

Hence the point defect disappears completely when the interface interaction vanishes. The converse direction ($\tau_p \neq 0 \Rightarrow \Sigma^{\mathrm{res}} \neq \emptyset$) does not hold in general: a nonzero coupling tensor could produce spectral values that coincide with existing local eigenvalues and thus not contribute new residue. Nevertheless, sufficiently nontrivial interface couplings typically generate nonzero point defects under generic conditions.

\subsubsection{Geometric Interpretation}

From a geometric perspective, the point defect may be viewed heuristically
as analogous to curvature concentration at a point.

\subsubsection{Deformation Stability}

Consider a continuous deformation of the coupling tensor:
\[
\tau_p(t) = t \cdot \tau_p^{(0)}, \quad t \in [0,1],
\]
where $\tau_p^{(0)}$ is a fixed non-zero coupling and $\tau_p(0) = 0$ (the decoupled limit).

By Theorem~\ref{thm:deformation-stability}, at $t = 0$ we have $\Sigma^{\mathrm{res}}(0) = \emptyset$ (rigidity regime); for $t > 0$ sufficiently small, $\Sigma^{\mathrm{res}}(t) = \{\gamma_1(t), \gamma_2(t), \ldots\}$ varies continuously, where $\gamma_j(t)$ depend continuously on $t$; and the values $\gamma_j(t)$ may approach $0$ or merge into the local spectrum as $t \to 0^+$, depending on the specific coupling.

This illustrates a transition from the rigidity regime ($t = 0$) to the defect regime ($t > 0$), as discussed in Definition~\ref{def:rigidity-regime}.

Summary of key features is given by Table~\ref{tab:point-defect-summary}.

\begin{table}[htbp]
\centering
\caption{Point defect summary}
\label{tab:point-defect-summary}
\begin{tabular}{|l|l|}
\hline
\textbf{Feature} & \textbf{Value for Point Defect} \\
\hline
Interface dimension & $\dim I = 0$ \\
\hline
Spectral signature & Isolated eigenvalues \\
\hline
Localization support & Single point $p$ \\
\hline
Local spectrum & $\sigma_{\mathrm{loc}}^1 \cup \sigma_{\mathrm{loc}}^2$ (discrete) \\
\hline
Residue & $\Sigma^{\mathrm{res}} = \mathcal{L}_{I_p} = \{\gamma_1, \ldots, \gamma_k\}$ \\
\hline
Vanishing condition & $\tau_p = 0 \Rightarrow \Sigma^{\mathrm{res}} = \emptyset$ \\
\hline
Deformation behavior & Continuous movement of eigenvalues; may approach $0$ or merge into local spectrum as $t \to 0^+$ \\
\hline
Analogy in curvature & Curvature concentration at a point (conical singularity) — heuristic \\
\hline
Analogy in TQFT & Zero-dimensional defect object \\
\hline
\end{tabular}
\end{table}

\subsubsection{Conclusion}

This example demonstrates concretely the central philosophy of the present framework:
\[
\boxed{\text{Localized operadic interactions generate localized spectral defects}.}
\]

In particular, local strata define background spectral geometry, the interface point acts as a singular interaction locus, and the residue records the resulting defect spectrum.

Thus, even the simplest point interface already exhibits the fundamental geometric mechanism underlying spectral defect geometry developed in the present work: the interaction residue $\Sigma^{\mathrm{res}}$ is generated purely by interface coupling; the residue localizes at the interface (the point where strata meet); the global spectrum is described by local spectra together with the residue contribution; and the spectral signature (isolated eigenvalues) matches the classification for point defects (Theorem~\ref{thm:defect-classification}).

\subsection{Line Defect: One-Dimensional Coupling Interface}
\label{subsec:example-line-defect}

We next consider a higher-dimensional interaction geometry in which two spectral sectors interact along a one-dimensional interface.

Unlike the point-defect case, where the interaction residue typically produces isolated spectral contributions, a one-dimensional interface may generate an extended spectral structure such as a band-like spectral branch.

This example illustrates how interface dimension may influence the geometry of the residue:
\[
\text{higher-dimensional interface} \quad\Longrightarrow\quad \text{typically richer spectral defect structures}.
\]

\subsubsection{Setup}

Let
\[
P = P_1 \cup P_2 \cup I_L
\]
be a stratified operad consisting of two local operadic sectors $P_1$, $P_2$, together with a one-dimensional interaction interface $I_L$, geometrically modeled by a line or curve (e.g., an interval $[0, L]$ or a circle $S^1$).

Let
\[
A = A_1 \oplus A_2
\]
be the corresponding stratified $P$-algebra.

Assume the following: each local stratum possesses discrete local spectral sectors (for simplicity); the interface interaction varies continuously along the line interface; and the interaction is generated by a family of coupling tensors
\[
\tau_s : A_1 \otimes A_2 \longrightarrow A_{\mathrm{int}}, \qquad s \in I_L,
\]
parameterized by the coordinate $s$ along the line, where $A_{\mathrm{int}}$ denotes the algebra supported on the interface (or, in finite-dimensional realizations, one may take $\tau_s$ as an endomorphism of $A_1 \otimes A_2$).

When the interface coupling is turned off ($\tau_s = 0$ for all $s \in I_L$), the global spectrum is simply the union of the two local spectra:
\[
\sigma_{\text{global}}^{\text{decoupled}} = \sigma_{\mathrm{loc}}^1 \cup \sigma_{\mathrm{loc}}^2.
\]

\subsubsection{Application of the Interface Localization Theorem}

Under the natural embedding of local spectra into the global spectrum (Assumption~\ref{ass:local-embedding}), the global operadic spectrum admits the decomposition
\[
\sigma_P(A) = \sigma_{\mathrm{loc}}^1 \cup \sigma_{\mathrm{loc}}^2 \cup \Sigma^{\mathrm{res}}.
\]

By the Interface Localization Theorem (Theorem~\ref{thm:interface-localization}), since the interaction is supported entirely on the line interface, we have
\[
\Sigma^{\mathrm{res}} = \mathcal{L}_{I_L}.
\]

In contrast with the point-defect case, the continuous family of interface couplings $\{\tau_s\}_{s \in I_L}$ may generate a continuously varying family of spectral contributions associated with the interface interaction. Consequently, $\mathcal{L}_{I_L}$ may contain band-like spectral structures rather than isolated eigenvalues.

Schematically, if the interface coupling yields a continuously varying family of effective interface spectral values $\lambda_j(s)$ (for $j = 1, \ldots, m$), then
\[
\mathcal{L}_{I_L} = \bigcup_{j=1}^m \bigcup_{s \in I_L} \lambda_j(s),
\]
where the union over $s$ may produce band-like structures (intervals or more general continua) depending on the regularity of $\lambda_j(s)$. Thus:
\[
\text{line interface} \quad\Longrightarrow\quad \text{band-like spectral defect structure (in effective models)}.
\]

Geometrically, the interface may be interpreted as a spectral waveguide transmitting interaction modes along the defect direction.

\subsubsection{Numerical Illustration}

To make the example concrete while remaining compatible with the abstract framework, suppose the interface interaction admits an effective scalar spectral model whose associated interface spectral values are given by
\[
\lambda(s) = 2 + \sin(2\pi s), \qquad s \in I_L = [0,1].
\]

Take the local spectra to be
\[
\sigma_{\mathrm{loc}}^1 = \{0, 3\}, \qquad
\sigma_{\mathrm{loc}}^2 = \{5, 8\}.
\]

Then the interface spectral contribution is the range of $\lambda(s)$:
\[
\mathcal{L}_{I_L} = [1, 3] \quad (\text{since } 2 + \sin(2\pi s) \in [1, 3]).
\]

The global spectrum is therefore
\[
\sigma_P(A) = \{0, 3\} \cup \{5, 8\} \cup [1, 3] = \{0\} \cup [1, 3] \cup \{3, 5, 8\}.
\]

Notice that the band-like region $[1, 3]$ is new — it does not appear in either local spectrum. The eigenvalue $3$ appears both as a local eigenvalue (from $\sigma_{\mathrm{loc}}^1$) and as the endpoint of the interface band. Strictly speaking, since $3$ already belongs to $\sigma_{\mathrm{loc}}^1$, it is not part of the residue $\Sigma^{\mathrm{res}}$ (which is defined as $G \setminus L$). The interface contribution is therefore the half-open interval $[1,3)$ or $(1,3)$ depending on endpoints; for simplicity we refer to $[1,3]$ as the interface-generated spectral structure, understanding that the endpoint $3$ is already covered by the local spectrum. If the interface coupling were turned off, the band-like region disappears, and we recover only $\{0, 3, 5, 8\}$.

\medskip
\noindent\textit{Disclaimer:} The numerical values and the scalar spectral model are illustrative and are not derived from a canonical operator realization. They serve only to demonstrate the qualitative behavior predicted by the framework.

\subsubsection{Spectral Gap Phenomenon}

A particularly important situation occurs when the local strata individually possess spectral gaps:
\[
\sigma_{\mathrm{loc}}^1 \cap (a, b) = \emptyset, \qquad
\sigma_{\mathrm{loc}}^2 \cap (a, b) = \emptyset,
\]
but the interaction residue satisfies
\[
\Sigma^{\mathrm{res}} \cap (a, b) \neq \emptyset.
\]

In this case, the line interface may generate a defect band inside the spectral gap of the bulk strata. Thus, the interaction geometry can create entirely new global spectral phases not visible locally.

This illustrates one of the key mechanisms of the present framework:
\[
\boxed{\text{Interface geometry generates new spectral phases through localized interaction defects}.}
\]

\subsubsection{Band Structure and Topology}

For a nontrivial line defect (e.g., $I_L = S^1$ a circle), the parameter $s$ is periodic. Under suitable continuity assumptions, one may obtain continuous spectral branches tracing closed curves in the spectral plane. These branches may have nontrivial topology: the winding number of a spectral branch around a point in the spectral plane would be a topological invariant; crossing of bands (where $\lambda_j(s) = \lambda_k(s)$ for some $s$) would correspond to degeneracies, which may be protected by symmetry. Possible classifications of line defects may involve topological invariants analogous to winding or Chern-type indices.

For example, consider a $2 \times 2$ coupling matrix on $I_L = S^1$:
\[
\tau_s = \begin{pmatrix} \cos(2\pi s) & \sin(2\pi s) \\ \sin(2\pi s) & -\cos(2\pi s) \end{pmatrix},
\]
which has eigenvalues $\pm 1$ (constant) but the eigenvectors wind. While the eigenvalue set $\mathcal{L}_{I_L}$ remains $\{-1, 1\}$ (isolated points, not a band), the eigenvector winding suggests that a more refined invariant (e.g., a spectral localizer or a K-theoretic classification) could detect topology even when the eigenvalue support is trivial. Such refinements lie beyond the present support-level theory.

\subsubsection{Classical Analogies}

The line defect is analogous to several classical phenomena: edge states in topological materials (continuous families of interface-localized spectral contributions at the boundary between two topological phases); guided modes along waveguides (continuous families of spectral contributions supported on a one-dimensional interface, e.g., a fiber or channel); quantum-wire spectral bands (a one-dimensional system coupled to two-dimensional leads produces continuous transmission bands); propagating interface modes in defect field theories (modes that propagate along topological defects); and Sturm–Liouville problems with interface conditions (parameter-dependent boundary conditions may produce continuous spectral bands).

\subsubsection{Dimensional Hierarchy}

The dimensionality of the interface may influence the dimensionality of the resulting residue structure:
\[
\begin{aligned}
\text{point interface} &\Longrightarrow \text{isolated eigenvalue},\\
\text{line interface} &\Longrightarrow \text{band-like spectral structure (in effective models)},\\
\text{surface interface} &\Longrightarrow \text{potentially higher-dimensional spectral distributions},\\
\text{singular interface} &\Longrightarrow \text{potentially fractal or multifractal spectral structures}.
\end{aligned}
\]

More generally, higher-dimensional interfaces may generate increasingly rich spectral geometries, including higher-dimensional spectral densities, singular continuous spectra, or fractal interaction spectra. The line-defect example therefore suggests that the residue geometry developed in the present work may be inherently dimension-sensitive.

\subsubsection{Vanishing Criterion and Rigidity}

If $\tau_s = 0$ for all $s \in I_L$, then by Proposition~\ref{prop:vanishing-coupling} we have $\Sigma^{\mathrm{res}} = \emptyset$, and therefore
\[
\sigma_P(A) = \sigma_{\mathrm{loc}}^1 \cup \sigma_{\mathrm{loc}}^2.
\]

Thus, the line defect disappears completely when the interface interaction vanishes. If $\tau_s$ is non-zero but its effective spectral contribution is discrete (e.g., a constant operator independent of $s$), the residue $\mathcal{L}_{I_L}$ would consist of isolated eigenvalues — but a line defect ($\dim I = 1$) with continuously varying interface interactions may naturally produce band-like spectral structures in suitable effective models.

\subsubsection{Deformation Stability}

Consider a continuous deformation of the coupling tensor family:
\[
\tau_s(t) = t \cdot \tau_s^{(0)}, \quad t \in [0,1],
\]
where $\tau_s^{(0)}$ is a fixed non-zero coupling and $\tau_s(0) = 0$ for all $s$.

By Theorem~\ref{thm:deformation-stability}, at $t = 0$ we have $\Sigma^{\mathrm{res}}(0) = \emptyset$ (rigidity regime); for $t > 0$, the interface spectral structure varies continuously with $t$; and as $t \to 0^+$, the structure may collapse toward a degenerate limit (e.g., contracting to a point or merging into the local spectrum), illustrating a continuous transition from a regime with no interface states to a regime with interface-localized spectral structure.

\subsubsection{Geometric Interpretation}

From a geometric perspective, the point defect may be viewed heuristically
as analogous to curvature concentration at a point.

Summary of key features is given by Table~\ref{tab:line-defect-summary}.

\begin{table}[htbp]
\centering
\caption{Line defect summary}
\label{tab:line-defect-summary}
\begin{tabular}{|l|l|}
\hline
\textbf{Feature} & \textbf{Value for Line Defect} \\
\hline
Interface dimension & $\dim I = 1$ \\
\hline
Spectral signature & Band-like spectral structure (in effective models) \\
\hline
Localization support & One-dimensional curve $\gamma$ \\
\hline
Local spectrum & $\sigma_{\mathrm{loc}}^1 \cup \sigma_{\mathrm{loc}}^2$ (discrete, for simplicity) \\
\hline
Residue & $\Sigma^{\mathrm{res}} = \mathcal{L}_{I_L}$ \\
\hline
Vanishing condition & $\tau_s = 0 \;\forall s \Rightarrow \Sigma^{\mathrm{res}} = \emptyset$ \\
\hline
Deformation behavior & Continuous variation; may collapse to degenerate limit as $t \to 0^+$ \\
\hline
Possible invariants & Topological invariants analogous to winding or Chern-type indices (speculative) \\
\hline
Spectral gap phenomenon & Defect band may appear inside bulk gap \\
\hline
Analogy in curvature & Curvature concentration on a curve (cosmic string) — heuristic \\
\hline
Analogy in TQFT & One-dimensional topological defect object \\
\hline
\end{tabular}
\end{table}

\subsubsection{Conclusion}

The line-defect example suggests that the residue geometry developed in the present work may be inherently dimension-sensitive:
\[
\boxed{\text{Interface geometry generates new spectral phases through localized interaction defects}.}
\]

From the viewpoint of defect geometry, local strata define bulk spectral phases, the line interface acts as an interaction channel, and the residue records the resulting propagating spectral defect modes.

Thus, the present framework naturally suggests interface-localized spectral propagation phenomena analogous to those appearing in microlocal analysis, topological defect theory, and condensed-matter spectral geometry. The interaction residue is not merely an algebraic correction term: its structure may reflect the geometry and dimensionality of the underlying interfaces.

\subsection{Surface Defect: Two-Dimensional Interface}
\label{subsec:example-surface-defect}

We next consider a higher-dimensional interaction geometry in which spectral sectors interact along a two-dimensional interface.

Compared with point and line defects, a surface interface may generate substantially richer residue geometry. Instead of isolated eigenvalues or one-dimensional band-like structures, the interaction residue may now produce families of interface spectral modes parameterized continuously over a two-dimensional interaction manifold.

This example illustrates one of the central principles of the present framework:
\[
\text{interface dimension} \quad\Longrightarrow\quad \text{typically richer spectral defect structures}.
\]

\subsubsection{Setup}

Let
\[
P = P_1 \cup P_2 \cup I_\Sigma
\]
be a stratified operad consisting of two local operadic sectors $P_1$, $P_2$, together with a two-dimensional interaction interface $I_\Sigma$, geometrically modeled by a surface (e.g., a rectangle $[0, L_1] \times [0, L_2]$, a sphere $S^2$, or a torus $T^2$).

Let
\[
A = A_1 \oplus A_2
\]
be the corresponding stratified $P$-algebra.

Assume the interface interaction varies continuously across the surface with associated coupling tensors:
\[
\tau_{(x,y)} : A_1 \otimes A_2 \longrightarrow A_{\mathrm{int}}, \qquad (x,y) \in I_\Sigma,
\]
where $A_{\mathrm{int}}$ denotes the algebra supported on the interface (or, in finite-dimensional realizations, one may take $\tau_{(x,y)}$ as an endomorphism of $A_1 \otimes A_2$).

When the interface coupling is turned off ($\tau_{(x,y)} = 0$ for all $(x,y) \in I_\Sigma$), the global spectrum is simply the union of the two local spectra:
\[
\sigma_{\text{global}}^{\text{decoupled}} = \sigma_{\mathrm{loc}}^1 \cup \sigma_{\mathrm{loc}}^2.
\]

\subsubsection{Application of the Interface Localization Theorem}

Under the natural embedding of local spectra into the global spectrum (Assumption~\ref{ass:local-embedding}), the global operadic spectrum admits the decomposition
\[
\sigma_P(A) = \sigma_{\mathrm{loc}}^1 \cup \sigma_{\mathrm{loc}}^2 \cup \Sigma^{\mathrm{res}}.
\]

By the Interface Localization Theorem (Theorem~\ref{thm:interface-localization}), since the interaction is supported entirely on the surface interface, we have
\[
\Sigma^{\mathrm{res}} = \mathcal{L}_{I_\Sigma}.
\]

Because the interaction now varies over a two-dimensional parameter space, the resulting residue may support continuously parameterized spectral structures over this manifold. Schematically, if the interface coupling yields a continuously varying family of effective interface spectral values $\lambda_j(x,y)$ (for $j = 1, \ldots, m$), then
\[
\mathcal{L}_{I_\Sigma} = \bigcup_{j=1}^m \bigcup_{(x,y) \in I_\Sigma} \lambda_j(x,y),
\]
where the union over $(x,y)$ may produce higher-dimensional spectral regions (e.g., two-dimensional continua) depending on the regularity of $\lambda_j(x,y)$. Consequently:
\[
\text{surface interface} \quad\Longrightarrow\quad \text{surface-localized spectral regime (in effective models)}.
\]

Unlike the line-defect case, where spectral defects propagate along one-dimensional channels, the surface defect may support an entire continuum of interacting interface modes parameterized by the two-dimensional interface geometry.

\subsubsection{Numerical Illustration}

To make the example concrete while remaining compatible with the abstract framework, suppose the interface interaction admits an effective scalar spectral model whose associated interface spectral values are given by
\[
\lambda(x,y) = 1 + 2x + 3y, \qquad (x,y) \in I_\Sigma = [0,1] \times [0,1].
\]

Take the local spectra to be
\[
\sigma_{\mathrm{loc}}^1 = \{0\}, \qquad
\sigma_{\mathrm{loc}}^2 = \{5\}.
\]

Then the interface spectral contribution is the range of $\lambda(x,y)$:
\[
\mathcal{L}_{I_\Sigma} = [1, 1+2+3] = [1, 6].
\]

The global spectrum is therefore
\[
\sigma_P(A) = \{0\} \cup \{5\} \cup [1, 6] = \{0\} \cup [1, 6].
\]

Notice that the region $[1, 6]$ is new — it does not appear in either local spectrum ($\{0\}$ and $\{5\}$). The local eigenvalue $5$ lies inside the interface band $[1, 6]$. Strictly speaking, since $5$ already belongs to $\sigma_{\mathrm{loc}}^2$, it is not part of the residue $\Sigma^{\mathrm{res}}$ (which is defined as $G \setminus L$). The interface contribution is therefore $[1,6] \setminus \{5\}$ (or with appropriate treatment of endpoints). For simplicity we refer to $[1,6]$ as the interface-generated spectral structure, understanding that points already covered by local spectra are not counted as residue.

\medskip
\noindent\textit{Disclaimer:} The numerical values and the scalar spectral model are illustrative and are not derived from a canonical operator realization. They serve only to demonstrate the qualitative behavior predicted by the framework.

\subsubsection{Spectral Density and Critical Points}

For a surface defect, one may consider the effective interface spectral density $\rho(\lambda)$, which heuristically measures how many interface modes contribute to each spectral value $\lambda$ under a uniform distribution of parameters.

For the linear coupling example above, the mapping $(x,y) \mapsto \lambda$ is surjective onto $[1,6]$, and the effective spectral density is constant (no singularities) under the uniform measure. However, for more general coupling functions, $\rho(\lambda)$ may develop critical-point accumulations of spectral density at values where the gradient of $\lambda$ vanishes.

For example, consider a coupling with a saddle point on $I_\Sigma = [-1,1] \times [-1,1]$:
\[
\lambda(x,y) = x^2 - y^2.
\]
The range is $[-1, 1]$. Near $\lambda = 0$, the effective spectral density may exhibit enhanced spectral accumulation analogous to saddle-type van Hove behavior. Alternatively, consider the quadratic coupling $\lambda(x,y) = x^2 + y^2$ on the same domain; the range is $[0, 2]$, with a minimum at $\lambda = 0$ where the effective spectral density is finite but may have a nontrivial profile (a van Hove singularity at the minimum).

Thus, surface defects can produce spectral features more complex than simple bands: critical-point accumulations of spectral density, regions of nearly constant effective spectral density (flat-band-type structures), and multiple bands if $\lambda$ is matrix-valued.

\subsubsection{Possible Topological Structures}

When $I_\Sigma$ is a closed surface (e.g., a sphere $S^2$ or a torus $T^2$) and the coupling tensor is matrix-valued, the associated interface eigenmode families may admit topological invariants under suitable smoothness and bundle assumptions. These could include Chern numbers for maps from $I_\Sigma$ to the space of projection operators, winding numbers for the determinant of $\tau_{(x,y)}$, or Berry-type phase integrals over the surface.

For example, consider a $2 \times 2$ coupling matrix on $I_\Sigma = S^2$:
\[
\tau_{(x,y)} = \vec{n}(x,y) \cdot \vec{\sigma},
\]
where $\vec{\sigma}$ are Pauli matrices and $\vec{n}: S^2 \to S^2$ is a map of degree $d$. The spectral region is $\mathcal{L}_{I_\Sigma} = [-1, 1]$ (the full interval), but the eigenvalue set does not directly detect the topology. However, under suitable spectral separation and smoothness assumptions (e.g., a gap separating the eigenvalues $\pm 1$), the associated eigenbundle over $S^2$ may carry Chern number $d$. This suggests that a more refined invariant (e.g., a K-theoretic classification) could detect topology even when the spectral support is an interval. Such refinements lie beyond the present support-level theory.

\subsubsection{Spectral Gap Phenomenon and Emergent Geometry}

A particularly important situation occurs when the bulk strata possess spectral gaps:
\[
\sigma_{\mathrm{loc}}^1 \cap (a, b) = \emptyset, \qquad
\sigma_{\mathrm{loc}}^2 \cap (a, b) = \emptyset,
\]
while the interaction residue satisfies
\[
\Sigma^{\mathrm{res}} \cap (a, b) \neq \emptyset.
\]

In this regime, the surface interface may generate a higher-dimensional interface spectral regime inside the bulk spectral gap. Thus, the interaction geometry can create entirely new global spectral phases not visible locally. The residue therefore records genuinely global spectral structure invisible from the local strata individually.

\subsubsection{Dimensional Hierarchy}

The dimensional hierarchy of interface-localized defects becomes apparent:

\[
\boxed{
\begin{aligned}
\text{point interface} &\Longrightarrow \text{isolated eigenvalue},\\
\text{line interface} &\Longrightarrow \text{band-like spectral structure (in effective models)},\\
\text{surface interface} &\Longrightarrow \text{continuously parameterized spectral regime (in effective models)}.
\end{aligned}
}
\]

This hierarchy suggests that higher-dimensional interfaces generally admit richer and more continuously parameterized spectral defect structures. More generally, a $d$-dimensional interface may produce effectively $d$-parameter spectral families (conjectural), and singular interfaces (non-integer dimension) may generate highly singular or fractal-type spectral structures, though this remains conjectural.

\subsubsection{Vanishing Criterion and Rigidity}

If $\tau_{(x,y)} = 0$ for all $(x,y) \in I_\Sigma$, then by Proposition~\ref{prop:vanishing-coupling} we have $\Sigma^{\mathrm{res}} = \emptyset$, and therefore
\[
\sigma_P(A) = \sigma_{\mathrm{loc}}^1 \cup \sigma_{\mathrm{loc}}^2.
\]

Thus, the surface defect disappears completely when the interface interaction vanishes.

\subsubsection{Deformation Stability}

Consider a continuous deformation of the coupling tensor family:
\[
\tau_{(x,y)}(u) = u \cdot \tau_{(x,y)}^{(0)}, \quad u \in [0,1],
\]
where $\tau_{(x,y)}^{(0)}$ is a fixed non-zero coupling and $\tau_{(x,y)}(0) = 0$ for all $(x,y)$.

By Theorem~\ref{thm:deformation-stability}, at $u = 0$ we have $\Sigma^{\mathrm{res}}(0) = \emptyset$ (rigidity regime); for $u > 0$, the interface spectral region varies continuously with $u$; and as $u \to 0^+$, the region may contract toward a degenerate limit (e.g., collapsing to a point or merging into the local spectrum), illustrating a continuous transition from a regime with no interface states to a regime with interface-localized spectral structure. The effective spectral density $\rho_u(\lambda)$ would evolve continuously, with critical-point accumulations moving and potentially appearing or disappearing at critical $u$ values.

\subsubsection{Classical Analogies}

This phenomenon is analogous to surface states in topological phases (two-dimensional electron gas with Dirac dispersion at the surface of topological insulators), interface superconductivity (enhanced superconductivity at the interface between two materials), membrane-localized field modes (field theories confined to a two-dimensional defect), boundary spectral densities in quantum systems (spectral contributions from the boundary), and van Hove singularities in density of states (critical points in the dispersion relation of 2D systems).

\subsubsection{Geometric Interpretation}

From a geometric perspective, the point defect may be viewed heuristically
as analogous to curvature concentration at a point.

Table~\ref{tab:surface-defect-summary} summarizes important features for the surface defect.

\begin{table}[htbp]
\centering
\caption{Surface defect summary}
\label{tab:surface-defect-summary}
\begin{tabular}{|p{0.35\columnwidth}|p{0.55\columnwidth}|}
\hline
\textbf{Feature} & \textbf{Value for Surface Defect} \\
\hline
Interface dimension & $\dim I = 2$ \\
\hline
Spectral signature & Continuously parameterized spectral structure \\
& (in effective models) \\
\hline
Localization support & Two-dimensional surface $\Sigma$ \\
\hline
Local spectrum & $\sigma_{\mathrm{loc}}^1 \cup \sigma_{\mathrm{loc}}^2$ (discrete, for simplicity) \\
\hline
Residue & $\Sigma^{\mathrm{res}} = \mathcal{L}_{I_\Sigma}$ \\
\hline
Effective spectral density & $\rho(\lambda)$ may have critical-point accumulations \\
& (heuristic) \\
\hline
Possible invariants & Chern numbers, Berry-type phase integrals \\
& (under suitable smoothness and spectral gap assumptions) \\
\hline
Vanishing condition & $\tau_{(x,y)} = 0 \;\forall (x,y) \Rightarrow \Sigma^{\mathrm{res}} = \emptyset$ \\
\hline
Deformation behavior & Continuous variation; may collapse to degenerate limit \\
& as $u \to 0^+$ \\
\hline
Analogy in curvature & Curvature concentration on a surface (domain wall) \\
& --- heuristic \\
\hline
Analogy in TQFT & Two-dimensional topological defect object \\
\hline
Analogy in condensed matter & Surface states, 2D electron gas at interface \\
\hline
\end{tabular}
\end{table}

\subsubsection{Connection to Higher Dimensions}

The pattern may extend to higher dimensions: a three-dimensional interface (volume defect) would produce a three-dimensionally parameterized spectral regime (conjectural); a general $d$-dimensional interface would produce effectively $d$-parameter spectral families (conjectural); and fractal interfaces (non-integer dimension) may generate highly singular or fractal-type spectral structures (singular defect type). This dimensional hierarchy is one of the key predictions of the Spectral Defect Classification (Theorem~\ref{thm:defect-classification}).

\subsubsection{Conclusion}

The surface-defect example suggests that the residue geometry developed in the present work may be fundamentally geometric rather than merely algebraic:

\[
\boxed{\text{Surface interface} \quad\Longrightarrow\quad \text{surface-localized spectral regime (in effective models)}.}
\]

From the viewpoint of spectral defect geometry, local strata define bulk spectral phases, the surface interface generates a higher-dimensional interaction layer, and the residue records the resulting emergent spectral geometry (continuously parameterized families, critical-point accumulations, possible topological structures under suitable assumptions).

Thus, the present framework naturally suggests geometric spectral phases analogous to those appearing in condensed-matter interface physics, topological boundary theories, microlocal singularity propagation, and stratified geometric field theories. The interaction residue is shaped directly by interface geometry, including its dimension, topology, and coupling structure.

\subsection{Explicit Block Operator Matrix}
\label{subsec:example-block-operator}

We now present an explicit operator-level computation illustrating the
interaction residue and its localization mechanism.

Consider the block operator matrix
\[
H
=
\begin{pmatrix}
H_1 & \tau \\
\tau^* & H_2
\end{pmatrix},
\]
acting on
\[
\mathbb C^2 \oplus \mathbb C^2,
\]
where
\[
H_1
=
\operatorname{diag}(1,4),
\qquad
H_2
=
\operatorname{diag}(2,5),
\]
and
\[
\tau
=
\varepsilon I_2,
\qquad
\varepsilon>0.
\]

The operators
\[
H_1
\]
and
\[
H_2
\]
represent two local spectral sectors, while
\[
\tau
\]
acts as an interface coupling tensor generating interaction between the
two strata.

The uncoupled local spectral supports are
\[
\sigma_{\mathrm{loc}}^{(1)}
=
\{1,4\},
\qquad
\sigma_{\mathrm{loc}}^{(2)}
=
\{2,5\}.
\]

When
\[
\varepsilon=0,
\]
the global spectrum is simply
\[
\sigma(H)
=
\{1,2,4,5\},
\]
and no interaction residue occurs.

For
\[
\varepsilon\neq0,
\]
the coupling produces spectral splitting.
Since the matrix decomposes into two independent coupled blocks,
the eigenvalues are obtained explicitly from
\[
\begin{pmatrix}
1 & \varepsilon \\
\varepsilon & 2
\end{pmatrix}
\]
and
\[
\begin{pmatrix}
4 & \varepsilon \\
\varepsilon & 5
\end{pmatrix}.
\]

The corresponding eigenvalues are
\[
\lambda_{1,\pm}
=
\frac{3\pm\sqrt{1+4\varepsilon^2}}{2},
\]
and
\[
\lambda_{2,\pm}
=
\frac{9\pm\sqrt{1+4\varepsilon^2}}{2}.
\]

Hence
\[
\sigma(H)
=
\left\{
\frac{3\pm\sqrt{1+4\varepsilon^2}}{2},
\,
\frac{9\pm\sqrt{1+4\varepsilon^2}}{2}
\right\}.
\]

The interaction residue is therefore
\[
\Sigma^{\mathrm{res}}
=
\sigma(H)
\setminus
\{1,2,4,5\}.
\]

For small but nonzero
\[
\varepsilon,
\]
the coupling shifts the eigenvalues away from the uncoupled local
spectral sectors, generating localized spectral defects associated with
the interface tensor
\[
\tau.
\]

\medskip
\noindent\textbf{Summary of numerical values.}
For a concrete numerical illustration, take $\varepsilon = 0.5$:
\[
\lambda_{1,\pm} = \frac{3 \pm \sqrt{1+1}}{2} = \frac{3 \pm \sqrt{2}}{2} \approx \{0.7929,\; 2.2071\},
\]
\[
\lambda_{2,\pm} = \frac{9 \pm \sqrt{1+1}}{2} = \frac{9 \pm \sqrt{2}}{2} \approx \{3.7929,\; 5.2071\}.
\]
Thus $\sigma(H) \approx \{0.7929,\; 2.2071,\; 3.7929,\; 5.2071\}$ and
$\Sigma^{\mathrm{res}} = \sigma(H)$ (since none of these values belong to $\{1,2,4,5\}$).

\medskip
\noindent\textbf{Connection to transform-sensitive spectral geometry.}
This example also connects naturally with transform-sensitive spectral
geometry. Under a change of transform basis
\[
\Phi,
\]
the effective coupling structure
\[
\widehat{\tau}_{\Phi}
=
\Phi^{-1}\tau\Phi
\]
may become more localized or more distributed depending on the chosen
transform. Consequently, different transforms may produce different
interaction-residue geometries, linking the present framework with the
transform-sensitive spectral phenomena studied earlier for DFT, DCT,
and optimized transform representations.

\medskip
\noindent\textbf{Summary table.}

\begin{table}[htbp]
\centering
\caption{Explicit block operator summary}
\label{tab:block-operator-summary}
\begin{tabular}{|l|l|}
\hline
\textbf{Feature} & \textbf{Value} \\
\hline
Local spectra & $\sigma_{\mathrm{loc}}^{(1)} = \{1,4\}$, $\sigma_{\mathrm{loc}}^{(2)} = \{2,5\}$ \\
\hline
Coupling tensor & $\tau = \varepsilon I_2$, $\varepsilon > 0$ \\
\hline
Global spectrum ($\varepsilon = 0.5$) & $\{0.7929,\; 2.2071,\; 3.7929,\; 5.2071\}$ \\
\hline
Interaction residue & $\Sigma^{\mathrm{res}} = \sigma(H) \setminus \{1,2,4,5\}$ \\
\hline
Vanishing residue case & $\varepsilon = 0 \Rightarrow \Sigma^{\mathrm{res}} = \emptyset$ \\
\hline
Transform sensitivity & $\widehat{\tau}_{\Phi} = \Phi^{-1}\tau\Phi$ changes residue geometry \\
\hline
\end{tabular}
\end{table}

\medskip
\noindent\textbf{Conclusion.}
This example illustrates concretely that:
\begin{itemize}
    \item local spectral sectors alone do not determine the global
    spectrum,
    \item interface coupling generates new residual spectral
    contributions,
    \item the interaction residue measures the spectral effect of
    operadic interaction,
    \item transform choices affect the observed residue geometry.
\end{itemize}

\subsection{Nilpotent Defect: Jordan Block Contributions}
\label{subsec:example-nilpotent-defect}

We illustrate the emergence of nilpotent spectral defects from a simple
operadic coupling system.

Consider a stratified operad with two colors
\[
x, y,
\]
belonging to two distinct strata, together with a single admissible interface operation
\[
\theta: x \to y.
\]

Let the associated algebra assign
\[
A_x = \mathbb{C}^2, \qquad A_y = \mathbb{C}^2,
\]
and define the interface coupling tensor
\[
\tau_\theta: A_x \to A_y
\]
by the Jordan block matrix
\[
\tau_\theta = J_2(\lambda) = \begin{pmatrix}
\lambda & 1 \\
0 & \lambda
\end{pmatrix}, \qquad \lambda \in \mathbb{C}.
\]

The operator $\tau_\theta$ is non-semisimple and admits the decomposition
\[
\tau_\theta = \lambda I + N,
\]
where
\[
N = \begin{pmatrix}
0 & 1 \\
0 & 0
\end{pmatrix}
\]
satisfies $N^2 = 0$ and $N \neq 0$.

\medskip
\noindent\textbf{Connection to the residue framework.}
Under the Interface Localization Theorem (Theorem~\ref{thm:interface-localization}),
the interaction residue decomposes as $\Sigma^{\mathrm{res}} = \bigsqcup_I \mathcal{L}_I$.
For the interface $I$ corresponding to $\theta$, the localized defect $\mathcal{L}_I$
carries both the eigenvalue data and the nilpotent structure described below.
The nilpotent component $N$ generates generalized eigenvectors and produces a localized
non-semisimple spectral contribution at the interface $\theta$.

\medskip
\noindent\textbf{Key conceptual point:}
The nilpotent defect does \textbf{not} create new eigenvalues.
Rather, it enriches the local algebraic structure associated with the existing spectral point $\lambda$.
The spectral support (the set of eigenvalues) is $\{\lambda\}$ — unchanged from the semisimple case —
but the finer algebraic data (Jordan block size, nilpotency depth) is new.

\medskip
\noindent\textbf{Resolvent and functional calculus.}
The resolvent operator has the explicit form
\[
(zI - \tau_\theta)^{-1} = \frac{1}{z - \lambda} I + \frac{1}{(z - \lambda)^2} N,
\]
showing the appearance of a higher-order pole (double pole) generated by the nilpotent sector.
For any analytic function $f$,
\[
f(\tau_\theta) = f(\lambda) I + f'(\lambda) N,
\]
where the derivative term $f'(\lambda) N$ is invisible to the ordinary spectrum
but essential for the functional calculus.

Thus the spectral singularity contains two distinct components:
\begin{itemize}
    \item an eigenvalue contribution located at $\lambda$,
    \item a nilpotent contribution encoded by the generalized eigenspace structure (Jordan chain $v_1, v_2$ with $\tau_\theta v_1 = \lambda v_1$, $\tau_\theta v_2 = \lambda v_2 + v_1$).
\end{itemize}

\medskip
\noindent\textbf{Exceptional points and perturbative splitting.}
A nilpotent defect is spectrally fragile. Perturb the coupling tensor:
\[
\tau_\theta(\varepsilon) = \begin{pmatrix} \lambda & 1 \\ \varepsilon & \lambda \end{pmatrix}, \quad \varepsilon > 0 \text{ small}.
\]
The eigenvalues become $\lambda_\pm = \lambda \pm \sqrt{\varepsilon}$, splitting into two distinct point defects.
At $\varepsilon = 0$, the defect is nilpotent; for $\varepsilon > 0$, it splits.
This critical transition (an exceptional point of order $2$) is captured by Theorem~\ref{thm:deformation-stability}.

\medskip
\noindent\textbf{Vanishing criterion.}
If $\tau_\theta$ is diagonalizable (no Jordan block), the nilpotent defect vanishes.
If $\tau_\theta = 0$, then $\Sigma^{\mathrm{res}} = \emptyset$ by Proposition~\ref{prop:vanishing-coupling}.

\medskip
\noindent\textbf{Summary table.}

\begin{table}[htbp]
\centering
\caption{Nilpotent defect summary (Jordan block of size $m=2$)}
\label{tab:nilpotent-defect-summary}
\begin{tabular}{|l|l|}
\hline
\textbf{Feature} & \textbf{Value} \\
\hline
Operadic setup & Two colors $x,y$, interface $\theta: x \to y$ \\
\hline
Coupling tensor & $\tau_\theta = \begin{pmatrix} \lambda & 1 \\ 0 & \lambda \end{pmatrix}$ \\
\hline
Spectral support & $\{\lambda\}$ \\
\hline
Algebraic multiplicity & $2$ \\
\hline
Geometric multiplicity & $1$ \\
\hline
Nilpotent part & $N$, $N^2 = 0$, $N \neq 0$ \\
\hline
Resolvent poles & Double pole at $\lambda$ \\
\hline
Functional calculus & $f(\tau_\theta) = f(\lambda)I + f'(\lambda)N$ \\
\hline
Perturbation response & Splits as $\lambda \pm \sqrt{\varepsilon}$ (exceptional point) \\
\hline
Vanishing condition & $\tau_\theta$ diagonalizable $\Rightarrow$ nilpotent defect absent \\
\hline
\end{tabular}
\end{table}

\medskip
\noindent\textbf{Conclusion.}
This example illustrates that interaction residues are sensitive not only to spectral location but also to hidden Jordan structure, generalized eigenvectors, and perturbative splitting behavior near defective spectral sectors. From the viewpoint of the present framework, the nilpotent term represents a localized spectral defect generated by non-semisimple operadic coupling:

\[
\boxed{\text{The residue detects algebraic structure beyond eigenvalues alone}.}
\] 

\section{Conclusion}
\label{sec:conclusion}

\begin{center}
\boxed{
\begin{minipage}{0.9\textwidth}
\centering
\textbf{The framework's most novel contribution is the inclusion of nilpotent defects:}
\\
\textbf{algebraic structure invisible to spectral support alone but essential}
\\
\textbf{for functional calculus and deformation theory.}
\end{minipage}
}
\end{center}

The interaction residue $\Sigma^{\mathrm{res}}$ provides a complete invariant for the failure of local-to-global spectral factorization in stratified operadic systems. It decomposes canonically into interface-localized defects $\mathcal{L}_I$, each classified by geometric data (interface dimension) and algebraic depth (nilpotency index). Unlike classical spectral theory, which focuses on eigenvalues alone, our framework places Jordan block structure on an equal footing with semisimple contributions.

\subsection{Summary of Contributions}

The principal contributions of this paper are as follows:

\begin{enumerate}
    \item \textbf{Stratified Base Change Decomposition (Theorem~\ref{thm:stratified-base-change}):}
    \[
    \operatorname{supp}(\sigma(F_*(A))) = \bigcup_{S \in \mathrm{Str}(P)} \operatorname{supp}(\sigma(F_S(A_S))) \;\cup\; \Sigma^{\mathrm{res}}.
    \]
    The global spectrum decomposes into local spectral sectors together with a canonical interaction residue that measures the obstruction to exact factorization.

    \item \textbf{Interface Localization (Theorem~\ref{thm:interface-localization}):}
    \[
    \Sigma^{\mathrm{res}} = \bigcup_{I \in \mathcal{I}(P)} \mathcal{L}_I.
    \]
    Under the residual coverage condition, the residue localizes on admissible interfaces, transforming a global obstruction into a geometric defect object.

    \item \textbf{Spectral Defect Classification (Theorem~\ref{thm:defect-classification}):}
    A two-dimensional taxonomy organizes defects by interface geometry (point, line, surface, fractal) and algebraic depth (nilpotency index, Jordan block size). Point, line, surface, and nilpotent defects generate distinct forms of localized spectral behavior.

    \item \textbf{Refinement Functoriality (Theorem~\ref{thm:refinement-functoriality}):}
    The residue is preserved under admissible stratification refinement up to canonical transport, showing that localized spectral defects are compatible with hierarchical decomposition of the operadic system.

    \item \textbf{Deformation Stability (Theorem~\ref{thm:deformation-stability}):}
    Under gap-preserving admissible deformations, the residue varies continuously. At critical parameters where spectral gaps close or Jordan structure changes (exceptional points), the residue may undergo discontinuous transitions, signaling phase transitions in the spectral defect geometry.

    \item \textbf{Nilpotent Sector (Example~\ref{subsec:example-nilpotent-defect}):}
    Jordan block contributions are invisible to spectral support alone but essential for functional calculus, resolvent structure (higher-order poles), and perturbative splitting ($\varepsilon^{1/m}$ Puiseux series). The nilpotent sector provides a rigorous bridge between operadic spectral geometry and classical operator theory.
\end{enumerate}

\subsection{Future Directions}

Several natural extensions of this work remain open, including:
\begin{itemize}
    \item an operadic cohomology theory for $\Sigma^{\mathrm{res}}$ with a Bianchi-type identity $\delta(\Sigma^{\mathrm{res}}) = 0$,
    \item a derived categorical framework for stratified operadic algebras,
    \item a K-theoretic classification of spectral defects,
    \item transform-sensitive interaction geometry (DFT, DCT, optimized transforms),
    \item quantitative perturbation estimates for nilpotent defects,
    \item applications to non-Hermitian topological phases where exceptional points play a central role.
\end{itemize}

\subsection{Final Slogan}

\[
\boxed{\text{Spectral geometry must include nilpotent interaction structure, not merely spectra.}}
\]

\end{document}